\RequirePackage{xcolor}
\documentclass[journal,twoside,web]{ieeecolor}
\usepackage[all=normal, mathspacing=tight]{savetrees}
\usepackage{generic}
\usepackage{cite}
\usepackage{amsmath,amssymb,amsfonts}
\usepackage{mathtools}
\usepackage{commath}
\usepackage{graphicx}
\usepackage{silence}
\usepackage[caption=false,font=footnotesize]{subfig}
\usepackage{algorithm}
\usepackage{algorithmic}
\usepackage[hidelinks]{hyperref}
\usepackage{textcomp}
\def\BibTeX{{\rm B\kern-.05em{\sc i\kern-.025em b}\kern-.08em
    T\kern-.1667em\lower.7ex\hbox{E}\kern-.125emX}}
\usepackage{nicefrac}       
\usepackage{bm}
\usepackage{bbm}

\usepackage{amsthm}

\newtheorem{definition}{Definition}[section]
\newtheorem{assumption}{Assumption}[section]
\newtheorem{lemma}{Lemma}[section]
\newtheorem{theorem}{Theorem}[section]
\newtheorem{proposition}{Proposition}[section]
\newtheorem{remark}{Remark}

\usepackage{tikz}
\usetikzlibrary{arrows.meta, positioning}

\begin{document}
\title{Convergence Analysis of Noisy Distributed Gradient Descent for Non-convex Optimization --- Saddle Point Escape}
\author{Lei Qin, Michael Cantoni, and Ye Pu
\thanks{This work was supported by a Melbourne Research Scholarship and the Australian Research Council (DE220101527 and DP210103272).}
\thanks{L. Qin, M. Cantoni, and Y. Pu are with the Department of Electrical and Electronic Engineering, University of Melbourne, Parkville VIC 3010, Australia \texttt{\small leqin@student.unimelb.edu.au, \{cantoni, ye.pu\}@unimelb.edu.au}.}}
\maketitle

\begin{abstract}
This paper studies noisy distributed gradient descent (\textbf{NDGD}) for smooth non-convex finite-sum optimization over networks. Random perturbations enable saddle-point escape while preserving distributed implementation and consensus. Under suitable regularity conditions, \textbf{NDGD} converges with high probability to a neighborhood of a common local minimizer. Its convergence complexity is comparable to centralized first-order saddle-point escape methods, reducing exponential dependence on problem dimension to polynomial dependence. Numerical experiments demonstrate improved saddle-point escape over standard \textbf{DGD}.
\end{abstract}


\section{Introduction}
\label{sec: Introduction}
Collaborative distributed optimization across networks has attracted significant attention due to applications in multi-agent control \cite{nedic2018distributed,molzahn2017survey,camponogara2010distributed}, sensor networks \cite{rabbat2004distributed,johansson2007simple,wan2009event}, and machine learning \cite{boyd2011distributed,konevcny2015federated,mcmahan2017communication}. Such frameworks support coordinated decision-making in distributed model predictive control, multi-agent learning, and networked resource allocation.

We consider objective functions of the form
\begin{gather}
\label{eq: Unconstrained optimization problem}
f(\mathbf{x})=\sum_{i=1}^{m}f_i(\mathbf{x}),
\end{gather}
where each $f_i:\mathbb{R}^{n}\rightarrow\mathbb{R}$ is smooth but potentially non-convex and $\mathbf{x}=[x_1,\ldots,x_n]^{\top}\in\mathbb{R}^{n}$ is the common decision vector. The goal is to design a distributed algorithm that enables $m$ agents to collaboratively compute a second-order stationary point of~\eqref{eq: Unconstrained optimization problem}. Agents communicate over an undirected graph $\mathcal{G}(\mathcal{V},\mathcal{E})$, where $\mathcal{V}:=\{1,\dots,m\}$ and $\mathcal{E}\subseteq\mathcal{V}\times\mathcal{V}$; each agent $i\in\mathcal{V}$ has access only to its local objective $f_i$ and exchanges information with its neighbors.

We study a perturbed variant of consensus-based distributed gradient descent (\textbf{DGD}) \cite{nedic2009distributed}; our preliminary work \cite{qin2023second} considered this idea under more restrictive assumptions. To solve~\eqref{eq: Unconstrained optimization problem} in a distributed manner, each agent $i$ maintains a local copy $\hat{\mathbf{x}}_{i}^{k}\in\mathbb{R}^{n}$ of the common decision vector $\mathbf{x}$; the subscript $i$ indexes the agent, not a component of $\mathbf{x}$. The local copy maintained by agent $i$ evolves according to
\begin{align}
\label{eq: NDGD}
\hat{\mathbf{x}}_{i}^{k+1}=\sum_{j=1}^{m}w_{ij}\hat{\mathbf{x}}_{j}^{k}-\alpha\left(\nabla f_i(\hat{\mathbf{x}}_{i}^{k})+\mathbf{n}_{i}^{k}\right),
\end{align}
where $\alpha>0$ is a fixed step-size, $\mathbf{W}\in\mathbb{R}^{m\times m}$ is the mixing matrix, and $\mathbf{n}_{i}^{k}\in\mathbb{R}^{n}$ is a stochastic perturbation generated locally by agent $i$. We refer to~\eqref{eq: NDGD} as noisy distributed gradient descent (\textbf{NDGD}). A constant step-size enables explicit characterization of the trade-off between convergence accuracy and robustness to noise. In non-convex optimization, standard first-order methods may approach saddle points, and escaping saddle-point regions can require exponential time in the worst case. Although second-order methods can exploit negative curvature directly, their Hessian-related computational and communication requirements make them less attractive in distributed settings. Random perturbations provide an efficient first-order alternative in centralized optimization, but their use in distributed optimization introduces an additional challenge: locally injected perturbations interact with the consensus dynamics and can generate disagreement across the network. The key technical challenge is therefore to establish efficient saddle-point escape while simultaneously controlling optimization accuracy, curvature, and consensus error, with guarantees stated for the original objective $f$.

\subsection{Main contributions}

The main technical challenge is to establish efficient second-order guarantees in the presence of consensus dynamics and network coupling, while controlling both optimization and agreement errors across the agents.

\begin{itemize}
    \item We establish finite-time, high-probability convergence of \textbf{NDGD} to a consensual approximate second-order stationary point, with explicit optimality, curvature, and consensus-error bounds (Theorem~\ref{the: Approximate second order}).
    \item Under the additional local regularity condition, we establish a finite-time agent-wise $O(\sqrt{\alpha})$ proximity guarantee to the local minimizer set of the original objective $f$, together with an explicit consensus-error bound (Theorem~\ref{the: Second order guarantee}).
    \item We explicitly characterize the effects of step-size, noise variance, and network structure on the convergence guarantees, and show that the resulting saddle-point escape complexity avoids exponential dependence on the effective distributed problem dimension. Numerical examples illustrate the mechanism in both a non-convex quartic optimization problem and a distributed binary-classification problem.
\end{itemize}

Our analysis combines stochastic perturbation and concentration arguments with the network consensus dynamics to establish these finite-time probabilistic second-order guarantees. This extended version provides the complete proofs and intermediate results underlying the main theorems, together with the additional non-convex quartic optimization example.

\subsection{Related work}

The \textbf{DGD} algorithm \cite{nedic2009distributed} combines local gradient steps with neighbor averaging. Dual-decomposition methods \cite{terelius2011decentralized} provide an alternative by explicitly enforcing agreement through constraints and dual updates, while variants of \textbf{DGD} have been extensively studied \cite{nedic2010constrained,jakovetic2014fast,chen2012fast1,chen2012fast2,duchi2011dual}. Exact-convergence variants, including \textbf{EXTRA} and gradient-tracking methods, enable constant-step convergence \cite{shi2015extra,nedic2017achieving,di2015distributed,di2016next}. Convergence guarantees for convex and strongly convex settings are well established \cite{yuan2016convergence,tsianos2012distributed,qu2017harnessing,chen2012diffusion,xin2019distributed,pu2020push,pu2021distributed}, with related results under relaxed assumptions \cite{yi2022primal,yazdani2021asynchronous,haddadpour2019local}.

For non-convex problems, constant step-size \textbf{DGD} converges to first-order stationary points \cite{zeng2018nonconvex}, which may include saddle points. Although saddle points are almost surely avoided under suitable conditions \cite{daneshmand2020second,swenson2019distributed,swenson2022distributed}, escaping them can require exponential time in worst-case scenarios \cite{du2017gradient}; moreover, almost-sure avoidance alone does not quantify how quickly an iterate leaves a saddle-point neighborhood. Second-order methods can exploit negative curvature directly, but their Hessian-related computational and communication requirements make them less attractive in distributed settings. Perturbation-based centralized first-order methods provide an alternative by enabling efficient saddle-point escape without explicitly computing Hessians \cite{ge2015escaping,jin2017escape,jin2021nonconvex}, with polynomial or polylogarithmic dependence on problem dimension.

Extending perturbation-based saddle-point escape to distributed optimization introduces an additional difficulty. Locally injected perturbations facilitate escape from negative-curvature regions, but they also interact with the consensus dynamics and can increase disagreement among agents. Consequently, a distributed analysis must simultaneously control saddle-point escape and network consensus, while ultimately providing guarantees for the individual agents rather than only for an averaged network trajectory. Existing work \cite{vlaski2021distributed1,vlaski2021distributed2} analyzes second-order behavior through network-centroid dynamics together with agent-agreement guarantees, whereas the present work provides an explicit finite-time agent-wise proximity bound to the local-minimizer set of the original objective together with a consensus-error bound. More recent approaches \cite{wang2023decentralized,bo2024quantization} consider alternative perturbation or quantization mechanisms under different step-size settings. Our preliminary work \cite{qin2023second} combines lifting techniques \cite{yuan2016convergence,zeng2018nonconvex,daneshmand2020second} with stochastic concentration analysis to study \textbf{NDGD}. The present work strengthens these results by refining the finite-time bounds, explicitly characterizing the effects of network structure and algorithm parameters, and establishing the agent-wise local-minimizer guarantee for the original objective.

\subsection{Paper Organization}

The assumptions and preliminary results are collected in Section~\ref{sec: Assumptions and Supporting Results}. The main convergence guarantees are presented in Section~\ref{sec: Main Results}. Complete proofs, including the intermediate concentration and saddle-point escape arguments, are provided in Section~\ref{sec: Proofs}. Section~\ref{sec: Numerical Examples} presents two illustrative numerical examples: a non-convex quartic optimization problem and a distributed binary-classification problem. Conclusions are given in Section~\ref{sec: Conclusions and Discussion}.

\subsection{Notation}
\label{sec: Notation}

Let $\mathbf{I}_{n}$ denote the $n\times n$ identity matrix, $\bm{1}_{n}$ denote the $n$-vector with all entries equal to $1$, and both $\mathbf{A}_{ij}$ and $a_{ij}$ denote the entry in row $i$ and column $j$ of the matrix $\mathbf{A}$. For a square symmetric matrix $\mathbf{B}$, the eigenvalues are real valued, and we use $\lambda_{\min}(\mathbf{B})$, $\lambda_{\max}(\mathbf{B})$, and $\|\mathbf{B}\|$ to denote the minimum eigenvalue, maximum eigenvalue, and spectral norm, respectively. Further, $\mathbf{B}_{1}\prec(\mathrm{resp.}\preceq)\mathbf{B}_{2}$ is equivalent to $\lambda_{\min}(\mathbf{B}_{2}-\mathbf{B}_{1})>(\mathrm{resp.}\geq)0$. The Kronecker product is denoted by $\otimes$. The distance from a point $\mathbf{x}\in\mathbb{R}^{n}$ to a set $\mathcal{Y}\subseteq\mathbb{R}^{n}$ is denoted by $\operatorname{dist}(\mathbf{x},\mathcal{Y}):=\inf_{\mathbf{y}\in\mathcal{Y}}\|\mathbf{x}-\mathbf{y}\|$. We say that $\mathbf{x}$ is $\delta$-close to a point $\mathbf{y}$ (resp., a set $\mathcal{Y}$) if $\|\mathbf{x}-\mathbf{y}\|\leq\delta$ (resp., $\operatorname{dist}(\mathbf{x},\mathcal{Y})\leq\delta$). Let $\mathbf{x}\sim\mathcal{N}(\bm{\mu},\bm{\Sigma})$ denote an $n$-dimensional Gaussian random vector with mean $\bm{\mu}\in\mathbb{R}^{n}$ and covariance matrix $\bm{\Sigma}\in\mathbb{R}^{n\times n}$. Let $\mathbb{E}$ and $\mathbb{P}$ denote the expectation and probability operators, respectively, and let $\lceil\cdot\rceil$ denote the ceiling function. Unless stated otherwise, all iteration indices are positive integers, and superscript indexing denotes iterates.

\section{Assumptions and Preliminary Results}
\label{sec: Assumptions and Supporting Results}
This section summarizes the definitions, assumptions, and preliminary results used in the lifted analysis of fixed step-size \textbf{DGD} and \textbf{NDGD}. In particular, we introduce the lifted objective and auxiliary function used to represent the distributed dynamics, and establish the relationship between local minimizers of the original objective $f$ and those of the lifted function $Q_{\alpha}$. These results underpin the convergence analysis developed in Section~\ref{sec: Main Results}.

\subsection{Definitions and Assumptions}

\begin{definition}[First order stationarity]
\label{def: First order stationary point}
For a differentiable function $h:\mathbb{R}^{n}\rightarrow\mathbb{R}$, the point $\mathbf{x}$ is said to be first-order stationary if $\|\nabla h(\mathbf{x})\|=0$, where $\nabla h:\mathbb{R}^{n}\rightarrow\mathbb{R}^{n}$ denotes the gradient of $h$.
\end{definition}

\begin{definition}[Second order stationarity]
\label{def: Local minimizer and saddle}
For a twice differentiable function $h:\mathbb{R}^{n}\rightarrow\mathbb{R}$, a first-order stationary point $\mathbf{x}$ is: (i) a local minimizer if $\nabla^{2}h(\mathbf{x})\succ0$; (ii) a local maximizer if $\nabla^{2}h(\mathbf{x})\prec0$; and (iii) a strict saddle point if $\lambda_{\min}(\nabla^{2}h(\mathbf{x}))<0$ and $\lambda_{\max}(\nabla^{2}h(\mathbf{x}))>0$, where $\nabla^{2}h:\mathbb{R}^{n}\rightarrow\mathbb{R}^{n\times n}$ denotes the Hessian of $h$.
\end{definition}

\begin{assumption}[Strict saddle]
\label{ass: Strict saddle}
The function $f:\mathbb{R}^{n}\rightarrow\mathbb{R}$ in~\eqref{eq: Unconstrained optimization problem} is such that for all first-order stationary points $\mathbf{x}$, either $\lambda_{\min}(\nabla^{2}f(\mathbf{x}))>0$ (i.e., $\mathbf{x}$ is a local minimizer), or $\lambda_{\min}(\nabla^{2}f(\mathbf{x}))<0$ (i.e., $\mathbf{x}$ is a saddle point or a maximizer).
\end{assumption}

\begin{assumption}[Lipschitz continuity]
\label{ass: Lipschitz}
Each $f_i$ in~\eqref{eq: Unconstrained optimization problem} is both $L_{f_i}^{g}$-gradient Lipschitz and $L_{f_i}^{H}$-Hessian Lipschitz, i.e., for all $\mathbf{x},\mathbf{y}\in\mathbb{R}^{n}$ and each $i\in\mathcal{V}$,
\begin{gather*}
\|\nabla f_i(\mathbf{x})-\nabla f_i(\mathbf{y})\|\le L_{f_i}^{g}\|\mathbf{x}-\mathbf{y}\|,
\end{gather*}
and
\begin{gather*}
\|\nabla^{2}f_i(\mathbf{x})-\nabla^{2}f_i(\mathbf{y})\|\le L_{f_i}^{H}\|\mathbf{x}-\mathbf{y}\|.
\end{gather*}
\end{assumption}

\begin{assumption}[Coercivity]
\label{ass: Coercivity}
Each $f_i$ in~\eqref{eq: Unconstrained optimization problem} is coercive, i.e., $f_i(\mathbf{x})\rightarrow\infty$ as $\|\mathbf{x}\|\rightarrow\infty$, which is equivalent to compactness of all sublevel sets by continuity. Additionally, $\|\nabla f\|$ is coercive.
\end{assumption}

\begin{remark}
Coercivity of $\|\nabla f\|$ is commonly satisfied in optimization. In particular, adding a regularization term of the form $\frac{\lambda}{2}\|\mathbf{x}\|^{2}$ to the objective function ensures this condition.
\end{remark}

\begin{assumption}[Bounded gradient disagreement]
\label{ass: Bounded gradient disagreement}
There exists a constant $D>0$ such that for any two agents $i,j\in\mathcal{V}$ and every $\mathbf{x}\in\mathbb{R}^{n}$, $f_i$ and $f_j$ in~\eqref{eq: Unconstrained optimization problem} satisfy
\begin{gather*}
\|\nabla f_i(\mathbf{x})-\nabla f_j(\mathbf{x})\|\le D.
\end{gather*}
\end{assumption}

\begin{remark}
Assumption~\ref{ass: Bounded gradient disagreement} is weaker than the more common assumption of uniformly bounded gradients. It can be satisfied when all agents are minimizing the same cost function with the same regularizers, with different parameters. For further discussion, see~\cite{vlaski2021distributed1,swenson2019annealing}.
\end{remark}

\begin{assumption}[Connected network]
\label{ass: Network}
The undirected graph $\mathcal{G}(\mathcal{V},\mathcal{E})$ is connected.
\end{assumption}

\begin{assumption}[Mixing matrix]
\label{ass: Mixing matrix}
The mixing matrix $\mathbf{W}\in\mathbb{R}^{m\times m}$ relating to the agent iteration~\eqref{eq: NDGD} satisfies:
\begin{enumerate}
\renewcommand{\labelenumi}{\roman{enumi})}
\item $\mathbf{W}$ is consistent with the graph $\mathcal{G}(\mathcal{V},\mathcal{E})$ in the sense that $w_{ii}>0$ for all $i\in\mathcal{V}$, $w_{ij}>0$ if $(i,j)\in\mathcal{E}$, and $w_{ij}=0$ otherwise;
\item $\mathbf{W}=\mathbf{W}^{\top}$;
\item $\operatorname{null}\{\mathbf{I}_{m}-\mathbf{W}\}=\operatorname{span}\{\bm{1}_{m}\}$;
\item $\bm{0}\prec\mathbf{W}\preceq\mathbf{I}_{m}$.
\end{enumerate}
\end{assumption}

\begin{remark}
\label{rem: Mixing matrix}
For a connected graph, a graph-compatible symmetric doubly stochastic $\mathbf{W}$ with positive self-weights satisfies conditions ii) and iii); condition iv) additionally requires $\mathbf{W}$ to be positive definite. In particular,
\begin{gather*}
\mathbf{W}\bm{1}_{m}=\bm{1}_{m},\qquad \bm{1}_{m}^{\top}\mathbf{W}=\bm{1}_{m}^{\top}.
\end{gather*}
\end{remark}

\begin{assumption}[Random perturbation]
\label{ass: Random perturbation}
Given $\sigma>0$, the perturbations $\{\mathbf{n}_{i}^{k}\}_{i,k}$ are independent, with
\begin{gather*}
\mathbf{n}_{i}^{k}\sim\mathcal{N}(\bm{0},\sigma^{2}\mathbf{I}_{n})
\end{gather*}
for all $i\in\mathcal{V}$ and $k>0$.
\end{assumption}

\begin{remark}
If Assumption~\ref{ass: Random perturbation} holds, then the sequence $\{\mathbf{n}^{k}\}_{k>0}$, where
\begin{gather*}
\mathbf{n}^{k}=[(\mathbf{n}_{1}^{k})^{\top},\ldots,(\mathbf{n}_{m}^{k})^{\top}]^{\top}\in(\mathbb{R}^{n})^{m},
\end{gather*}
is i.i.d., with
\begin{gather*}
\mathbf{n}^{k}\sim\mathcal{N}(\bm{0},\sigma^{2}\mathbf{I}_{mn}),\qquad \mathbb{E}[\|\mathbf{n}^{k}\|^{2}]=mn\sigma^{2}.
\end{gather*}
\end{remark}

\subsection{Distributed Gradient Descent}

In this section, we examine the standard \textbf{DGD} algorithm through the lens of an auxiliary function in a lifted setting. This methodology is subsequently applied to \textbf{NDGD} with random perturbations as per~\eqref{eq: NDGD}. We first introduce the lifted objective
\begin{gather}
\label{eq: Constrained optimization problem}
F(\hat{\mathbf{x}})=\sum_{i=1}^{m}f_i(\hat{\mathbf{x}}_i),
\end{gather}
where $\hat{\mathbf{x}}=[\hat{\mathbf{x}}_{1}^{\top},\ldots,\hat{\mathbf{x}}_{m}^{\top}]^{\top}\in(\mathbb{R}^{n})^{m}$ is the stacked vector collecting the local copies $\hat{\mathbf{x}}_i\in\mathbb{R}^{n}$ maintained by the $m$ agents. When the local copies are in consensus, i.e., $\hat{\mathbf{x}}=\bm{1}_{m}\otimes\mathbf{x}$, we have
\begin{gather*}
F(\bm{1}_{m}\otimes\mathbf{x})=\sum_{i=1}^{m}f_i(\mathbf{x})=f(\mathbf{x}),
\end{gather*}
so minimizing $F$ subject to agreement of all local copies is equivalent to~\eqref{eq: Unconstrained optimization problem}; the lifted representation is used to analyze its distributed implementation. Moreover,
\begin{gather*}
\nabla F(\hat{\mathbf{x}})=[\nabla f_{1}(\hat{\mathbf{x}}_{1})^{\top},\ldots,\nabla f_{m}(\hat{\mathbf{x}}_{m})^{\top}]^{\top},
\end{gather*}
and
\begin{gather*}
\nabla^{2}F(\hat{\mathbf{x}})=\bigoplus_{i=1}^{m}\nabla^{2}f_i(\hat{\mathbf{x}}_i),
\end{gather*}
so $\nabla^{2}F$ is block diagonal.

\begin{remark}
Since Assumption~\ref{ass: Lipschitz} holds for all points in $\mathbb{R}^{n}$, by the block structure of $F$ in~\eqref{eq: Constrained optimization problem}, $F$ has an $L_F^{g}$-Lipschitz continuous gradient and an $L_F^{H}$-Lipschitz continuous Hessian, where
\begin{gather*}
L_F^{g}=\max_i\{L_{f_i}^{g}\},\qquad L_F^{H}=\max_i\{L_{f_i}^{H}\}.
\end{gather*}
These constants are introduced for later use in Section~\ref{sec: Main Results}.
\end{remark}

\begin{definition}[Consensual first order stationarity]
\label{def: consensual first order stationary point}
For $F:(\mathbb{R}^{n})^{m}\rightarrow\mathbb{R}$ in~\eqref{eq: Constrained optimization problem}, the point $\hat{\mathbf{x}}\in(\mathbb{R}^{n})^{m}$ is said to be a consensual first order stationary point if:
\begin{enumerate}
\renewcommand{\labelenumi}{\roman{enumi})}
\item $\hat{\mathbf{x}}=\bm{1}_{m}\otimes\operatorname{av}(\hat{\mathbf{x}})$, where
\begin{align}
\label{eq: Average of xhat}
\operatorname{av}(\hat{\mathbf{x}})=\frac{1}{m}(\bm{1}_{m}\otimes\mathbf{I}_{n})^{\top}\hat{\mathbf{x}}\in\mathbb{R}^{n};
\end{align}
\item $(\bm{1}_{m}\otimes\mathbf{I}_{n})^{\top}\nabla F(\hat{\mathbf{x}})=\sum_{i=1}^{m}\nabla f_i(\hat{\mathbf{x}}_i)=\bm{0}$.
\end{enumerate}
\end{definition}

If condition i) holds, then $\hat{\mathbf{x}}$ is in consensus, i.e., $\hat{\mathbf{x}}_i=\hat{\mathbf{x}}_j$ for all $i,j\in\mathcal{V}$. Further, if condition ii) also holds, then for every $i\in\mathcal{V}$, the point $\hat{\mathbf{x}}_i$ is first-order stationary for $f$.

\begin{definition}[Consensual second-order stationarity]
\label{def: Approximately consensual second order stationary point}
For $F:(\mathbb{R}^{n})^{m}\rightarrow\mathbb{R}$ in~\eqref{eq: Constrained optimization problem}, the point $\hat{\mathbf{x}}\in(\mathbb{R}^{n})^{m}$ is said to be an $(\eta,\epsilon,\gamma)$-consensual second-order stationary point if:
\begin{enumerate}
\renewcommand{\labelenumi}{\roman{enumi})}
\item $\|\hat{\mathbf{x}}-\bm{1}_{m}\otimes\operatorname{av}(\hat{\mathbf{x}})\|\le\eta$ with $\operatorname{av}(\cdot)$ as per~\eqref{eq: Average of xhat};
\item $\|(\bm{1}_{m}\otimes\mathbf{I}_{n})^{\top}\nabla F(\hat{\mathbf{x}})\|=\|\sum_{i=1}^{m}\nabla f_i(\hat{\mathbf{x}}_i)\|\le\epsilon$;
\item $\lambda_{\min}(H(\hat{\mathbf{x}}))=\lambda_{\min}(\sum_{i=1}^{m}\nabla^{2}f_i(\hat{\mathbf{x}}_i))\ge-\gamma$, where
\begin{gather*}
H(\hat{\mathbf{x}})=(\bm{1}_{m}\otimes\mathbf{I}_{n})^{\top}\nabla^{2}F(\hat{\mathbf{x}})(\bm{1}_{m}\otimes\mathbf{I}_{n}).
\end{gather*}
\end{enumerate}
\end{definition}

Recall the standard fixed step-size \textbf{DGD} algorithm~\cite{nedic2009distributed}:
\begin{align*}
\hat{\mathbf{x}}_{i}^{k+1}=\sum_{j=1}^{m}w_{ij}\hat{\mathbf{x}}_{j}^{k}-\alpha\nabla f_i(\hat{\mathbf{x}}_{i}^{k}),
\end{align*}
for $i\in\mathcal{V}$, with $\alpha>0$. Since $\mathbf{W}\bm{1}_{m}=\bm{1}_{m}$, the mixing term satisfies
\begin{gather*}
\sum_{j=1}^{m}w_{ij}\hat{\mathbf{x}}_{j}^{k}
=\hat{\mathbf{x}}_{i}^{k}-\sum_{j=1}^{m}w_{ij}(\hat{\mathbf{x}}_{i}^{k}-\hat{\mathbf{x}}_{j}^{k}),
\end{gather*}
so the weighted-averaging form is equivalent to the standard consensus-error form. In aggregate,
\begin{align}
\label{eq: DGD aggregate form}
\hat{\mathbf{x}}^{k+1}=\hat{\mathbf{W}}\hat{\mathbf{x}}^{k}-\alpha\nabla F(\hat{\mathbf{x}}^{k}),
\end{align}
where $\hat{\mathbf{W}}:=\mathbf{W}\otimes\mathbf{I}_{n}$.

As proposed in earlier works including~\cite{yuan2016convergence,zeng2018nonconvex,daneshmand2020second}, the convergence properties can be studied via the auxiliary function
\begin{align}
\label{eq: Q}
\begin{aligned}
Q_{\alpha}(\hat{\mathbf{x}})
&=\sum_{i=1}^{m}f_i(\hat{\mathbf{x}}_i)+\frac{1}{2\alpha}\sum_{i=1}^{m}\sum_{j=1}^{m}(\mathbf{I}_{m}-\mathbf{W})_{ij}(\hat{\mathbf{x}}_i)^{\top}\hat{\mathbf{x}}_j\\
&=F(\hat{\mathbf{x}})+\frac{1}{2\alpha}\|\hat{\mathbf{x}}\|_{\mathbf{I}_{mn}-\hat{\mathbf{W}}}^{2}.
\end{aligned}
\end{align}
The \textbf{DGD} update~\eqref{eq: DGD aggregate form} can be interpreted as standard gradient descent applied to~\eqref{eq: Q}, i.e.,
\begin{align}
\label{eq: GD on Q}
\hat{\mathbf{x}}^{k+1}=\hat{\mathbf{x}}^{k}-\alpha\nabla Q_{\alpha}(\hat{\mathbf{x}}^{k}).
\end{align}
Iterating~\eqref{eq: DGD aggregate form} and~\eqref{eq: GD on Q} from the same initialization yields the same sequence of iterates; see~\cite[Lemma~1]{zeng2018nonconvex}. Similarly, the \textbf{NDGD} update~\eqref{eq: NDGD} can be interpreted as noisy gradient descent applied to~\eqref{eq: Q}, i.e.,
\begin{align}
\label{eq: NDGD2}
\hat{\mathbf{x}}^{k+1}=\hat{\mathbf{x}}^{k}-\alpha(\nabla Q_{\alpha}(\hat{\mathbf{x}}^{k})+\mathbf{n}^{k}).
\end{align}

\begin{remark}
If Assumption~\ref{ass: Lipschitz} holds, with $L_F^{g}$ and $L_F^{H}$ defined above, then $Q_{\alpha}$ in~\eqref{eq: Q} has an $L_{Q_{\alpha}}^{g}$-Lipschitz continuous gradient and an $L_{Q_{\alpha}}^{H}$-Lipschitz continuous Hessian, where
\begin{gather*}
L_{Q_{\alpha}}^{g}=L_F^{g}+\alpha^{-1}(1-\lambda_{\min}(\mathbf{W})),\qquad L_{Q_{\alpha}}^{H}=L_F^{H}.
\end{gather*}
\end{remark}

\begin{remark}
If Assumption~\ref{ass: Coercivity} holds, then for any fixed $\alpha>0$, $Q_{\alpha}$ is coercive. In particular, for all $c>0$, the sublevel set
\begin{gather*}
\mathcal{S}_{c}=\{\hat{\mathbf{x}}\in(\mathbb{R}^{n})^{m}:Q_{\alpha}(\hat{\mathbf{x}})\le c\}
\end{gather*}
is contained in
\begin{gather*}
\mathcal{S}_{c}^{\prime}=\{\hat{\mathbf{x}}\in(\mathbb{R}^{n})^{m}:F(\hat{\mathbf{x}})\le c\},
\end{gather*}
since $\frac{1}{2\alpha}\|\hat{\mathbf{x}}\|_{\mathbf{I}_{mn}-\hat{\mathbf{W}}}^{2}\ge0$. Thus, $\mathcal{S}_{c}\subseteq\mathcal{S}_{c}^{\prime}$. By Assumption~\ref{ass: Coercivity}, $F$ is coercive, so every $\mathcal{S}_{c}^{\prime}$ is compact by continuity. Therefore, each closed subset $\mathcal{S}_{c}$ is bounded, and $Q_{\alpha}$ is coercive.
\end{remark}

\begin{definition}[Local minimizers]
\label{def: Local minimizers}
Let $\mathcal{X}_{f}^{\star}$ and $\mathcal{X}_{Q_{\alpha}}^{\star}$ denote the sets of local minimizers of $f$ and $Q_{\alpha}$, respectively:
\begin{align}
\label{eq: minimizer set}
\begin{aligned}
\mathcal{X}_{f}^{\star}
&=\{\mathbf{x}\in\mathbb{R}^{n}:\nabla f(\mathbf{x})=\bm{0},~\nabla^{2}f(\mathbf{x})\succ0\},\\
\mathcal{X}_{Q_{\alpha}}^{\star}
&=\{\hat{\mathbf{x}}\in(\mathbb{R}^{n})^{m}:\nabla Q_{\alpha}(\hat{\mathbf{x}})=\bm{0},~\nabla^{2}Q_{\alpha}(\hat{\mathbf{x}})\succ0\}.
\end{aligned}
\end{align}
\end{definition}

Next, we establish a second-order relationship between $F$ and $Q_{\alpha}$.

\begin{lemma}
\label{lem: Minimum eigenvalue of Q}
Given any $\alpha>0$, for all $\hat{\mathbf{x}}\in(\mathbb{R}^{n})^{m}$,
\begin{gather*}
\lambda_{\min}(\nabla^{2}F(\hat{\mathbf{x}}))
\le\lambda_{\min}(\nabla^{2}Q_{\alpha}(\hat{\mathbf{x}}))
\le\frac{1}{m}\lambda_{\min}\left(\sum_{i=1}^{m}\nabla^{2}f_i(\hat{\mathbf{x}}_i)\right).
\end{gather*}
\end{lemma}

\begin{proof}
To establish the lower bound, recall that
\begin{gather*}
\nabla^{2}Q_{\alpha}(\hat{\mathbf{x}})
=\nabla^{2}F(\hat{\mathbf{x}})+\alpha^{-1}(\mathbf{I}_{mn}-\hat{\mathbf{W}}).
\end{gather*}
Since $\mathbf{I}_{mn}-\hat{\mathbf{W}}\succeq0$,
\begin{gather*}
\lambda_{\min}(\nabla^{2}Q_{\alpha}(\hat{\mathbf{x}}))
\ge\lambda_{\min}(\nabla^{2}F(\hat{\mathbf{x}})).
\end{gather*}
For the upper bound, let $\mathbf{u}\in\mathbb{R}^{n}$ be a unit-norm eigenvector corresponding to $\lambda_{\min}(\sum_{i=1}^{m}\nabla^{2}f_i(\hat{\mathbf{x}}_i))$ and define $\hat{\mathbf{v}}=\frac{1}{\sqrt{m}}\bm{1}_{m}\otimes\mathbf{u}$. Then $\|\hat{\mathbf{v}}\|=1$ and $(\mathbf{I}_{mn}-\hat{\mathbf{W}})\hat{\mathbf{v}}=\bm{0}$. Hence,
\begin{align*}
\lambda_{\min}(\nabla^{2}Q_{\alpha}(\hat{\mathbf{x}}))
&\le\hat{\mathbf{v}}^{\top}\nabla^{2}Q_{\alpha}(\hat{\mathbf{x}})\hat{\mathbf{v}}\\
&=\hat{\mathbf{v}}^{\top}\nabla^{2}F(\hat{\mathbf{x}})\hat{\mathbf{v}}\\
&=\frac{1}{m}\mathbf{u}^{\top}\left(\sum_{i=1}^{m}\nabla^{2}f_i(\hat{\mathbf{x}}_i)\right)\mathbf{u}\\
&=\frac{1}{m}\lambda_{\min}\left(\sum_{i=1}^{m}\nabla^{2}f_i(\hat{\mathbf{x}}_i)\right),
\end{align*}
as claimed.
\end{proof}

\subsection{Relationships Between Local Minimizers of \texorpdfstring{$f$ and $Q_{\alpha}$}{f and Q-alpha}}

In this section, additional regularity conditions on $f$ are introduced. These local conditions are formulated to ensure that the $i$-th component $\hat{\mathbf{x}}_{i}^{\star}$ of a local minimizer $\hat{\mathbf{x}}^{\star}\in\mathcal{X}_{Q_{\alpha}}^{\star}$ can be made arbitrarily close to the set $\mathcal{X}_{f}^{\star}$ by choosing a sufficiently small step-size $\alpha>0$. This expands upon~\cite{daneshmand2020second}, which studies the relationship between stationary points of $f$ and $Q_{\alpha}$.

\begin{assumption}[Local regularity]
\label{ass: Local convexity}
There exist positive constants $\alpha^{\prime},\gamma,\delta,\epsilon,\eta,\mu>0$ such that for any point $\hat{\mathbf{x}}=[\hat{\mathbf{x}}_{1}^{\top},\ldots,\hat{\mathbf{x}}_{m}^{\top}]^{\top}\in(\mathbb{R}^{n})^{m}$ satisfying
\begin{gather*}
\|\hat{\mathbf{x}}-\bm{1}_{m}\otimes\operatorname{av}(\hat{\mathbf{x}})\|\le\eta,
\end{gather*}
with $\operatorname{av}(\cdot)$ as per~\eqref{eq: Average of xhat}, at least one of the following is true:
\begin{enumerate}
\renewcommand{\labelenumi}{\roman{enumi})}
\item $\|\sum_{i=1}^{m}\nabla f_i(\hat{\mathbf{x}}_i)\|>\epsilon$;
\item $\lambda_{\min}(\sum_{i=1}^{m}\nabla^{2}f_i(\hat{\mathbf{x}}_i))<-\gamma$;
\item for all $0<\alpha\le\alpha^{\prime}$, there exists $\hat{\mathbf{x}}^{\prime}\in(\mathbb{R}^{n})^{m}$ such that $\|\hat{\mathbf{x}}^{\prime}-\hat{\mathbf{x}}\|\le\delta$,
\begin{gather*}
\sum_{j=1}^{m}w_{ij}(\hat{\mathbf{x}}_{j}^{\prime}-\hat{\mathbf{x}}_{i}^{\prime})
=\alpha\nabla f_i(\hat{\mathbf{x}}_{i}^{\prime})
\quad\text{for all }i\in\mathcal{V},
\end{gather*}
and
\begin{gather*}
\lambda_{\min}(\nabla^{2}F(\hat{\mathbf{y}}))\ge\mu
\quad\text{for all }\hat{\mathbf{y}}\text{ with }\|\hat{\mathbf{y}}-\hat{\mathbf{x}}^{\prime}\|\le3\delta,
\end{gather*}
i.e., $F$ is strongly convex in the vicinity of $\hat{\mathbf{x}}^{\prime}$.
\end{enumerate}
\end{assumption}

Intuitively, for points with small consensus error, conditions i)--ii) cover the cases of a sufficiently large aggregate gradient or a direction of sufficiently negative curvature. If neither occurs, condition iii) requires a nearby compatible stationary configuration with robust local strong convexity, so an approximate consensual second-order stationary point can be associated with a nearby local minimizer. The large-gradient, negative-curvature, and local-curvature structure in Assumption~\ref{ass: Local convexity} is analogous to geometric conditions studied in structured non-convex problems such as low-rank matrix optimization~\cite{zhu2021global}, phase retrieval~\cite{sun2018geometric}, and tensor decomposition~\cite{ge2015escaping}. Assumption~\ref{ass: Local convexity} is tailored to the distributed formulation and is more specialized than a standard strict-saddle condition because condition iii) additionally imposes local strong convexity when the large-gradient and negative-curvature alternatives fail. This additional structure is used only for the local-minimizer guarantee in Theorem~\ref{the: Second order guarantee} and is not required for Theorem~\ref{the: Approximate second order}.

\begin{lemma}
\label{lem: Local minimizers of Q}
Let Assumptions~\ref{ass: Strict saddle}, \ref{ass: Lipschitz}, \ref{ass: Coercivity}, \ref{ass: Network}, \ref{ass: Mixing matrix}, and \ref{ass: Local convexity} hold. With $\alpha^{\prime},\delta,\mu>0$ as per Assumption~\ref{ass: Local convexity}, there exists a threshold $0<\bar{\alpha}(\delta)\le\alpha^{\prime}$ such that for every $0<\alpha\le\bar{\alpha}(\delta)$, if
\begin{gather}
\label{eq: consensus error at xstar}
\hat{\mathbf{x}}^{\star}\in\left\{\hat{\mathbf{x}}\in\mathcal{X}_{Q_{\alpha}}^{\star}:\|\hat{\mathbf{x}}-\bm{1}_{m}\otimes\operatorname{av}(\hat{\mathbf{x}})\|\le\frac{5}{\mu}\sqrt{\alpha}\right\},
\end{gather}
with $\operatorname{av}(\cdot)$ as per~\eqref{eq: Average of xhat}, then for all $i\in\mathcal{V}$,
\begin{gather*}
\operatorname{dist}(\hat{\mathbf{x}}_{i}^{\star},\mathcal{X}_{f}^{\star})
\le\left(1+\frac{\sqrt{m}L_F^{g}}{\mu}\right)\frac{5}{\mu}\sqrt{\alpha}.
\end{gather*}
\end{lemma}

The complete proof is provided in Section~\ref{sec: Proof of Lemma Local minimizers of Q}. The essence of Lemma~\ref{lem: Local minimizers of Q} is that, under the stated regularity conditions and with sufficiently small $\alpha$, each agent-wise component of a local minimizer of $Q_{\alpha}$ with sufficiently small consensus error lies close to a local minimizer of $f$. Later, we show that the set in~\eqref{eq: consensus error at xstar} is nonempty under the conditions of the main convergence result.

\subsection{Other Useful Preliminary Results}

Before presenting the main results, we introduce several additional probabilistic and logical tools used in the proofs.

\begin{proposition}[Boole-Fréchet Inequality]
\label{pro: Boole-frechet Inequality}
Let $\mathcal{E}_{1},\mathcal{E}_{2},\ldots,\mathcal{E}_{n}$ be events in a probability space. Then
\begin{gather*}
\max_{1\le i\le n}\mathbb{P}(\mathcal{E}_{i})
\le\mathbb{P}\left(\bigcup_{i=1}^{n}\mathcal{E}_{i}\right)
\le\sum_{i=1}^{n}\mathbb{P}(\mathcal{E}_{i}),
\end{gather*}
which further implies
\begin{gather}
\label{eq: Boole's inequality}
\mathbb{P}\left[\bigcap_{i=1}^{n}\mathcal{E}_{i}\right]
\ge1-\sum_{i=1}^{n}\mathbb{P}[\bar{\mathcal{E}}_{i}]
=1-\sum_{i=1}^{n}(1-\mathbb{P}[\mathcal{E}_{i}]),
\end{gather}
where $\bar{\mathcal{E}}_{i}$ denotes the complement of $\mathcal{E}_{i}$.
\end{proposition}

\begin{proposition}[Material Implication Equivalence, {\cite{hurley2011concise}}]
\label{pro: Material implication}
For any statements $P_A$ and $P_B$, the material implication
\begin{gather*}
P_A\Rightarrow P_B
\end{gather*}
is logically equivalent to
\begin{gather*}
\bar{P}_A\lor P_B.
\end{gather*}
\end{proposition}

\begin{proposition}[{\cite[Lemma~1]{laurent2000adaptive}}]
\label{pro: chi-square concentration}
Let $(Y_1,\ldots,Y_D)$ be i.i.d. Gaussian variables with mean $0$ and variance $1$. Let $a_1,\ldots,a_D$ be nonnegative and define
\begin{gather*}
|a|_{\infty}=\sup_{i=1,\ldots,D}|a_i|,\qquad |a|_{2}^{2}=\sum_{i=1}^{D}a_i^{2}.
\end{gather*}
Let $Z=\sum_{i=1}^{D}a_i(Y_i^{2}-1)$. Then, for any $x>0$,
\begin{gather*}
\mathbb{P}\left[Z\ge2|a|_{2}\sqrt{x}+2|a|_{\infty}x\right]\le\exp(-x),
\end{gather*}
and
\begin{gather*}
\mathbb{P}\left[Z\le-2|a|_{2}\sqrt{x}\right]\le\exp(-x).
\end{gather*}
As an immediate corollary, let $U$ be a $\chi^{2}$ random variable with $D$ degrees of freedom. Then, for any $x>0$,
\begin{align}
\label{eq: chi-square concentration}
\begin{aligned}
\mathbb{P}[U-D\ge2\sqrt{Dx}+2x]&\le\exp(-x),\\
\mathbb{P}[D-U\ge2\sqrt{Dx}]&\le\exp(-x).
\end{aligned}
\end{align}
\end{proposition}

\begin{proposition}[{\cite[Theorem~7.3]{chung2006concentration}}]
\label{pro: Supermartingale concentration}
Given a filtration
\begin{gather*}
\{\emptyset,\Omega\}=\mathcal{F}^{0}\subset\mathcal{F}^{1}\subset\cdots\subset\mathcal{F}^{t}=\mathcal{F},
\end{gather*}
for $1\le k\le t$, suppose a random variable $X^{k}$ is $\mathcal{F}^{k}$-measurable and satisfies almost surely:
\begin{enumerate}
\item $\mathbb{E}[X^{k}\mid\mathcal{F}^{k-1}]\le X^{k-1}$;
\item $\operatorname{Var}(X^{k}\mid\mathcal{F}^{k-1})\le\sigma_{k}^{2}+\phi_{k}X^{k-1}$;
\item $X^{k}-\mathbb{E}[X^{k}\mid\mathcal{F}^{k-1}]\le a_{k}+M$.
\end{enumerate}
Here, $\sigma_k,a_k,\phi_k$, and $M$ are nonnegative constants. Then, for all $\lambda>0$,
\begin{multline*}
\mathbb{P}(X^{t}\ge X^{0}+\lambda)\\
\le\exp\left(
-\frac{\lambda^{2}}
{2\left(\sum_{k=1}^{t}(\sigma_{k}^{2}+a_{k}^{2})+(X^{0}+\lambda)\sum_{k=1}^{t}\phi_{k}+\frac{M\lambda}{3}\right)}
\right).
\end{multline*}
\end{proposition}

\section{Main Results}
\label{sec: Main Results}
Theorem~\ref{the: Approximate second order} establishes high-probability convergence of~\eqref{eq: NDGD} to a consensual approximate second-order stationary point, while Theorem~\ref{the: Second order guarantee} further provides local-minimizer guarantees under Assumption~\ref{ass: Local convexity}. The supporting lemmas are proved in Section~\ref{sec: Proofs}. The mixing matrix $\mathbf{W}$, step-size $\alpha$, and noise variance $\sigma^{2}$ are selected offline using global problem and network information, fixed during execution, and shared with the agents. During execution, each agent uses only its local mixing weights and gradient information and generates $\mathbf{n}_{i}^{k}$ locally without additional communication.

We begin with Lemma~\ref{lem: Consensus error}, which provides a probabilistic bound on the consensus error. The complete proof is provided in Section~\ref{sec: Proof of Lemma Consensus error}.

\begin{lemma}
\label{lem: Consensus error}
Let Assumptions~\ref{ass: Lipschitz}, \ref{ass: Bounded gradient disagreement}, \ref{ass: Network}, \ref{ass: Mixing matrix}, and \ref{ass: Random perturbation} hold. For any $k>0$, there exists $\rho_{1}\ge1$ such that for $\rho\ge\rho_{1}$, with step-size
\begin{gather}
\label{eq: Alpha}
\alpha=\frac{\lambda_{\min}(\mathbf{W})}{L_{F}^{g}}\sqrt{\frac{1}{\rho}},
\end{gather}
and noise standard deviation
\begin{gather}
\label{eq: Sigma}
\sigma=\frac{1}{40\sqrt{mn}L_{F}^{H}}\alpha\left(\frac{1}{\rho}\right)^{3},
\end{gather}
the sequence $\{\hat{\mathbf{x}}^{k}\}$ generated by \textbf{NDGD} satisfies
\begin{multline*}
\mathbb{P}\Big[\|\hat{\mathbf{x}}^{k}-\bm{1}_{m}\otimes\operatorname{av}(\hat{\mathbf{x}}^{k})\|\\
\le(\lambda_{2}+\alpha L_{F}^{g})^{k}\|\hat{\mathbf{x}}^{0}-\bm{1}_{m}\otimes\operatorname{av}(\hat{\mathbf{x}}^{0})\|+\zeta\Big]
\ge1-k\mathrm{e}^{-\rho},
\end{multline*}
with $\operatorname{av}(\cdot)$ as per~\eqref{eq: Average of xhat}, where
\begin{gather}
\label{eq: Consensus error}
\zeta=\frac{\alpha D+\alpha\sigma(\sqrt{2\rho}+\sqrt{mn})}{1-\lambda_{2}-\alpha L_{F}^{g}},
\end{gather}
and $0<\lambda_{2}<1$ denotes the second-largest magnitude eigenvalue of $\mathbf{W}$; i.e.,
\begin{gather}
\label{eq: lambda_2}
\lambda_{2}=\max\{|\lambda|:\lambda\neq1~\land~(\exists\mathbf{v})~\mathbf{W}\mathbf{v}=\lambda\mathbf{v}\}.
\end{gather}
\end{lemma}

Building on Lemma~\ref{lem: Consensus error}, Lemma~\ref{lem: Approximate second order} establishes the attainment of an approximate second-order stationary point of $Q_{\alpha}$ in~\eqref{eq: Q} under small consensus errors. The complete proof is provided in Section~\ref{sec: Proof of Lemma Approximate second order}.

\begin{lemma}
\label{lem: Approximate second order}
Let Assumptions~\ref{ass: Lipschitz}, \ref{ass: Coercivity}, \ref{ass: Bounded gradient disagreement}, \ref{ass: Network}, \ref{ass: Mixing matrix}, and \ref{ass: Random perturbation} hold with corresponding constants $L_{f_i}^{g},L_{f_i}^{H},D>0$. Further, let $f_i^{\star}$ denote the global minimum of $f_i$ for $i\in\mathcal{V}$, and let $\operatorname{av}(\cdot)$ be the averaging operator in~\eqref{eq: Average of xhat}. Then, there exists $\rho_{2}\ge1$ such that for $\rho\ge\rho_{2}$, with $\alpha$ as per~\eqref{eq: Alpha}, $\sigma$ as per~\eqref{eq: Sigma}, $K$ as per~\eqref{eq: Iteration}, and $\zeta$ as per~\eqref{eq: Consensus error}, after
\begin{gather}
\label{eq: Iteration}
K=\left\lceil\left(Q_{\alpha}(\hat{\mathbf{x}}^{0})-\sum_{i=1}^{m}f_i^{\star}\right)\alpha^{-4}\rho^{5}\right\rceil
\end{gather}
iterations of~\eqref{eq: NDGD} from any $\hat{\mathbf{x}}^{0}$ satisfying $\hat{\mathbf{x}}^{0}=\bm{1}_{m}\otimes\operatorname{av}(\hat{\mathbf{x}}^{0})$,
\begin{multline}
\label{eq: Approximate second order}
\mathbb{P}\Big[\exists k\in(0,K],~\|\nabla Q_{\alpha}(\hat{\mathbf{x}}^{k})\|\le\sqrt{\alpha}\\
\land~\lambda_{\min}(\nabla^{2}Q_{\alpha}(\hat{\mathbf{x}}^{k}))\ge-\sqrt{L_{F}^{H}\sqrt{\alpha}}\\
\land~\|\hat{\mathbf{x}}^{k}-\bm{1}_{m}\otimes\operatorname{av}(\hat{\mathbf{x}}^{k})\|\le\zeta\Big]
\ge1-4K^{2}\mathrm{e}^{-\rho},
\end{multline}
where $0<\lambda_{2}<1$ denotes the second-largest magnitude eigenvalue of $\mathbf{W}$ as per~\eqref{eq: lambda_2}.
\end{lemma}

\begin{remark}
$Q_{\alpha}(\hat{\mathbf{x}}^{0})$ in $K$ does not depend on $\alpha$ given $\hat{\mathbf{x}}^{0}=\bm{1}_{m}\otimes\operatorname{av}(\hat{\mathbf{x}}^{0})$.
\end{remark}

Next, Theorem~\ref{the: Approximate second order}, which follows from Lemma~\ref{lem: Approximate second order}, guarantees that an approximate second-order stationary point of $F$ in~\eqref{eq: Constrained optimization problem} is attained under the same conditions. Theorem~\ref{the: Approximate second order} establishes high-probability convergence to a neighborhood of a consensual approximate second-order stationary point, characterizing the saddle-point escape behavior of the proposed \textbf{NDGD} algorithm. Building on this result, Theorem~\ref{the: Second order guarantee} further establishes convergence guarantees to a local minimizer under the additional local regularity condition introduced in Assumption~\ref{ass: Local convexity}.

\begin{theorem}
\label{the: Approximate second order}
Let Assumptions~\ref{ass: Lipschitz}, \ref{ass: Coercivity}, \ref{ass: Bounded gradient disagreement}, \ref{ass: Network}, \ref{ass: Mixing matrix}, and \ref{ass: Random perturbation} hold with corresponding constants $L_{f_i}^{g},L_{f_i}^{H},D>0$. Further, let $f_i^{\star}$ denote the global minimum of $f_i$ for $i\in\mathcal{V}$, and let $\operatorname{av}(\cdot)$ be the averaging operator in~\eqref{eq: Average of xhat}. Given confidence parameter $0<p<1$, there exists $\rho_{3}\ge1$ such that for $\rho\ge\max\{-2\ln(p),\rho_{3}\}$, with $\alpha$ as per~\eqref{eq: Alpha}, $\sigma$ as per~\eqref{eq: Sigma}, $K$ as per~\eqref{eq: Iteration}, and $\zeta$ as per~\eqref{eq: Consensus error}, the following holds:
\begin{multline}
\label{eq: Approximate second order 2}
\mathbb{P}\Big[\exists k\in(0,K],~\left\|\sum_{i=1}^{m}\nabla f_i(\hat{\mathbf{x}}_{i}^{k})\right\|\le\sqrt{m\alpha}\\
\land~\lambda_{\min}\left(\sum_{i=1}^{m}\nabla^{2}f_i(\hat{\mathbf{x}}_{i}^{k})\right)\ge-m\sqrt{L_{F}^{H}\sqrt{\alpha}}\\
\land~\|\hat{\mathbf{x}}^{k}-\bm{1}_{m}\otimes\operatorname{av}(\hat{\mathbf{x}}^{k})\|\le\zeta\Big]
\ge1-p,
\end{multline}
where $\hat{\mathbf{x}}^{k}=[(\hat{\mathbf{x}}_{1}^{k})^{\top},\ldots,(\hat{\mathbf{x}}_{m}^{k})^{\top}]^{\top}$ according to~\eqref{eq: NDGD} for $k\in\mathbb{N}$, given any initial condition satisfying $\hat{\mathbf{x}}^{0}=\bm{1}_{m}\otimes\operatorname{av}(\hat{\mathbf{x}}^{0})$.
\end{theorem}

\begin{proof}
By Lemma~\ref{lem: Approximate second order}, there exists $\rho_{2}\ge1$ such that for $\rho\ge\rho_{2}$,
\begin{multline*}
\mathbb{P}\Big[\exists k\in(0,K],~\|\nabla Q_{\alpha}(\hat{\mathbf{x}}^{k})\|\le\sqrt{\alpha}\\
\land~\lambda_{\min}(\nabla^{2}Q_{\alpha}(\hat{\mathbf{x}}^{k}))\ge-\sqrt{L_{F}^{H}\sqrt{\alpha}}\\
\land~\|\hat{\mathbf{x}}^{k}-\bm{1}_{m}\otimes\operatorname{av}(\hat{\mathbf{x}}^{k})\|\le\zeta\Big]
\ge1-4K^{2}\mathrm{e}^{-\rho}.
\end{multline*}
Further, given $\hat{\mathbf{x}}_{i}\in\mathbb{R}^{n}$,
\begin{align}
\label{eq: Upper bound for sum of f}
\begin{aligned}
\left\|\sum_{i=1}^{m}\nabla f_i(\hat{\mathbf{x}}_{i})\right\|
&=\|(\bm{1}_{m}\otimes\mathbf{I}_{n})^{\top}\nabla F(\hat{\mathbf{x}})\|\\
&=\|(\bm{1}_{m}\otimes\mathbf{I}_{n})^{\top}\nabla Q_{\alpha}(\hat{\mathbf{x}})\|\\
&\le\sqrt{m}\|\nabla Q_{\alpha}(\hat{\mathbf{x}})\|,
\end{aligned}
\end{align}
and by Lemma~\ref{lem: Minimum eigenvalue of Q},
\begin{gather}
\label{eq: Lower bound for Hessian}
\lambda_{\min}\left(\sum_{i=1}^{m}\nabla^{2}f_i(\hat{\mathbf{x}}_{i})\right)
\ge m\,\lambda_{\min}(\nabla^{2}Q_{\alpha}(\hat{\mathbf{x}})).
\end{gather}
Thus, there exists $\rho_{3}\ge\rho_{2}$ such that for $\rho\ge\rho_{3}$, $4K^{2}\mathrm{e}^{-\rho}\le\mathrm{e}^{-\rho/2}$. Finally,~\eqref{eq: Approximate second order 2} follows from~\eqref{eq: Approximate second order}, \eqref{eq: Upper bound for sum of f}, and~\eqref{eq: Lower bound for Hessian}.
\end{proof}

\begin{remark}
\label{rem: Theorem 1}
In view of Theorem~\ref{the: Approximate second order} and Definition~\ref{def: Approximately consensual second order stationary point}, \textbf{NDGD} according to~\eqref{eq: NDGD} eventually produces an $(\eta,\epsilon,\gamma)$-consensual second-order stationary point of $F$ with probability $1-p\in(0,1)$. In particular,
\begin{gather*}
\eta=\zeta,\qquad \epsilon=\sqrt{m\alpha},\qquad \gamma=m\sqrt{L_{F}^{H}\sqrt{\alpha}},
\end{gather*}
which can be made arbitrarily small by choosing sufficiently large $\rho>1$, and thus small $\alpha$ and $\sigma$ as per~\eqref{eq: Alpha} and~\eqref{eq: Sigma}. To increase this confidence, i.e., make $p$ smaller, one must reduce the step-size $\alpha$ by increasing $\rho$, which in turn leads to a larger number of iterations $K$. This trade-off implies that achieving higher confidence requires smaller steps, potentially slowing convergence.
\end{remark}

By the following lemma, once the iterates of~\eqref{eq: NDGD} reach a neighborhood of local minimizers of $Q_{\alpha}$, they remain within this neighborhood with high probability provided $\rho$ is sufficiently large. The complete proof is provided in Section~\ref{sec: Proof of Lemma Stay close to minimizers}.

\begin{lemma}
\label{lem: Stay close to minimizers}
Let Assumptions~\ref{ass: Lipschitz}, \ref{ass: Coercivity}, \ref{ass: Random perturbation}, and \ref{ass: Local convexity} hold with corresponding constants $L_{F}^{g},L_{F}^{H},\alpha^{\prime},\delta,\eta,\mu>0$. Further, let $f_i^{\star}$ denote the global minimum of $f_i$ for $i\in\mathcal{V}$, and let $\operatorname{av}(\cdot)$ be the averaging operator in~\eqref{eq: Average of xhat}. Then, there exists $\rho_{4}\ge1$ such that for $\rho\ge\rho_{4}$, with $\alpha$ as per~\eqref{eq: Alpha}, $\sigma$ as per~\eqref{eq: Sigma}, and $K$ as per~\eqref{eq: Iteration}, the following hold: i) $0<\alpha\le\alpha^{\prime}$; ii) $\delta\ge\sqrt{\alpha}/\mu$; iii) there exists $\hat{\mathbf{x}}^{\star}\in\mathcal{X}_{Q_{\alpha}}^{\star}$ such that
\begin{gather}
\label{eq: Stay close to minimizers}
\mathbb{P}\left[\forall k\in(0,K],~\|\hat{\mathbf{x}}^{k}-\hat{\mathbf{x}}^{\star}\|\le\frac{2}{\mu}\sqrt{\alpha}\right]
\ge1-K(K+1)\mathrm{e}^{-\rho},
\end{gather}
where $\hat{\mathbf{x}}^{k}=[(\hat{\mathbf{x}}_{1}^{k})^{\top},\ldots,(\hat{\mathbf{x}}_{m}^{k})^{\top}]^{\top}$ according to~\eqref{eq: NDGD} for $k\in\mathbb{N}$, given any initial condition satisfying
\begin{gather*}
\|\hat{\mathbf{x}}^{0}-\bm{1}_{m}\otimes\operatorname{av}(\hat{\mathbf{x}}^{0})\|\le\eta,
\end{gather*}
and
\begin{gather*}
\operatorname{dist}(\hat{\mathbf{x}}^{0},\mathcal{X}_{Q_{\alpha}}^{\star})\le\frac{\sqrt{\alpha}}{\mu}.
\end{gather*}
\end{lemma}

\begin{remark}
As $\rho\rightarrow\infty$, $K(K+1)\mathrm{e}^{-\rho}\rightarrow0$, increasing the confidence that all iterates remain within the stated neighborhood of the minimizer.
\end{remark}

As the main result, it is established that under suitable regularity conditions, and for sufficiently small step-size $\alpha$ and noise variance $\sigma^{2}$, the iterates of all agents converge with high probability to a neighborhood of a common local minimizer of $f$, with explicit bounds on both the optimality gap and the consensus error.

\begin{theorem}
\label{the: Second order guarantee}
Let Assumptions~\ref{ass: Strict saddle}, \ref{ass: Lipschitz}, \ref{ass: Coercivity}, \ref{ass: Bounded gradient disagreement}, \ref{ass: Network}, \ref{ass: Mixing matrix}, \ref{ass: Random perturbation}, and \ref{ass: Local convexity} hold with corresponding constants $L_{f_i}^{g},L_{f_i}^{H},D,\mu>0$. Further, let $f_i^{\star}$ denote the global minimum of $f_i$ for $i\in\mathcal{V}$, and let $\operatorname{av}(\cdot)$ be the averaging operator in~\eqref{eq: Average of xhat}. Given confidence parameter $0<p<1$, there exists $\rho_{0}\ge1$ such that for $\rho\ge\rho_{0}$, with $\alpha$ as per~\eqref{eq: Alpha}, $\sigma$ as per~\eqref{eq: Sigma}, $K$ as per~\eqref{eq: Iteration}, and $\zeta$ as per~\eqref{eq: Consensus error}, the following holds:
\begin{multline}
\label{eq: Theorem 2}
\mathbb{P}\Big[\forall i\in\mathcal{V},~\operatorname{dist}(\hat{\mathbf{x}}_{i}^{K},\mathcal{X}_{f}^{\star})
\le\frac{7\mu+5\sqrt{m}L_{F}^{g}}{\mu^{2}}\sqrt{\alpha}\\
\land~\|\hat{\mathbf{x}}_{i}^{K}-\operatorname{av}(\hat{\mathbf{x}}^{K})\|\le\zeta\Big]
\ge1-p,
\end{multline}
where $\hat{\mathbf{x}}_{i}^{K}$ is the $K$-th iteration of~\eqref{eq: NDGD} given any initial condition $\hat{\mathbf{x}}^{0}=\bm{1}_{m}\otimes\operatorname{av}(\hat{\mathbf{x}}^{0})$.
\end{theorem}

\begin{remark}
With $\alpha$ as per~\eqref{eq: Alpha}, $\sigma$ as per~\eqref{eq: Sigma}, and $\zeta$ in~\eqref{eq: Consensus error}, the bounds in~\eqref{eq: Theorem 2} are monotonically decreasing in $\rho\ge1$.
\end{remark}

\begin{remark}
The bounds in Theorem~\ref{the: Second order guarantee} depend on the network size $m$, problem dimension $n$, and the spectral gap $1-\lambda_{2}$. In particular, $\zeta$ in~\eqref{eq: Consensus error} scales with $\sqrt{mn}$ and inversely with $1-\lambda_{2}$. For poorly connected networks, the consensus error may become larger and can be partially mitigated by choosing a smaller step-size $\alpha$, at the cost of increased iteration complexity $K$. In contrast, better-connected networks with larger spectral gaps lead to improved consensus behavior and smaller disagreement among agents.
\end{remark}

\begin{remark}
Theorem~\ref{the: Second order guarantee} provides, with probability at least $1-p$, an agent-wise $O(\sqrt{\alpha})$ bound to the local minimizer set of $f$ together with the explicit consensus bound $\zeta$, after $K=O(\alpha^{-4}\rho^{5})=O(\rho^{7})$ iterations under~\eqref{eq: Alpha}. First-order methods based only on random initialization can avoid strict saddles almost surely under suitable conditions~\cite{daneshmand2020second,swenson2019distributed,swenson2022distributed}, but do not by themselves provide the same finite-time escape guarantee; unperturbed gradient descent can require exponential time to escape saddle points in the worst case~\cite{du2017gradient}. Second-order methods can exploit negative curvature directly but require Hessian information. Perturbation-based centralized first-order methods avoid explicit Hessian computation and achieve second-order stationarity with polynomial or polylogarithmic dependence on problem dimension~\cite{jin2017escape,jin2021nonconvex}. Similarly, the present bound has no exponential dependence on the effective distributed problem dimension, while additionally accounting for network coupling and consensus error. Compared with distributed results,~\cite{daneshmand2020second} establishes neighborhood second-order stationarity for \textbf{DGD}, while~\cite{vlaski2021distributed2} establishes strict-saddle escape in $O(1/\mu_{s})$ iterations, where $\mu_{s}$ denotes its step-size, and polynomial-time approximate second-order stationarity. The distinction of Theorem~\ref{the: Second order guarantee} is that it directly provides a finite-time agent-wise proximity bound to the local minimizer set of the original objective $f$, together with an explicit consensus-error bound. Since the assumptions and convergence criteria differ across distributed methods, their iteration complexities are not directly ordered.
\end{remark}

\begin{proof}
By Assumption~\ref{ass: Local convexity}, let $\alpha^{\prime},\epsilon,\gamma,\eta,\delta,\mu>0$ be the corresponding constants. For any $0<\alpha\le\alpha^{\prime}$ and $K\ge0$, define
\begin{align*}
\mathcal{A}_{1}(K):=\left\{\begin{aligned}
&\exists k\in(0,K],~\|\nabla Q_{\alpha}(\hat{\mathbf{x}}^{k})\|\le\sqrt{\alpha},\\
&\|\hat{\mathbf{x}}^{k}-\bm{1}_{m}\otimes\operatorname{av}(\hat{\mathbf{x}}^{k})\|\le\eta,\quad \hat{\mathbf{x}}^{k}\in\mathcal{B}_{\alpha}
\end{aligned}\right\},
\end{align*}
where
\begin{align*}
\mathcal{B}_{\alpha}:=\left\{\begin{aligned}
\hat{\mathbf{x}}:\ &\exists\hat{\mathbf{x}}^{\prime}\in(\mathbb{R}^{n})^{m},~\|\hat{\mathbf{x}}^{\prime}-\hat{\mathbf{x}}\|\le\delta,\quad \nabla Q_{\alpha}(\hat{\mathbf{x}}^{\prime})=\bm{0},\\
&\lambda_{\min}(\nabla^{2}F(\hat{\mathbf{y}}))\ge\mu~\text{for all }\hat{\mathbf{y}}\text{ with }\|\hat{\mathbf{y}}-\hat{\mathbf{x}}^{\prime}\|\le3\delta
\end{aligned}\right\},
\end{align*}
and
\begin{align*}
\mathcal{A}_{2}(K):=\left\{\begin{aligned}
&\exists k\in(0,K],~\|\nabla Q_{\alpha}(\hat{\mathbf{x}}^{k})\|\le\sqrt{\alpha},\\
&\left\|\sum_{i=1}^{m}\nabla f_i(\hat{\mathbf{x}}_{i}^{k})\right\|\le\epsilon,\quad
\lambda_{\min}\left(\sum_{i=1}^{m}\nabla^{2}f_i(\hat{\mathbf{x}}_{i}^{k})\right)\ge-\gamma,\\
&\|\hat{\mathbf{x}}^{k}-\bm{1}_{m}\otimes\operatorname{av}(\hat{\mathbf{x}}^{k})\|\le\eta
\end{aligned}\right\}.
\end{align*}
By Assumption~\ref{ass: Local convexity}, $\mathcal{A}_{2}(K)\subseteq\mathcal{A}_{1}(K)$. Moreover, with $\alpha$ depending on $\rho$ as in~\eqref{eq: Alpha}, there exists $\rho_{5}>0$ such that for any $\rho\ge\rho_{5}$,
\begin{align*}
\mathcal{A}_{3}(K):=\left\{\begin{aligned}
&\exists k\in(0,K],~\|\nabla Q_{\alpha}(\hat{\mathbf{x}}^{k})\|\le\sqrt{\alpha},\\
&\left\|\sum_{i=1}^{m}\nabla f_i(\hat{\mathbf{x}}_{i}^{k})\right\|\le\sqrt{m\alpha},\\
&\lambda_{\min}\left(\sum_{i=1}^{m}\nabla^{2}f_i(\hat{\mathbf{x}}_{i}^{k})\right)\ge-m\sqrt{L_{F}^{H}\sqrt{\alpha}},\\
&\|\hat{\mathbf{x}}^{k}-\bm{1}_{m}\otimes\operatorname{av}(\hat{\mathbf{x}}^{k})\|\le\eta
\end{aligned}\right\}
\end{align*}
satisfies $\mathcal{A}_{3}(K)\subseteq\mathcal{A}_{2}(K)$. By~\eqref{eq: Upper bound for sum of f} and~\eqref{eq: Lower bound for Hessian},
\begin{align*}
\mathcal{A}_{4}(K):=\left\{\begin{aligned}
&\exists k\in(0,K],~\|\nabla Q_{\alpha}(\hat{\mathbf{x}}^{k})\|\le\sqrt{\alpha},\\
&\lambda_{\min}(\nabla^{2}Q_{\alpha}(\hat{\mathbf{x}}^{k}))\ge-\sqrt{L_{F}^{H}\sqrt{\alpha}},\\
&\|\hat{\mathbf{x}}^{k}-\bm{1}_{m}\otimes\operatorname{av}(\hat{\mathbf{x}}^{k})\|\le\eta
\end{aligned}\right\}
\end{align*}
satisfies $\mathcal{A}_{4}(K)\subseteq\mathcal{A}_{3}(K)$. Therefore,
\begin{gather*}
\mathcal{A}_{4}(K)\subseteq\mathcal{A}_{3}(K)\subseteq\mathcal{A}_{2}(K)\subseteq\mathcal{A}_{1}(K).
\end{gather*}

By Lemma~\ref{lem: Approximate second order}, there exists $\rho_{2}\ge1$ such that for any $\rho\ge\max\{\rho_{2},\rho_{5}\}$, with $\sigma$ as per~\eqref{eq: Sigma} and $K$ as per~\eqref{eq: Iteration},
\begin{gather*}
\mathbb{P}[\mathcal{A}_{1}(K)]
\ge\mathbb{P}[\mathcal{A}_{2}(K)]
\ge\mathbb{P}[\mathcal{A}_{3}(K)]
\ge\mathbb{P}[\mathcal{A}_{4}(K)]
\ge1-4K^{2}\mathrm{e}^{-\rho}.
\end{gather*}

Let
\begin{align*}
\mathcal{A}_{0}(K):=\left\{\begin{aligned}
&\exists k\in(0,K],~\operatorname{dist}(\hat{\mathbf{x}}^{k},\mathcal{X}_{Q_{\alpha}}^{\star})\le\frac{\sqrt{\alpha}}{\mu},\\
&\|\hat{\mathbf{x}}^{k}-\bm{1}_{m}\otimes\operatorname{av}(\hat{\mathbf{x}}^{k})\|\le\eta
\end{aligned}\right\}.
\end{align*}
For any $k\ge0$ such that $\hat{\mathbf{x}}^{k}\in\mathcal{B}_{\alpha}$, by the definition of $\mathcal{B}_{\alpha}$ there exists $\hat{\mathbf{x}}^{\prime}\in(\mathbb{R}^{n})^{m}$ such that $\|\hat{\mathbf{x}}^{\prime}-\hat{\mathbf{x}}^{k}\|\le\delta$, $\nabla Q_{\alpha}(\hat{\mathbf{x}}^{\prime})=\bm{0}$, and $\lambda_{\min}(\nabla^{2}F(\hat{\mathbf{y}}))\ge\mu$ for all $\hat{\mathbf{y}}$ with $\|\hat{\mathbf{y}}-\hat{\mathbf{x}}^{\prime}\|\le3\delta$. By Lemma~\ref{lem: Minimum eigenvalue of Q},
\begin{gather*}
\lambda_{\min}(\nabla^{2}Q_{\alpha}(\hat{\mathbf{y}}))
\ge\lambda_{\min}(\nabla^{2}F(\hat{\mathbf{y}}))
\ge\mu
\end{gather*}
in this neighborhood. This implies $\hat{\mathbf{x}}^{\prime}\in\mathcal{X}_{Q_{\alpha}}^{\star}$ and lies within distance $\delta$ of $\hat{\mathbf{x}}^{k}$. Moreover, if $\|\nabla Q_{\alpha}(\hat{\mathbf{x}}^{k})\|\le\sqrt{\alpha}$ also holds, then by local strong convexity,
\begin{gather*}
\operatorname{dist}(\hat{\mathbf{x}}^{k},\mathcal{X}_{Q_{\alpha}}^{\star})
\le\|\hat{\mathbf{x}}^{k}-\hat{\mathbf{x}}^{\prime}\|
\le\frac{\|\nabla Q_{\alpha}(\hat{\mathbf{x}}^{k})-\nabla Q_{\alpha}(\hat{\mathbf{x}}^{\prime})\|}{\mu}
\le\frac{\sqrt{\alpha}}{\mu}.
\end{gather*}
Thus, $\mathcal{A}_{1}(K)\subseteq\mathcal{A}_{0}(K)$, and
\begin{gather}
\label{eq: First time close}
\mathbb{P}[\mathcal{A}_{0}(K)]
\ge\mathbb{P}[\mathcal{A}_{1}(K)]
\ge1-4K^{2}\mathrm{e}^{-\rho}.
\end{gather}

By Lemma~\ref{lem: Stay close to minimizers} and the time-invariance of the update~\eqref{eq: NDGD}, for any $\rho\ge\max\{\rho_{2},\rho_{4},\rho_{5}\}$,
\begin{gather*}
\mathbb{P}\left[
\operatorname{dist}(\hat{\mathbf{x}}^{K},\mathcal{X}_{Q_{\alpha}}^{\star})
\le\frac{2\sqrt{\alpha}}{\mu}
~\middle|~
\mathcal{A}_{0}(K)
\right]
\ge1-K(K+1)\mathrm{e}^{-\rho},
\end{gather*}
because conditioning on $\mathcal{A}_{0}(K)$ ensures there is an iterate before $K$ that satisfies the corresponding initial requirement. By total probability and~\eqref{eq: First time close},
\begin{multline}
\label{eq: Dist prob}
\mathbb{P}\left[
\operatorname{dist}(\hat{\mathbf{x}}^{K},\mathcal{X}_{Q_{\alpha}}^{\star})
\le\frac{2\sqrt{\alpha}}{\mu}
\right]\\
\ge
\mathbb{P}\left[
\operatorname{dist}(\hat{\mathbf{x}}^{K},\mathcal{X}_{Q_{\alpha}}^{\star})
\le\frac{2\sqrt{\alpha}}{\mu}
~\middle|~
\mathcal{A}_{0}(K)
\right]\mathbb{P}[\mathcal{A}_{0}(K)]\\
\ge(1-K(K+1)\mathrm{e}^{-\rho})(1-4K^{2}\mathrm{e}^{-\rho})\\
\ge1-(5K^{2}+K)\mathrm{e}^{-\rho}+4K^{3}(K+1)\mathrm{e}^{-2\rho}.
\end{multline}

By Lemma~\ref{lem: Consensus error} with $k=K$ and $\|\hat{\mathbf{x}}^{0}-\bm{1}_{m}\otimes\operatorname{av}(\hat{\mathbf{x}}^{0})\|=0$,
\begin{gather}
\label{eq: Consensus error at K}
\mathbb{P}\left[
\|\hat{\mathbf{x}}^{K}-\bm{1}_{m}\otimes\operatorname{av}(\hat{\mathbf{x}}^{K})\|\le\zeta
\right]
\ge1-K\mathrm{e}^{-\rho}.
\end{gather}

Applying the Boole-Fréchet inequality in Proposition~\ref{pro: Boole-frechet Inequality}, for any $\zeta>0$,
\begin{multline*}
\mathbb{P}\Big[
\operatorname{dist}(\hat{\mathbf{x}}^{K},\mathcal{X}_{Q_{\alpha}}^{\star})\le\frac{2\sqrt{\alpha}}{\mu}\\
\land~
\|\hat{\mathbf{x}}^{K}-\bm{1}_{m}\otimes\operatorname{av}(\hat{\mathbf{x}}^{K})\|\le\zeta
\Big]\\
\ge
\mathbb{P}\left[
\|\hat{\mathbf{x}}^{K}-\bm{1}_{m}\otimes\operatorname{av}(\hat{\mathbf{x}}^{K})\|\le\zeta
\right]\\
-\left(
1-
\mathbb{P}\left[
\operatorname{dist}(\hat{\mathbf{x}}^{K},\mathcal{X}_{Q_{\alpha}}^{\star})\le\frac{2\sqrt{\alpha}}{\mu}
\right]
\right).
\end{multline*}

Then, with $\zeta$ per~\eqref{eq: Consensus error}, there exists $\rho_{6}\ge1$ such that $\zeta\le\sqrt{\alpha}/\mu$ for every $\rho\ge\rho_{6}$. In view of~\eqref{eq: Dist prob} and~\eqref{eq: Consensus error at K}, for any $\rho\ge\max\{\rho_{1},\rho_{2},\rho_{4},\rho_{5},\rho_{6}\}$,
\begin{multline*}
\mathbb{P}\Big[
\operatorname{dist}(\hat{\mathbf{x}}^{K},\mathcal{X}_{Q_{\alpha}}^{\star})\le\frac{2\sqrt{\alpha}}{\mu}\\
\land~
\|\hat{\mathbf{x}}^{K}-\bm{1}_{m}\otimes\operatorname{av}(\hat{\mathbf{x}}^{K})\|\le\zeta
\Big]\\
\ge1-(5K^{2}+2K)\mathrm{e}^{-\rho}+4K^{3}(K+1)\mathrm{e}^{-2\rho}.
\end{multline*}

Given $K\ge0$, if $\hat{\mathbf{x}}^{K}$ satisfies the conditions that there exists $\hat{\mathbf{x}}^{\star}\in\mathcal{X}_{Q_{\alpha}}^{\star}$, $\|\hat{\mathbf{x}}^{\star}-\hat{\mathbf{x}}^{K}\|\le2\sqrt{\alpha}/\mu$, and $\|\hat{\mathbf{x}}^{K}-\bm{1}_{m}\otimes\operatorname{av}(\hat{\mathbf{x}}^{K})\|\le\sqrt{\alpha}/\mu$, then
\begin{align*}
\|\hat{\mathbf{x}}^{\star}-\bm{1}_{m}\otimes\operatorname{av}(\hat{\mathbf{x}}^{\star})\|
&\le\|\hat{\mathbf{x}}^{\star}-\hat{\mathbf{x}}^{K}\|+\|\hat{\mathbf{x}}^{K}-\bm{1}_{m}\otimes\operatorname{av}(\hat{\mathbf{x}}^{K})\|\\
&\quad+\|\bm{1}_{m}\otimes\operatorname{av}(\hat{\mathbf{x}}^{K})-\bm{1}_{m}\otimes\operatorname{av}(\hat{\mathbf{x}}^{\star})\|\\
&\le2\|\hat{\mathbf{x}}^{\star}-\hat{\mathbf{x}}^{K}\|+\|\hat{\mathbf{x}}^{K}-\bm{1}_{m}\otimes\operatorname{av}(\hat{\mathbf{x}}^{K})\|\\
&\le\frac{5}{\mu}\sqrt{\alpha}.
\end{align*}

Therefore, define
\begin{align*}
\mathcal{A}_{5}(K):=\left\{\begin{aligned}
&\exists\hat{\mathbf{x}}^{\star}\in\mathcal{X}_{Q_{\alpha}}^{\star},~
\|\hat{\mathbf{x}}^{K}-\hat{\mathbf{x}}^{\star}\|\le\frac{2\sqrt{\alpha}}{\mu},\\
&\|\hat{\mathbf{x}}^{K}-\bm{1}_{m}\otimes\operatorname{av}(\hat{\mathbf{x}}^{K})\|\le\zeta,\\
&\|\hat{\mathbf{x}}^{\star}-\bm{1}_{m}\otimes\operatorname{av}(\hat{\mathbf{x}}^{\star})\|
\le\frac{5}{\mu}\sqrt{\alpha}
\end{aligned}\right\}.
\end{align*}
Then,
\begin{gather*}
\mathbb{P}[\mathcal{A}_{5}(K)]
\ge1-(5K^{2}+2K)\mathrm{e}^{-\rho}+4K^{3}(K+1)\mathrm{e}^{-2\rho}.
\end{gather*}

By Lemma~\ref{lem: Local minimizers of Q}, for all $i\in\mathcal{V}$ and any $\hat{\mathbf{x}}^{\star}\in\mathcal{X}_{Q_{\alpha}}^{\star}$ satisfying the conditions in $\mathcal{A}_{5}(K)$,
\begin{multline*}
\operatorname{dist}(\hat{\mathbf{x}}_{i}^{K},\mathcal{X}_{f}^{\star})
\le\|\hat{\mathbf{x}}_{i}^{K}-\hat{\mathbf{x}}_{i}^{\star}\|
+\operatorname{dist}(\hat{\mathbf{x}}_{i}^{\star},\mathcal{X}_{f}^{\star})\\
\le\|\hat{\mathbf{x}}^{K}-\hat{\mathbf{x}}^{\star}\|
+\operatorname{dist}(\hat{\mathbf{x}}_{i}^{\star},\mathcal{X}_{f}^{\star})
\le\left(7+\frac{5\sqrt{m}L_{F}^{g}}{\mu}\right)\frac{\sqrt{\alpha}}{\mu}.
\end{multline*}

Define
\begin{align*}
\mathcal{A}_{6}(K):=\left\{\begin{aligned}
&\forall i\in\mathcal{V},~
\operatorname{dist}(\hat{\mathbf{x}}_{i}^{K},\mathcal{X}_{f}^{\star})
\le\left(7+\frac{5\sqrt{m}L_{F}^{g}}{\mu}\right)\frac{\sqrt{\alpha}}{\mu},\\
&\|\hat{\mathbf{x}}_{i}^{K}-\operatorname{av}(\hat{\mathbf{x}}^{K})\|\le\zeta
\end{aligned}\right\}.
\end{align*}
Since $\mathcal{A}_{5}(K)\subseteq\mathcal{A}_{6}(K)$, there exists $\rho_{7}\ge\max\{\rho_{1},\rho_{2},\rho_{4},\rho_{5},\rho_{6}\}$ such that for any $\rho\ge\rho_{7}$,
\begin{multline}
\label{eq: Neighbourhood of local minimizers}
\mathbb{P}[\mathcal{A}_{6}(K)]\\
\ge1-(5K^{2}+2K)\mathrm{e}^{-\rho}+4K^{3}(K+1)\mathrm{e}^{-2\rho}
\ge1-\mathrm{e}^{-\rho/2}.
\end{multline}
Choosing $\rho_{0}\ge\max\{-2\ln(p),\rho_{7}\}$,~\eqref{eq: Theorem 2} follows from~\eqref{eq: Neighbourhood of local minimizers}.
\end{proof}

\section{Proofs}
\label{sec: Proofs}
The proof section is organized into several subsections leading to the main results. We first prove Lemma~\ref{lem: Local minimizers of Q}, which relates local minimizers of the lifted function $Q_{\alpha}$ to those of the original objective $f$, followed by Lemma~\ref{lem: Consensus error}. We then establish Lemma~\ref{lem: Approximate second order} using the reformulated \textbf{NDGD} dynamics in~\eqref{eq: NDGD2}, and finally prove Lemma~\ref{lem: Stay close to minimizers}.

\subsection{Proof of Lemma \ref{lem: Local minimizers of Q}}
\label{sec: Proof of Lemma Local minimizers of Q}
The proof of Lemma \ref{lem: Local minimizers of Q} relies on a collection of intermediate results. Lemma~\ref{lem: first order guarantee for xhat*} shows that if a second-order stationary point of $Q_{\alpha}$ is almost in consensus then the size of the gradient of $f$ at the average is bounded by $\alpha^{1/2}$. Lemma~\ref{lem: second order guarantee for xhat*} shows that if a second-order stationary point of $Q_{\alpha}$ is almost in consensus then the Hessian of $f$ is lower bounded by $\alpha^{1/2}$. Combining Lemma~\ref{lem: first order guarantee for xhat*} and Lemma~\ref{lem: second order guarantee for xhat*} with another technical result demonstrates that, for each agent, the corresponding component $\hat{\mathbf{x}}_{i}^{\star}$ of such a second-order stationary point of $Q_{\alpha}$ can be made arbitrarily close to a local minimizer of $f$ by choosing a sufficiently small step-size $\alpha > 0$. These results collectively indicate that achieving a local minimizer for $Q_{\alpha}$ effectively brings each agent very close to a local minimizer of the function $f$.

\begin{lemma}
    \label{lem: first order guarantee for xhat*}
    Let Assumption~\ref{ass: Lipschitz} hold. Given 
    $\alpha,\mu > 0$, if $\hat{\mathbf{x}}^{\star} \in \{\hat{\mathbf{x}} \in\mathcal{X}_{Q_{\alpha}}^{\star}:~\|\hat{\mathbf{x}} - \bm{1}_{m} \otimes \operatorname{av}(\hat{\mathbf{x}})\| \le 5 \sqrt{\alpha}/\mu\}$, 
    then
    \begin{align*}
        \| \nabla f(\operatorname{av}(\hat{\mathbf{x}}^{\star})) \| \le \sqrt{m}L_{F}^{g} \frac{5}{\mu}\sqrt{\alpha}.
    \end{align*}
\end{lemma}

\begin{proof}
    Since $(\bm{1}_{m} \otimes \mathbf{I}_{n})^{\top} \nabla F(\hat{\mathbf{x}}^{\star}) = (\bm{1}_{m} \otimes \mathbf{I}_{n})^{\top} (\nabla Q_{\alpha}(\hat{\mathbf{x}}^{\star}) - \alpha^{-1}(\mathbf{I}_{mn}-\hat{\mathbf{W}})\hat{\mathbf{x}}^{\star}) = 0$ and $\nabla f(\operatorname{av}(\hat{\mathbf{x}}^{\star})) = (\bm{1}_{m} \otimes \mathbf{I}_{n})^{\top} \nabla F(\bm{1}_{m} \otimes \operatorname{av}(\hat{\mathbf{x}}^{\star}))$, it follows under Assumption \ref{ass: Lipschitz} that 
    \begin{align*}
        \| \nabla f(\operatorname{av}(\hat{\mathbf{x}}^{\star})) \| &= \|(\bm{1}_{m} \otimes \mathbf{I}_{n})^{\top} (\nabla F (\bm{1}_{m} \otimes \operatorname{av}(\hat{\mathbf{x}}^{\star})) - \nabla F (\hat{\mathbf{x}}^{\star}))\|\\
        & \le \sqrt{m} L_{F}^{g} \| \bm{1}_{m} \otimes \operatorname{av}(\hat{\mathbf{x}}^{\star}) - \hat{\mathbf{x}}^{\star} \| \le  \sqrt{m}L_{F}^{g} \frac{5}{\mu}\sqrt{\alpha}
    \end{align*}
    since $\|\bm{1}_{m} \otimes \mathbf{I}_{n}\| = \sqrt{m}$ as claimed.
\end{proof}

\begin{lemma}
    \label{lem: second order guarantee for xhat*}
    Let Assumption~\ref{ass: Lipschitz} hold. Given $\alpha,\mu > 0$, if $\hat{\mathbf{x}}^{\star} \in \{\hat{\mathbf{x}} \in\mathcal{X}_{Q_{\alpha}}^{\star}:~\|\hat{\mathbf{x}} - \bm{1}_{m} \otimes \operatorname{av}(\hat{\mathbf{x}})\| \le 5\sqrt{\alpha}/\mu\}$, then
    \begin{align*}
        \lambda_{\min}(\nabla^{2} f (\operatorname{av}(\hat{\mathbf{x}}^{\star}))) \ge - m L_{F}^{H}\frac{5}{\mu} \sqrt{\alpha}.
    \end{align*}
\end{lemma}

\begin{proof}
    Since $(\bm{1}_{m} \otimes \mathbf{I}_{n})^{\top} \nabla^{2} F(\hat{\mathbf{x}}^{\star})  (\bm{1}_{m} \otimes \mathbf{I}_{n}) = (\bm{1}_{m} \otimes \mathbf{I}_{n})^{\top} (\nabla^{2} Q_{\alpha}(\hat{\mathbf{x}}^{\star}) - \frac{1}{\alpha}(\mathbf{I}_{mn}-\hat{\mathbf{W}})) (\bm{1}_{m} \otimes \mathbf{I}_{n}) = (\bm{1}_{m} \otimes \mathbf{I}_{n})^{\top} (\nabla^{2} Q_{\alpha}(\hat{\mathbf{x}}^{\star})) (\bm{1}_{m} \otimes \mathbf{I}_{n}) \succeq 0$, it follows under Assumption \ref{ass: Lipschitz} that with $\|\bm{1}_{m} \otimes \mathbf{I}_{n}\| = \sqrt{m}$,
    \begin{align*}
        \nabla^{2} f(\operatorname{av}(\hat{\mathbf{x}}^{\star})) &= (\bm{1}_{m} \otimes \mathbf{I}_{n})^{\top} \nabla^{2} F(\bm{1}_{m} \otimes \operatorname{av}(\hat{\mathbf{x}}^{\star})) (\bm{1}_{m} \otimes \mathbf{I}_{n})\\
        &\succeq (\bm{1}_{m} \otimes \mathbf{I}_{n})^{\top} ( \nabla^{2} F(\bm{1}_{m} \otimes \operatorname{av}(\hat{\mathbf{x}}^{\star}))\\
        &\quad- \nabla^{2} F(\hat{\mathbf{x}}^{\star}) ) (\bm{1}_{m} \otimes \mathbf{I}_{n})\\
        &\succeq -m L_{F}^{H}  \| \hat{\mathbf{x}}^{\star} - \bm{1}_{m} \otimes \operatorname{av}(\hat{\mathbf{x}}^{\star}) \| \mathbf{I}_{n},
    \end{align*}
    because $\textbf{A} \succeq -\|\textbf{A}\| \mathbf{I}_{d}$ for any square symmetric matrix $\textbf{A}$, whereby,
    \begin{align*}
        \lambda_{\min}( \nabla^{2} f(\operatorname{av}(\hat{\mathbf{x}}^{\star}))) \ge - m L_{F}^{H}\frac{5}{\mu} \sqrt{\alpha}
    \end{align*}
    as claimed.
\end{proof}

\begin{lemma}
    \label{lem: arbitrarily close}
    Let Assumptions \ref{ass: Strict saddle} and \ref{ass: Lipschitz} hold. Given constant $\mu > 0$, suppose $\mathcal{X}\subset \mathbb{R}^{n}$ is a compact set for which there exists $\alpha^{\dagger} > 0$ such that for all $0 < \alpha \le \alpha^{\dagger}$, ${\mathcal{X}^{\alpha}_{f}} \cap \mathcal{X} \ne \emptyset$, where 
    \begin{multline*}
        \mathcal{X}_{f}^{\alpha} \coloneqq \{ \mathbf{x}: \|\nabla f(\mathbf{x})\| \le \sqrt{m}L_{F}^{g} \frac{5}{\mu}\sqrt{\alpha},\\
        \lambda_{\min} (\nabla^{2} f(\mathbf{x})) \ge - m L_{F}^{H}\frac{5}{\mu} \sqrt{\alpha}\}. 
    \end{multline*}
    Then,
    \begin{align*}
        \lim_{\alpha\downarrow0} (\sup\{\textup{dist}(\mathbf{x}, \mathcal{X}_{f}^{\star}): \mathbf{x} \in {\mathcal{X}^{\alpha}_{f}} \cap \mathcal{X}\}) = 0.
    \end{align*}
\end{lemma}

\begin{proof}
    Following the approach used in  \cite[Lemma 3.8]{daneshmand2020second} to establish a similar result in the absence of the second-order requirement in the definition of $\mathcal{X}_{f}^{\alpha}$, we prove the lemma by contradiction. Suppose
    \begin{align}
    \label{eq: Contradiction hypothesis 2}
        \inf_{0 < \alpha \le \alpha^{\dagger}} \{\sup\{\textup{dist}(\mathbf{x}, \mathcal{X}_{f}^{\star}): \mathbf{x} \in {\mathcal{X}^{\alpha}_{f}} \cap \mathcal{X}\}\} = d^{\star} > 0.
    \end{align}
    Then, there exists a sequence $\{\mathbf{x}^{k}\}$ with $\mathbf{x}^{k} \in \mathcal{X}_{f}^{\alpha^{\dagger}/k} \cap \mathcal{X}$ and $\textup{dist} (\mathbf{x}^{k}, \mathcal{X}_{f}^{\star}) \ge d^{\star}$ for all $k \ge 1$. Since $\|\nabla f\|$ and $\lambda_{\min} (\nabla^{2} f)$ are continuous functions (by Assumptions \ref{ass: Lipschitz}), $\mathcal{X}_{f}^{\alpha^{\dagger}/k}$ is closed. Thus, $\mathcal{X}_{f}^{\alpha^{\dagger}/k} \cap \mathcal{X}$ is compact. Since $\{\mathbf{x}^{k}\} \subset \mathcal{X}$, we can find a convergent sub-sequence $\{\mathbf{x}^{t_k}\}$ with limit point $\mathbf{x}^{\infty}$ satisfying $\textup{dist} (\mathbf{x}^{\infty}, \mathcal{X}_{f}^{\star}) \ge d^{\star}$. Since $(\mathcal{X}_{f}^{\alpha^{\dagger}/k} \cap \mathcal{X})\supset (\mathcal{X}_{f}^{\alpha^{\dagger}/(k+1)}  \cap \mathcal{X})$, it follows that, $\mathbf{x}^{\infty} \in \mathcal{X}_{f}^{\alpha^{\dagger}/k} \cap \mathcal{X}$ for all $k \ge 0$. This means there exist constants $c_{1},c_{2} > 0$ such that $\|\nabla f(\mathbf{x}^{\infty})\| \le c_{1} / \sqrt{k} $, $\lambda_{\min} (\nabla^{2} f(\mathbf{x}^{\infty})) \ge - c_{2} / \sqrt{k} $ for all $k \ge 0$, implying $\|\nabla f(\mathbf{x}^{\infty})\| = 0$, $\lambda_{\min} (\nabla^{2} f(\mathbf{x}^{\infty})) \ge 0$. By Assumption \ref{ass: Strict saddle}, it follows that $\lambda_{\min} (\nabla^{2} f(\mathbf{x}^{\infty})) > 0$, and hence, $\textup{dist}(\mathbf{x}^{\infty}, \mathcal{X}_{f}^{\star}) = 0$, which contradicts the initial hypothesis \eqref{eq: Contradiction hypothesis 2}.
\end{proof}

Combining Lemmas~\ref{lem: first order guarantee for xhat*} through \ref{lem: arbitrarily close} leads to the following proof of Lemma \ref{lem: Local minimizers of Q}.

\begin{proofs}
    \textbf{Proof of Lemma \ref{lem: Local minimizers of Q}}: Set $\alpha^{\prime}$, $\eta$, $\epsilon$, $\gamma$, $\delta$, $\mu > 0$, as per Assumption \ref{ass: Local convexity}. By Lemma \ref{lem: first order guarantee for xhat*} and Lemma \ref{lem: second order guarantee for xhat*}, there exists threshold $\alpha^{\prime\prime}$ (dependent on $\epsilon$ and $\gamma$) such that for all $0 < \alpha \le \alpha^{\prime\prime}$,
    \begin{align}
    \label{eq: 1st condition}
        \|\sum_{i=1}^{m} \nabla f_{i}(\operatorname{av}(\hat{\mathbf{x}}^{\star})) \| \le  \sqrt{m}L_{F}^{g} \frac{5}{\mu}\sqrt{\alpha} < \epsilon,
    \end{align}
    and
    \begin{align}
    \label{eq: 2nd condition}
        \lambda_{\min}(\sum_{i=1}^{m} \nabla^{2} f_{i} (\operatorname{av}(\hat{\mathbf{x}}^{\star}))) \ge - mL_{F}^{H} \frac{5}{\mu}\sqrt{\alpha} > -\gamma,
    \end{align}
    Therefore, statement 3) in Assumption \ref{ass: Local convexity} must hold with $\hat{\mathbf{x}} = \bm{1}_{m} \otimes \operatorname{av}(\hat{\mathbf{x}}^{\star})$, which is clearly in consensus so that $\|\hat{\mathbf{x}} - \bm{1}_{m} \otimes \operatorname{av}(\hat{\mathbf{x}})\| = 0 \le \eta$. In particular, 
    \begin{itemize}
    \item $\operatorname{av}(\hat{\mathbf{x}}^{\star})\in\mathcal{X}_{f}^{\alpha}$ with $\mathcal{X}_{f}^{\alpha}$ as defined in Lemma~\ref{lem: arbitrarily close} in view of \eqref{eq: 1st condition} and \eqref{eq: 2nd condition};
    \item for all $0 < \alpha \le \min\{\alpha^{\prime},\alpha^{\prime\prime}\}$, there exists $\hat{\mathbf{x}}^{\prime} \in (\mathbb{R}^{n})^{m}$ such that: 
    \begin{align}
     & \label{eq: xprime} \|\hat{\mathbf{x}}^{\prime} - \bm{1}_{m} \otimes \operatorname{av}(\hat{\mathbf{x}}^{\star})\| \le \delta; \\
     & \label{eq: Hessian of F}
    \|\hat{\mathbf{y}} - \hat{\mathbf{x}}^{\prime}\| \le 3\delta \implies \lambda_{\min}(\nabla^{2} F(\hat{\mathbf{y}})) \ge \mu.
    \end{align}
    \end{itemize}
    Towards application of Lemma~\ref{lem: arbitrarily close}, let
    \begin{gather*}
        \alpha^{\dagger} = \min\{\alpha^{\prime},\alpha^{\prime\prime}\},
        \qquad
        \mathcal{X} \coloneqq \left\{\mathbf{x}: \|\nabla f(\mathbf{x})\| \le \sqrt{m}L_{F}^{g}\frac{5}{\mu}\sqrt{\alpha^{\dagger}}\right\}.
    \end{gather*}
    The set $\mathcal{X}$ is compact by Assumption~\ref{ass: Coercivity}. Moreover, for every $0<\alpha\le\alpha^{\dagger}$, the point $\operatorname{av}(\hat{\mathbf{x}}^{\star})$ belongs to $\mathcal{X}_{f}^{\alpha}\cap\mathcal{X}$ by \eqref{eq: 1st condition} and \eqref{eq: 2nd condition}. Now applying Lemma~\ref{lem: arbitrarily close}, there exists $0 < \alpha_{\delta} \le \alpha^{\dagger}$ such that $0 < \alpha \le \alpha_{\delta}$ implies
    $\textup{dist} (\operatorname{av}(\hat{\mathbf{x}}^{\star}), \mathcal{X}_{f}^{\star}) \le \delta/\sqrt{m}$. As such, there exists $\mathbf{x}^{\star} \in \mathcal{X}_{f}^{\star}$ such that 
    \begin{align} \label{eq: xbarstar}
    \|\bm{1}_{m} \otimes \operatorname{av}(\hat{\mathbf{x}}^{\star}) - \bm{1}_{m} \otimes \mathbf{x}^{\star}\|
    =\sqrt{m}\|\operatorname{av}(\hat{\mathbf{x}}^{\star}) - \mathbf{x}^{\star} \|
    \le \delta.        
    \end{align}
    With~\eqref{eq: xprime}, $\|\bm{1}_{m} \otimes \mathbf{y} - \bm{1}_{m} \otimes \mathbf{x}^{\star}\| \le \delta$ implies 
    \begin{align*}
        \|\bm{1}_{m} \otimes \mathbf{y} - \hat{\mathbf{x}}^{\prime}\| &=
    \| (\bm{1}_{m} \otimes \mathbf{y} - \bm{1}_{m} \otimes \mathbf{x}^{\star}) \\
    &\quad - (\bm{1}_{m} \otimes \operatorname{av}(\hat{\mathbf{x}}^{\star}) - \bm{1}_{m} \otimes \mathbf{x}^{\star}) \\
    &\quad - (\hat{\mathbf{x}}^{\prime} - \bm{1}_{m} \otimes \operatorname{av}(\hat{\mathbf{x}}^{\star})) \|\\
    &\le 3 \delta,
    \end{align*}
    by the triangle inequality, and thus, 
    \begin{gather*}
        \lambda_{\min}(\nabla^{2} f(\mathbf{y})) \ge m\lambda_{\min}(\nabla^{2} F(\bm{1}_{m} \otimes \mathbf{y})) \ge m\mu,
    \end{gather*}
    by~\eqref{eq: Hessian of F}; the first inequality holds by Weyl's inequality and the relations $\nabla^2 F(\bm{1}_{m} \otimes \mathbf{x}) = \mathrm{diag}(\nabla^2 f_1(\mathbf{x}),\ldots, \nabla^2 f_m(\mathbf{x}))$ and $\nabla^2 f(\mathbf{x})= \sum_{i=1}^m \nabla^2 f_{i}(\mathbf{x})$. To summarize, $f$ is locally strongly convex with modulus $m\mu$ in a $\delta$-neighbourhood of $\mathbf{x}^{\star}$. Therefore,
    in view of~\eqref{eq: xbarstar}, whereby $\|\mathbf{x}^{\star} - \operatorname{av}(\hat{\mathbf{x}}^{\star})\| \leq \delta/\sqrt{m} \leq \delta$, in conjunction with~\eqref{eq: 1st condition},
    \begin{multline*}
        \|\mathbf{x}^{\star} - \operatorname{av}(\hat{\mathbf{x}}^{\star})\| \le \frac{1}{m\mu} \|\nabla f(\mathbf{x}^{\star}) - \nabla f(\operatorname{av}(\hat{\mathbf{x}}^{\star}))\|\\
        \le \frac{1}{m\mu} \|\nabla f(\operatorname{av}(\hat{\mathbf{x}}^{\star}))\| \le \frac{ L_{F}^{g} }{\sqrt{m}\mu} \frac{5}{\mu}\sqrt{\alpha}.
    \end{multline*}
    As such, 
    given $\|\hat{\mathbf{x}}^{\star} - \bm{1}_{m} \otimes \operatorname{av}(\hat{\mathbf{x}}^{\star})\| \le \frac{5}{\mu} \sqrt{\alpha}$ as per \eqref{eq: consensus error at xstar}, 
    \begin{multline*}
        \textup{dist} (\hat{\mathbf{x}}_{i}^{\star}, \mathcal{X}_{f}^{\star}) \le \|\hat{\mathbf{x}}_{i}^{\star} - \operatorname{av}(\hat{\mathbf{x}}^{\star})\| + \textup{dist} (\operatorname{av}(\hat{\mathbf{x}}^{\star}), \mathcal{X}_{f}^{\star})\\
        \le \left(1 + \frac{L_{F}^{g}}{\sqrt{m}\mu}\right)\frac{5}{\mu}\sqrt{\alpha}
        \le \left(1 + \frac{\sqrt{m}L_{F}^{g}}{\mu}\right)\frac{5}{\mu}\sqrt{\alpha},
    \end{multline*}
    where the last inequality uses $m\ge1$. This proves the claimed bound.
\end{proofs}

\subsection{Proof of Lemma \ref{lem: Consensus error}}
\label{sec: Proof of Lemma Consensus error}
\begin{proof}
    Since $\bm{1}_{m}^{\top} \mathbf{W} = \bm{1}_{m}^{\top}$ and $\mathbf{W} \bm{1}_{m} = \bm{1}_{m}$ (see remark \ref{rem: Mixing matrix}), it holds that $\bm{1}_{m}\bm{1}_{m}^{\top} \mathbf{W} = \bm{1}_{m} \bm{1}_{m}^{\top} = \mathbf{W} \bm{1}_{m}\bm{1}_{m}^{\top}$, which means $\bm{1}_{m}\bm{1}_{m}^{\top}$ and $\mathbf{W}$ commute. Since the largest eigenvector of $\hat{\mathbf{W}}$ lies in the null space of $\mathbf{I}_{mn} - \frac{1}{m} (\bm{1}_{m}\bm{1}_{m}^{\top}) \otimes \mathbf{I}_{n}$, it follows that for all $\hat{\mathbf{x}} \in (\mathbb{R}^{n})^{m}$,
    \begin{multline*}
        \|(\mathbf{I}_{mn} - \frac{1}{m} (\bm{1}_{m}\bm{1}_{m}^{\top}) \otimes \mathbf{I}_{n}) \hat{\mathbf{W}}\hat{\mathbf{x}}\|\\
        \le \lambda_{2} \|(\mathbf{I}_{mn} - \frac{1}{m} (\bm{1}_{m}\bm{1}_{m}^{\top}) \otimes \mathbf{I}_{n})\hat{\mathbf{x}}\|. 
    \end{multline*}
    By Assumption \ref{ass: Bounded gradient disagreement},
    \begin{align}
    \label{eq: Consensus error bound}
        \begin{aligned}
        &\|\hat{\mathbf{x}}_{i}^{k+1} - \mathrm{av}(\hat{\mathbf{x}}^{k+1})\| \le \|\hat{\mathbf{x}}^{k+1} - \bm{1}_{m} \otimes \mathrm{av}(\hat{\mathbf{x}}^{k+1})\|\\
        &= \|(\mathbf{I}_{mn} - \frac{1}{m} (\bm{1}_{m}\bm{1}_{m}^{\top}) \otimes \mathbf{I}_{n}) (\hat{\mathbf{W}} \hat{\mathbf{x}}^{k} - \alpha (\nabla F(\hat{\mathbf{x}}^{k}) + \mathbf{n}^{k}))\|\\
        &\le \lambda_{2} \|\hat{\mathbf{x}}^{k} - \bm{1}_{m} \otimes \mathrm{av}(\hat{\mathbf{x}}^{k})\|+ \alpha \|(\mathbf{I}_{mn} - \frac{1}{m} (\bm{1}_{m}\bm{1}_{m}^{\top}) \otimes \mathbf{I}_{n})\\
        &\quad(\nabla F(\hat{\mathbf{x}}^{k}) - \frac{1}{m}\bm{1}_{m} \otimes \nabla f(\mathrm{av}(\hat{\mathbf{x}}^{k})))\| + \alpha \|\mathbf{n}^{k}\|\\
        &\le \lambda_{2} \|\hat{\mathbf{x}}^{k} - \bm{1}_{m} \otimes \mathrm{av}(\hat{\mathbf{x}}^{k})\| + \alpha (\|\nabla F(\hat{\mathbf{x}}^{k}) - \frac{1}{m}\mathbf{G}^{k}\|\\
        &\quad+ \|\frac{1}{m}\mathbf{G}^{k} - \frac{1}{m}\bm{1}_{m} \otimes \nabla f(\mathrm{av}(\hat{\mathbf{x}}^{k}))\|) + \alpha \|\mathbf{n}^{k}\|\\
        &\le (\lambda_{2}+\alpha L_{F}^{g})\|\hat{\mathbf{x}}^{k} - \bm{1}_{m} \otimes \mathrm{av}(\hat{\mathbf{x}}^{k})\| + \frac{\alpha}{m} \sum_{i=1}^{m} \|\nabla f_{i}(\hat{\mathbf{x}}_{i}^{k})\\
        &\quad- \nabla f(\mathrm{av}(\hat{\mathbf{x}}^{k}))\| + \alpha \|\mathbf{n}^{k}\|\\
        &\le (\lambda_{2}+\alpha L_{F}^{g})\|\hat{\mathbf{x}}^{k} - \bm{1}_{m} \otimes \mathrm{av}(\hat{\mathbf{x}}^{k})\|\\
        &\quad+ \frac{\alpha}{m} \sum_{i=1}^{m} \|\frac{1}{m} \sum_{j=1}^{m} \nabla f_{i}(\hat{\mathbf{x}}_{i}^{k}) - \nabla f_{j}(\hat{\mathbf{x}}_{i}^{k})\| + \alpha \|\mathbf{n}^{k}\|\\
        &\le (\lambda_{2}+\alpha L_{F}^{g})\|\hat{\mathbf{x}}^{k} - \bm{1}_{m} \otimes \mathrm{av}(\hat{\mathbf{x}}^{k})\| + \alpha D + \alpha \|\mathbf{n}^{k}\|.
        \end{aligned}
    \end{align}
    where $\mathbf{G}^{k} = [(\nabla f(\hat{\mathbf{x}}_{1}^{k}))^{\top},\cdots,(\nabla f(\hat{\mathbf{x}}_{m}^{k}))^{\top}]^{\top}$. By \eqref{eq: chi-square concentration}, for any $k > 0$,
    \begin{gather*}
         \mathbb{P}\bigg[\|\mathbf{n}^{k}\| \le \sigma (\sqrt{2\rho} + \sqrt{mn})\bigg]\ge 1-\mathrm{e}^{-\rho}.
    \end{gather*}
    Applying \eqref{eq: Boole's inequality} in Proposition \ref{pro: Boole-frechet Inequality} yields for any $k > 0$,
    \begin{gather*}
         \mathbb{P}\bigg[\forall\tau \in (0,k],~\|\mathbf{n}^{\tau}\| \le \sigma (\sqrt{2\rho} + \sqrt{mn})\bigg]\ge 1-k\mathrm{e}^{-\rho}.
    \end{gather*}
    Therefore, iterating \eqref{eq: Consensus error bound} from $k=0$ yields that there exists constant $\rho_{1} > 0$ such that for any $\rho \ge \rho_{1}$, it holds that $\lambda_{2} + \alpha L_{F}^{g} < 1$, and
    \begin{multline*}
        \mathbb{P}\bigg[\| \hat{\mathbf{x}}^{k} - \bm{1}_{m} \otimes \mathrm{av}(\hat{\mathbf{x}}^{k})\| \le (\lambda_{2} + \alpha L_{F}^{g})^{k} \|\hat{\mathbf{x}}^{0} - \bm{1}_{m} \otimes \mathrm{av}(\hat{\mathbf{x}}^{0})\|\\
        + \frac{\alpha D + \alpha \sigma (\sqrt{2\rho} + \sqrt{mn})}{1-\lambda_{2} - \alpha L_{F}^{g}}\bigg]\ge 1-k\mathrm{e}^{-\rho}
    \end{multline*}
    as claimed.
\end{proof}

\begin{figure*}[t]
\centering
\begin{tikzpicture}[
  node distance=1cm and 2.5cm,
  every node/.style={draw, font=\rmfamily, minimum width=3cm, minimum height=1.5cm, align=center},
  every path/.style={draw, thick, -{Latex[length=3mm]}},
  line width=0.7pt,
  box/.style={draw, font=\sffamily, minimum width=4cm, minimum height=2cm, align=center}
  ]

\node (l2) {Lemma \ref{lem: Decrease}\\ Gradient Descent};
\node[right=of l2] (l3) {Lemma \ref{lem: Improve or Localize}};
\node[right=of l3] (l4) {Lemma \ref{lem: Localization}};

\node[below=of l3] (l56) {Lemma \ref{lem: Dynamic}~\&~Lemma \ref{lem: Dynamic concentration 1}};
\node[below=of l4] (l7) {Lemma \ref{lem: Dynamic concentration 2}\\ Properties of  Coupling Sequences};
\node[below=of l7] (l8) {Lemma \ref{lem: Decrease for saddle by coupling}\\ Saddle Point Escape};

\node[left=of l8] (p2) {Lemma \ref{lem: Approximate second order}\\ Approximate Second-order Stationarity};

\path (l2) edge (l3);
\path (l3) edge (l4);
\path (l2) edge (p2);
\path (l4) edge (l7);
\path (l56) edge (l7);
\path (l56) edge (l8);
\path (l7) edge (l8);
\path (l8) edge (p2);

\end{tikzpicture}
\caption{Logical structure of the proof of Lemma \ref{lem: Approximate second order}.}
\label{fig: Proof structure}
\end{figure*}
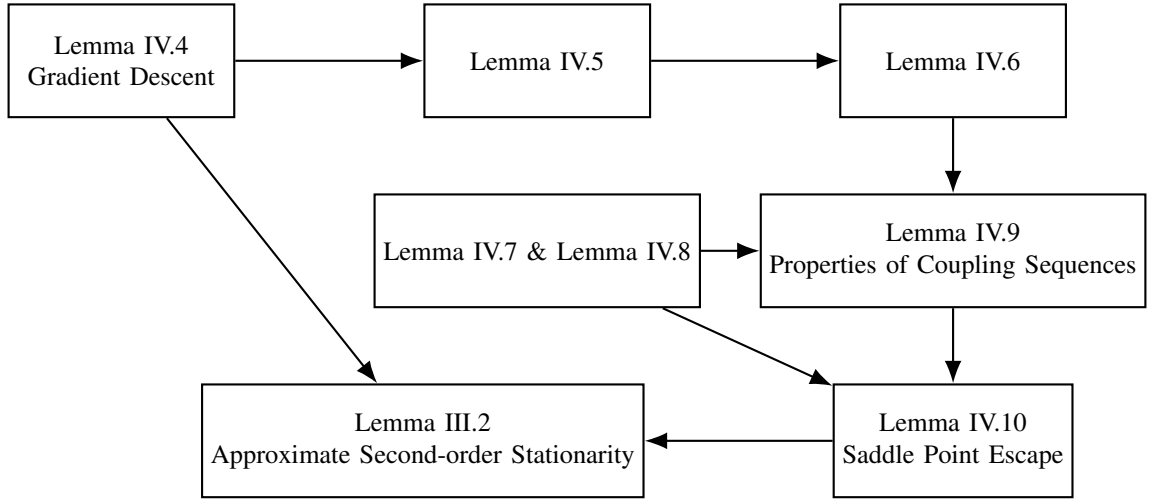

\subsection{Proof of Lemma \ref{lem: Approximate second order}}
\label{sec: Proof of Lemma Approximate second order}
For the sake of clarity, given $\rho \ge 1$ (to be specified later), let 
\begin{gather}
\label{eq: parameters}
    \begin{gathered}
        \alpha \coloneqq \frac{\lambda_{\min}(\mathbf{W})}{L_{F}^{g} \sqrt{\rho}},~d \coloneqq \frac{\sqrt{\alpha}}{40L_{F}^{H}\sqrt{\rho^{3}}},~\sigma^{2} \coloneqq \frac{\alpha d^{2}}{mn\rho^{3}},\\
        \epsilon_{g} \coloneqq \sqrt{\alpha},~\epsilon_{H} \coloneqq \sqrt{\epsilon_{g}L_{F}^{H}},~\text{ and } r \coloneqq \lceil\alpha^{-\tfrac{3}{2}}\rho\rceil.
    \end{gathered}
 \end{gather}

For readability, we abuse logical notation when dealing with events. Let $\Omega$ be the whole sample space. Specifically, given events of the form $\mathcal{E}_{i} := \{\omega \in \Omega : P_{i}\} \subseteq \Omega$ and $\mathcal{E}_{j} := \{\omega \in \Omega : P_{j}\} \subseteq \Omega$, where $P_{i}$ and $P_{j}$ are logical predicates defined on $\omega$, we denote:
\begin{multline}
\label{eq: Material implication}
    \mathcal{E}_{i} \Rightarrow \mathcal{E}_{j} := \left\{ \omega \in \Omega :~P_{i} \Rightarrow P_{j} \right\} \\
    = \left\{ \omega \in \Omega :~\neg P_{i} \lor P_{j} \right\} = \bar{\mathcal{E}}_{i} \cup \mathcal{E}_{j}.
\end{multline}

The middle step follows the classical logic of material implication (see Proposition~\ref{pro: Material implication}).

The logical structure of the proof of Lemma \ref{lem: Approximate second order} is illustrated in Fig. \ref{fig: Proof structure}. To establish the second order property of \textbf{NDGD}, we first decompose the change in $Q_{\alpha}(\hat{\mathbf{x}}^{k})$ from time $t_{0}$ to $t_{0} + t$ into two parts as considered in Lemma \ref{lem: Decrease}: i) the decrease due to gradient descent; and ii) the possible increase due to random perturbations. Then, it is proved that with high probability over certain iterations, either the function value decreases significantly, or the iterates stay within a small local region around the initial point (see Lemma \ref{lem: Improve or Localize}). The following proposition, which provides a concentration bound for the chi-square distribution, will be used in the proof of Lemma \ref{lem: Decrease}.

\begin{lemma}
\label{lem: Decrease}
    Let Assumptions \ref{ass: Lipschitz} and \ref{ass: Random perturbation} hold. Given $\rho \ge 1$, let $\alpha$, $d$, and $\sigma$ depend on $\rho$ as defined in \eqref{eq: parameters}. Then, for any $\hat{\mathbf{x}}^{0} \in (\mathbb{R}^{n})^{m}$, and $t_{0},~t \ge 0$,
    \begin{gather}
    \label{eq: Decomposition of the change}
        \mathbb{P} \Bigg[Q_{\alpha} (\hat{\mathbf{x}}^{t_{0}+t}) - Q_{\alpha} (\hat{\mathbf{x}}^{t_{0}}) \le
         q_{\rho}(t_{0},t) \Bigg] \ge 1 - 2\mathrm{e}^{-\rho},
    \end{gather}
    where
    \begin{multline}
    \label{eq: qrho}
        q_{\rho}(t_{0},t) = - \frac{\alpha}{2} \sum_{k=0}^{t-1} \|\nabla Q_{\alpha} (\hat{\mathbf{x}}^{t_{0}+k})\|^{2}\\
        +   mn \alpha \sigma^{2} (t + \sqrt{t\rho} + \rho),
    \end{multline}
    and $\{\hat{\mathbf{x}}^{k}\}$ are the iterates generated according to \eqref{eq: NDGD}.
\end{lemma}

\begin{proof}
Since the updates in \eqref{eq: NDGD} are time-invariant, it suffices to prove for the special case $t_{0}=0$. By Assumption \ref{ass: Lipschitz} and its definition in \eqref{eq: Q}, $Q_{\alpha}$ has $L_{Q_{\alpha}}^{g}$-Lipschitz continuous gradient. Then, by the definition of the directional derivative and application of the Cauchy--Schwarz inequality, for all $\hat{\mathbf{z}}, \hat{\mathbf{v}}\in (\mathbb{R}^{n})^{m}$,
    \begin{align}
    \label{eq: Lipschitz}
    \begin{aligned}
        \hat{\mathbf{v}}^{\top}\nabla^{2} Q_{\alpha}(\hat{\mathbf{x}})\hat{\mathbf{v}} &= \hat{\mathbf{v}}^{\top} \lim_{\epsilon \rightarrow 0} \frac{\nabla Q_{\alpha}(\hat{\mathbf{x}}+\epsilon\hat{\mathbf{v}}) - \nabla Q_{\alpha}(\hat{\mathbf{x}})}{\epsilon}\\
        &\le \lim_{\epsilon \rightarrow 0} \|\hat{\mathbf{v}}\| \frac{\|\nabla Q_{\alpha}(\hat{\mathbf{x}}+\epsilon\hat{\mathbf{v}}) - \nabla Q_{\alpha}(\hat{\mathbf{x}})\|}{\epsilon}\\
        &\le \lim_{\epsilon \rightarrow 0} \|\hat{\mathbf{v}}\| \frac{L_{Q_{\alpha}}^{g} \epsilon \|\hat{\mathbf{v}}\|}{\epsilon}\\
        &= L_{Q_{\alpha}}^{g} \|\hat{\mathbf{v}}\|^{2}.
    \end{aligned}
    \end{align}
    As such, by Taylor's theorem and in view of \eqref{eq: Lipschitz}, for any $t > 0$,
    \begin{align*}
        &Q_{\alpha} (\hat{\mathbf{x}}^{t+1}) - Q_{\alpha} (\hat{\mathbf{x}}^{t})\\
        &\quad= \nabla Q_{\alpha} (\hat{\mathbf{x}}^{t})^{\top} (\hat{\mathbf{x}}^{t+1} - \hat{\mathbf{x}}^{t}) + (\hat{\mathbf{x}}^{t+1} - \hat{\mathbf{x}}^{t})^{\top}\Big(\int^{1}_{0} (1-s)\\
        &\quad\quad\nabla^2 Q_{\alpha} (\hat{\mathbf{x}}^{t} + s(\hat{\mathbf{x}}^{t+1} - \hat{\mathbf{x}}^{t})) \textup{d}s\Big) (\hat{\mathbf{x}}^{t+1} - \hat{\mathbf{x}}^{t})\\
        &\quad\le \nabla Q_{\alpha} (\hat{\mathbf{x}}^{t})^{\top} (\hat{\mathbf{x}}^{t+1} - \hat{\mathbf{x}}^{t}) + \frac{L_{Q_{\alpha}}^{g}}{2} \|\hat{\mathbf{x}}^{t+1} - \hat{\mathbf{x}}^{t}\|^{2}\\
        &\quad= (-\alpha+\frac{L_{Q_{\alpha}}^{g} \alpha^{2}}{2}) \|\nabla Q_{\alpha} (\hat{\mathbf{x}}^{t})\|^{2} + \frac{L_{Q_{\alpha}}^{g} \alpha^{2}}{2}\|\mathbf{n}^{t}\|^{2}\\
        &\quad\quad  + (-\alpha+L_{Q_{\alpha}}^{g} \alpha^{2}) 
        \nabla Q_{\alpha} (\hat{\mathbf{x}}^{t})^{\top} \mathbf{n}^{t}.
    \end{align*}
    As such, by Cauchy-Schwarz inequality and Young's inequality,
    \begin{multline*}
        Q_{\alpha} (\hat{\mathbf{x}}^{t+1}) - Q_{\alpha} (\hat{\mathbf{x}}^{t}) \le (-\alpha+\frac{L_{Q_{\alpha}}^{g} \alpha^{2}}{2}) \|\nabla Q_{\alpha} (\hat{\mathbf{x}}^{t})\|^{2}\\
        + \frac{|-\alpha+L_{Q_{\alpha}}^{g} \alpha^{2}|}{2} (\|\nabla Q_{\alpha} (\hat{\mathbf{x}}^{t})\|^{2} + \|\mathbf{n}^{t}\|^{2}) + \frac{L_{Q_{\alpha}}^{g} \alpha^{2}}{2}\|\mathbf{n}^{t}\|^{2}.
    \end{multline*}
    By Assumption \ref{ass: Mixing matrix}, since $\rho \ge 1$, $\alpha = \lambda_{\min}(\mathbf{W})/(L_{F}^{g}\sqrt{\rho}) \le \lambda_{\min}(\mathbf{W})/L_{F}^{g}$, and it follows that $0 < \alpha \le 1/L_{Q_{\alpha}}^{g}$. Thus,
    \begin{gather*}
        Q_{\alpha} (\hat{\mathbf{x}}^{t+1}) - Q_{\alpha} (\hat{\mathbf{x}}^{t}) \le -\frac{\alpha}{2} \|\nabla Q_{\alpha} (\hat{\mathbf{x}}^{t})\|^{2} + \frac{\alpha}{2} \|\mathbf{n}^{t}\|^{2}.
    \end{gather*}
    Summing both sides over $t$ yields
    \begin{gather}
    \label{eq: descent}
        Q_{\alpha} (\hat{\mathbf{x}}^{t}) - Q_{\alpha} (\hat{\mathbf{x}}^{0})
        \le -\frac{\alpha}{2} \sum_{k=0}^{t-1} \Big( \|\nabla Q_{\alpha} (\hat{\mathbf{x}}^{k})\|^{2}
        - \|\mathbf{n}^{k}\|^{2}\Big).
    \end{gather}
    Let filtration $\mathcal{F}^{t} = \sigma\{\mathbf{n}^{0},\dots,\mathbf{n}^{t-1}\}$ denote the sigma-algebra generated by the perturbations up to time $t-1$, where $\sigma\{\cdot\}$ denotes the sigma field. Since $\sum_{k=0}^{t-1}\|\mathbf{n}^{k}/\sigma\|^{2}$ is the sum of squares of independent standard normal random variables, by definition, it follows a chi-square distribution with $D = tmn$. Then, by \eqref{eq: chi-square concentration} in Proposition \ref{pro: chi-square concentration}, it follows that
    \begin{gather}
    \label{eq: descent third term}
         \mathbb{P} \left[\sum_{k=0}^{t-1} \|\mathbf{n}^{k}\|^{2} \le 2mn\sigma^{2} (t+\sqrt{t\rho}+\rho)\right] \ge 1 - 2\mathrm{e}^{-\rho}.
    \end{gather}
    Therefore, \eqref{eq: Decomposition of the change} holds by substituting \eqref{eq: descent third term} into \eqref{eq: descent}.
\end{proof}

\begin{lemma}
\label{lem: Improve or Localize}
    Let Assumptions \ref{ass: Lipschitz} and \ref{ass: Random perturbation} hold. Given $\rho \ge 1$, let $\alpha$, $d$, and $\sigma$ dependent on $\rho$ as defined in \eqref{eq: parameters}. Then, for any $\hat{\mathbf{x}}^{0} \in (\mathbb{R}^{n})^{m}$, and $t_{0},~t \ge 0$,
    \begin{gather*}
        \mathbb{P} \Big[\forall \tau \in (0,t],~\|\hat{\mathbf{x}}^{t_{0}+\tau} - \hat{\mathbf{x}}^{t_{0}}\|^{2} \le p_{\rho}(t_{0},t,\tau) \Big] \ge 1 - 4t\mathrm{e}^{-\rho}.
    \end{gather*}
    where 
    \begin{multline*}
        p_{\rho}(t_{0},t,\tau) = 4 \alpha t
        ((Q_{\alpha} (\hat{\mathbf{x}}^{t_{0}}) - Q_{\alpha} (\hat{\mathbf{x}}^{t_{0}+\tau}))\\
        + 2  mn\alpha \sigma^{2} (t + \sqrt{t\rho} + \rho)),
    \end{multline*}
   where $\{\hat{\mathbf{x}}^{k}\}$ are the iterates generated according to \eqref{eq: NDGD}.
\end{lemma}

\begin{proof}
    Due to the time-invariant updates, it suffices to prove for the special case $t_{0} = 0$. Let
    \begin{multline*}
        \mathcal{E}_{a,\tau} = \Bigg\{\sum_{k=0}^{\tau-1} \|\nabla Q_{\alpha} (\hat{\mathbf{x}}^{k})\|^{2} \le 2 \alpha^{-1}(Q_{\alpha} (\hat{\mathbf{x}}^{0}) - Q_{\alpha} (\hat{\mathbf{x}}^{\tau})) \\
        + 2   mn \sigma^{2}(\tau + \sqrt{\tau\rho} + \rho)\Bigg\},
    \end{multline*}
    and
    \begin{gather*}
        \mathcal{E}_{b,\tau} = \Bigg\{\sum_{k=0}^{\tau-1}\|\mathbf{n}^{k}\|^2 \le 2  mn \sigma^{2}(\tau + \sqrt{\tau\rho} + \rho) \Bigg\}.
    \end{gather*}
    Applying Lemma \ref{lem: Decrease} with $t_{0} = 0$ yields that for any $\tau \ge 0$, $\rho \ge 1$,
    \begin{gather}
    \label{eq: localize first term}
        \mathbb{P} [\mathcal{E}_{a,\tau}]\ge 1-2\mathrm{e}^{-\rho}.
    \end{gather}
    By Cauchy-Schwarz inequality,
    \begin{align}
    \label{eq: localize}
        \begin{aligned}
            \|\hat{\mathbf{x}}^{\tau} -\hat{\mathbf{x}}^{0}\|^{2} & \overset{\eqref{eq: NDGD2}}{=} \alpha^{2}\|\sum_{k=0}^{\tau-1}(\nabla Q_{\alpha} (\hat{\mathbf{x}}^{k}) + \mathbf{n}^{k})\|^2\\
            & \le 2\alpha^{2} (\|\sum_{k=0}^{\tau-1}\nabla Q_{\alpha} (\hat{\mathbf{x}}^{k})\|^{2} + \|\sum_{k=0}^{\tau-1}\mathbf{n}^{k}\|^2)\\
            & \le 2\alpha^{2} \tau (\sum_{k=0}^{\tau-1}\|\nabla Q_{\alpha} (\hat{\mathbf{x}}^{k})\|^{2} + \sum_{k=0}^{\tau-1}\|\mathbf{n}^{k}\|^2).
        \end{aligned}
    \end{align}
    Thus,
    \begin{gather*}
        \mathbb{P} \Big[\forall \tau \in (0,t],~\|\hat{\mathbf{x}}^{\tau} - \hat{\mathbf{x}}^{0}\|^{2} \le p_{\rho}(0,t,\tau) \Big] \ge  \mathbb{P} \Big[\bigcap_{\tau=1}^{t} \left( \mathcal{E}_{a,\tau} \cap \mathcal{E}_{b,\tau} \right) \Big].
    \end{gather*}
    Recall \eqref{eq: Boole's inequality} in Proposition \ref{pro: Boole-frechet Inequality}, it follows that
    \begin{gather*}
        \mathbb{P}\left[ \bigcap_{i=1}^{n} \mathcal{E}_i \right] \ge 1 - \sum_{i=1}^{n} \mathbb{P}[\bar{\mathcal{E}}_{i}] = 1- \sum_{i=1}^{n} (1 - \mathbb{P}[\mathcal{E}_i]),
    \end{gather*}
    where $\bar{\mathcal{E}}_{i}$ denotes the complement of $\mathcal{E}_{i}$. Therefore, in view of \eqref{eq: descent third term} and \eqref{eq: localize first term}, for any $t > 0$,
    \begin{multline*}
        \mathbb{P} \Big[\forall \tau \in (0,t],~\|\hat{\mathbf{x}}^{\tau} - \hat{\mathbf{x}}^{0}\|^{2} \le p_{\rho}(0,t,\tau) \Big]\\
        \overset{\eqref{eq: Boole's inequality}}{\ge} 1- \sum_{\tau=1}^{t}(1 - \mathbb{P}[\mathcal{E}_{a,\tau}]) - \sum_{\tau=1}^{t}(1 - \mathbb{P}[\mathcal{E}_{b,\tau}]) \ge 1 - 4t\mathrm{e}^{-\rho}.
    \end{multline*}
\end{proof}

In the following, we introduce the definition of coupling sequences. Before that, we define some notations. Let $\gamma_{\alpha}^{k} \coloneqq -\lambda_{\min}(\nabla^{2} Q_{\alpha}(\hat{\mathbf{x}}^{k}))$ for $k \ge 0$, and let $\mathbf{e}_{\alpha}$ denote the eigenvector of $\nabla^{2} Q_{\alpha}(\hat{\mathbf{x}}^{0})$ corresponding to the eigenvalue $-\gamma_{\alpha}^{0}$.

\begin{definition}[Coupling sequences]
\label{def: Coupling sequence}
    A pair of sequences $\{\hat{\mathbf{y}}^{k}\}$ and $\{\hat{\mathbf{z}}^{k}\}$ of \eqref{eq: NDGD2} initialized at the same point $\hat{\mathbf{x}}^{0}$, i.e., $\hat{\mathbf{y}}^{0} = \hat{\mathbf{z}}^{0} = \hat{\mathbf{x}}^{0}$, are called coupling sequences if for all $k > 0$, the corresponding random perturbations $\mathbf{n}_{\hat{\mathbf{y}}}^{k}$ and $\mathbf{n}_{\hat{\mathbf{z}}}^{k}$ are the same in the directions different from $\mathbf{e}_{\alpha}$, i.e., $(\mathbf{n}_{\hat{\mathbf{y}}}^{k} - \mathbf{n}_{\hat{\mathbf{z}}}^{k}) \in \mathrm{span}(\mathbf{e}_{\alpha})$, and opposite in the direction of $\mathbf{e}_{\alpha}$, i.e., $\mathbf{e}_{\alpha}^{\top} \mathbf{n}_{\hat{\mathbf{y}}}^{k} = -\mathbf{e}_{\alpha}^{\top}\mathbf{n}_{\hat{\mathbf{z}}}^{k}$. In particular, coupling sequences share a common source of randomness.
\end{definition}

Given $t_{0},t \ge 0$, let
\begin{gather}
\label{eq: varsigma}
    \hat{\varsigma}^{t_{0}}(t) \coloneqq \sqrt{4\frac{(1+\alpha\gamma_{\alpha}^{t_{0}})^{2t}}{2\alpha\gamma_{\alpha}^{t_{0}} + (\alpha\gamma_{\alpha}^{t_{0}})^{2}} \sigma^{2}}.
\end{gather}
Next, we analyze the updating sequences initialized near a saddle point. To establish the following results, let
\begin{gather}
\label{eq: ls}
    \ell_{s} = \frac{d^{2}}{4\alpha r} - 2  mn \alpha \sigma^{2}(r + \sqrt{r\rho} + \rho),
\end{gather}
and $\alpha$, $d$, and $r$ depend on $\rho$ as defined in \eqref{eq: parameters}. Further, with $\hat{\varsigma}^{t_{0}}(t)$ defined in \eqref{eq: varsigma}, define the following two events:
\begin{multline}
\label{eq: Event A}
    \mathcal{E}_{A}^{t_{0},t} \coloneqq \{\exists \tau \in (0,t],~\min\{Q_{\alpha}(\hat{\mathbf{y}}^{t_{0}+\tau}) - Q_{\alpha}(\hat{\mathbf{y}}^{t_{0}}),\\
    Q_{\alpha}(\hat{\mathbf{z}}^{t_{0}+\tau}) - Q_{\alpha}(\hat{\mathbf{z}}^{t_{0}})\} \le -\ell_{s}\};
\end{multline}
\vspace{-2em}
\begin{multline}
\label{eq: Event B}
    \mathcal{E}_{B}^{t_{0},t} \coloneqq \{\forall \tau \in (0,t],~\max\{\|\hat{\mathbf{y}}^{t_{0}+\tau}-\hat{\mathbf{y}}^{t_{0}}\|,\\
    \|\hat{\mathbf{z}}^{t_{0}+\tau}-\hat{\mathbf{z}}^{t_{0}}\|\} \le d\}.
\end{multline}

The following results can be summarized as follows: If the function values of both coupling sequences do not exhibit a sufficient decrease, then both coupling sequences are localized in a small ball around $\hat{\mathbf{x}}^{k}$ within $r$ iterations (see Lemma \ref{lem: Localization}). We study the differences between coupling sequences by decomposing their dynamics into two parts (see Lemma \ref{lem: Dynamic}): i) noise part, in which we analyze the tail properties (see Lemma \ref{lem: Dynamic concentration 1}); ii) Hessian part, in which we find either one of the sequences exhibits a sufficient decrease, or the Hessian part will stay small compared to the noise part (see Lemma \ref{lem: Dynamic concentration 2}). Combining the above results, we show that the updates decrease significantly after certain iterations, with high probability when they are initialized near a saddle point (see Lemma \ref{lem: Decrease for saddle by coupling}).

\begin{lemma}
\label{lem: Localization}
    Let Assumptions \ref{ass: Lipschitz} and \ref{ass: Random perturbation} hold. Given $\rho \ge 1$, let $\alpha$, $d$, and $\sigma$ depend on $\rho$ as defined in \eqref{eq: parameters}. Then, for any $t_{0},~t \ge 0$, the following holds: If $\{\hat{\mathbf{y}}^{k}\}$ and $\{\hat{\mathbf{z}}^{k}\}$ are coupling sequences, then,
    \begin{gather}
    \label{eq: Localization}
        \mathbb{P} \Big[\mathcal{E}_{A}^{t_{0},r} \cup
        \mathcal{E}_{B}^{t_{0},r}\Big] \ge 1 - 8r\mathrm{e}^{-\rho},
    \end{gather}
    where $\mathcal{E}_{A}^{t_{0},r}$ and $\mathcal{E}_{B}^{t_{0},r}$ are defined in \eqref{eq: Event A} and \eqref{eq: Event B}.
\end{lemma}

\begin{proof}
    Due to the time-invariant updates, it suffices to prove for the special case $t_{0} = 0$. Considering coupling sequences $\{\hat{\mathbf{y}}^{k}\}$ and $\{\hat{\mathbf{z}}^{k}\}$, applying Lemma \ref{lem: Improve or Localize} with $t_{0} = 0$ yields that for any $t > 0$,
    \begin{gather*}
        \mathbb{P} [\mathcal{E}_{c,t}] \ge 1 - 4t\mathrm{e}^{-\rho},\qquad
        \mathbb{P} [\mathcal{E}_{d,t}] \ge 1 - 4t\mathrm{e}^{-\rho},
    \end{gather*}
    where
    \begin{multline*}
        \mathcal{E}_{c,t} = \Bigg\{\forall \tau \in (0,t],~\|\hat{\mathbf{y}}^{\tau} - \hat{\mathbf{y}}^{0}\|^{2} \le 4 
        \alpha t ((Q_{\alpha} (\hat{\mathbf{y}}^{0}) - Q_{\alpha} (\hat{\mathbf{y}}^{\tau}))\\
        + 2  mn \alpha \sigma^{2} (t + \sqrt{t\rho} + \rho))\Bigg\}
    \end{multline*}
    and
    \begin{multline*}
        \mathcal{E}_{d,t} = \Bigg\{\forall \tau \in (0,t],~\|\hat{\mathbf{z}}^{\tau} - \hat{\mathbf{z}}^{0}\|^{2} \le 4 
        \alpha t ((Q_{\alpha} (\hat{\mathbf{z}}^{0})- Q_{\alpha} (\hat{\mathbf{z}}^{\tau}))\\
        + 2  mn \alpha \sigma^{2} (t + \sqrt{t\rho} + \rho))\Bigg\}
    \end{multline*}
    By \eqref{eq: Boole's inequality} in Proposition \ref{pro: Boole-frechet Inequality},
    \begin{align}
    \label{eq: coupling sequences 1}
        \begin{aligned}
            &\mathbb{P} \Big[\forall \tau \in (0,t],~\max\{\|\hat{\mathbf{y}}^{\tau} - \hat{\mathbf{y}}^{0}\|^{2},~\|\hat{\mathbf{z}}^{\tau} - \hat{\mathbf{z}}^{0}\|^{2}\}\le 4 \alpha t \cdot\\ 
            &\quad\quad\Big(\max\{Q_{\alpha}(\hat{\mathbf{y}}^{0}) - Q_{\alpha}(\hat{\mathbf{y}}^{\tau}),~Q_{\alpha}(\hat{\mathbf{z}}^{0}) - Q_{\alpha}(\hat{\mathbf{z}}^{\tau})\}\\
            &\quad\quad+ 2  mn \alpha \sigma^{2} (t + \sqrt{t\rho} + \rho)\Big) \Big]\\
            &\quad\ge 1- (1 - \mathbb{P}[\mathcal{E}_{c,t}]) - (1 - \mathbb{P}[\mathcal{E}_{d,t}])\\
            &\quad\ge 1 - 8t\mathrm{e}^{-\rho}.   
        \end{aligned}
    \end{align}
    Define events
    \begin{multline}
    \label{eq: D}
        \mathcal{C}^{r} = \Bigg\{\forall~\tau \in (0,r],~\max\{\|\hat{\mathbf{y}}^{\tau} - \hat{\mathbf{y}}^{0}\|^{2},~\|\hat{\mathbf{z}}^{\tau} - \hat{\mathbf{z}}^{0}\|^{2}\}\\
        \le 4 \alpha r \Big(\max\{Q_{\alpha}(\hat{\mathbf{y}}^{0}) - Q_{\alpha}(\hat{\mathbf{y}}^{\tau}),~Q_{\alpha}(\hat{\mathbf{z}}^{0}) - Q_{\alpha}(\hat{\mathbf{z}}^{\tau})\}\\
        + 2mn \alpha \sigma^{2} ( r + \sqrt{r\rho} + \rho)\Big)\Bigg\},
    \end{multline}
    By \eqref{eq: Material implication},
    \begin{gather*}
        (\mathcal{C}^{r} \cap (\bar{\mathcal{E}}_{A}^{0,r})) \Rightarrow \mathcal{E}_{B}^{0,r} =  \bar{\mathcal{C}^{r}} \cup \mathcal{E}_{A}^{0,r} \cup \mathcal{E}_{B}^{0,r}
    \end{gather*}
    with $\bar{\mathcal{E}}_{A}^{0,r}$ denoting the complementary event of $\mathcal{E}_{A}^{0,r}$. Applying \eqref{eq: Material implication} again yields
    \begin{gather*}
        (\mathcal{C}^{r} \cap \bar{\mathcal{E}}_{A}^{0,r}) \Rightarrow \mathcal{E}_{B}^{0,r} = \mathcal{C}^{r} \Rightarrow (\mathcal{E}_{A}^{0,r} \cup \mathcal{E}_{B}^{0,r}).
    \end{gather*}
    By \eqref{eq: Material implication} and \eqref{eq: ls}, it holds that $(\mathcal{C}^{r} \cap (\bar{\mathcal{E}}_{A}^{0,r})) \subset \mathcal{E}_{B}^{0,r}$, which means $(\mathcal{C}^{r} \cap (\bar{\mathcal{E}}_{A}^{0,r})) \Rightarrow \mathcal{E}_{B}^{0,r} = \Omega$. Therefore, $\mathcal{C}^{r} \Rightarrow (\mathcal{E}_{A}^{0,r} \cup \mathcal{E}_{B}^{0,r}) = \Omega$, which yields $\mathcal{C}^{r} \subset (\mathcal{E}_{A}^{0,r} \cup \mathcal{E}_{B}^{0,r})$ and
    \begin{gather*}
        \mathbb{P} \Big[\mathcal{E}_{A}^{0,r} \cup \mathcal{E}_{B}^{0,r} \Big] \ge \mathbb{P} \Big[\mathcal{C}^{r} \Big].
    \end{gather*}
    Then, by \eqref{eq: coupling sequences 1}, $\mathbb{P} \left[\mathcal{E}_{A}^{0,r} \cup \mathcal{E}_{B}^{0,r}\right] \ge 1 - 8r\mathrm{e}^{-\rho}$ as claimed.
\end{proof}

To establish the following results, define $\hat{\Delta}^{k} \coloneqq \hat{\mathbf{y}}^{k} - \hat{\mathbf{z}}^{k}$,  $\hat{\delta}_{\mathbf{n}}^{k} = \mathbf{n}_{\hat{\mathbf{y}}}^{k} - \mathbf{n}_{\hat{\mathbf{z}}}^{k}$,
\begin{gather*}
    \mathcal{I}^{k} = \int_{0}^{1} \nabla^{2} Q_{\alpha}(s\hat{\mathbf{y}}^{k} + (1-s)\hat{\mathbf{z}}^{k}) \textup{d}s~~\textup{and}~~\mathcal{H}^{0} = \nabla^{2} Q_{\alpha}(\hat{\mathbf{x}}^{0}),
\end{gather*}
and
\begin{gather}
\label{eq: Dynamic expression}
    \begin{aligned}
        \hat{\Delta}_{1}^{k} &\coloneqq -\alpha \sum_{\tau=0}^{k-1} (\mathbf{I}_{mn} - \alpha \mathcal{H}^{0})^{k-1-\tau} (\mathcal{I}^{\tau} - \mathcal{H}^{0}) \cdot \hat{\Delta}^{\tau},\\
        \hat{\Delta}_{2}^{k} &\coloneqq -\alpha \sum_{\tau=0}^{k-1} (\mathbf{I}_{mn} - \alpha \mathcal{H}^{0})^{k-1-\tau} \hat{\delta}_{\mathbf{n}}^{\tau}.
    \end{aligned}
\end{gather}
Given $t_{0},t \ge 0$, also define event
\begin{gather}
\label{eq: Event C}
    \mathcal{E}_{C}^{t_{0},t} \coloneqq \{\forall \tau \in (0,t],~ \|\hat{\Delta}_{1}^{\tau}\| \le \frac{\alpha \hat{\varsigma}^{t_{0}}(\tau)}{10}\}.
\end{gather}

\begin{lemma}
\label{lem: Dynamic}
    Let Assumption \ref{ass: Random perturbation} hold. Given $\rho \ge 1$, let $\alpha$ depend on $\rho$ as defined in \eqref{eq: parameters}. If $\{\hat{\mathbf{y}}^{k}\}$ and $\{\hat{\mathbf{z}}^{k}\}$ are coupling sequences, then for any $k \ge 0$,
    \begin{gather*}
        \hat{\Delta}^{k} =\hat{\Delta}_{1}^{k} + \hat{\Delta}_{2}^{k},
    \end{gather*}
    where $\hat{\Delta}_{1}^{k}$ and $\hat{\Delta}_{2}^{k}$ are defined in \eqref{eq: Dynamic expression}.
\end{lemma}

\begin{proof}
    Recall that the update formulas for coupling sequences are
    \begin{gather*}
        \hat{\mathbf{y}}^{k+1} = \hat{\mathbf{y}}^{k} - \alpha (\nabla Q_{\alpha}(\hat{\mathbf{y}}^{k}) + \mathbf{n}_{\hat{\mathbf{y}}}^{k}),\\
        \hat{\mathbf{z}}^{k+1} = \hat{\mathbf{z}}^{k} - \alpha (\nabla Q_{\alpha}(\hat{\mathbf{z}}^{k}) + \mathbf{n}_{\hat{\mathbf{z}}}^{k}).
    \end{gather*}
    By Definition \ref{def: Coupling sequence} and Assumption \ref{ass: Lipschitz}, for any $k > 0$, it follows that
    \begin{align*}
        \hat{\Delta}^{k+1} &= \hat{\mathbf{y}}^{k+1} - \hat{\mathbf{z}}^{k+1} = \hat{\Delta}^{k} - \alpha(\nabla Q_{\alpha}( \hat{\mathbf{y}}^{k}) - \nabla Q_{\alpha}( \hat{\mathbf{z}}^{k}) + \hat{\delta}_{\mathbf{n}}^{k})\\
        &= \hat{\Delta}^{k} - \alpha\Big(\int_{\hat{\mathbf{z}}^{k}}^{\hat{\mathbf{y}}^{k}} \nabla^{2} Q_{\alpha} (\hat{\mathbf{x}}) d\hat{\mathbf{x}} \cdot \hat{\Delta}^{k} + \hat{\delta}_{\mathbf{n}}^{k}\Big)\\
        &= \hat{\Delta}^{k} - \alpha\Big(\int_{0}^{1} \nabla^{2} Q_{\alpha}(s\hat{\mathbf{y}}^{k} + (1-s)\hat{\mathbf{z}}^{k}) \textup{d}s \cdot \hat{\Delta}^{k} + \hat{\delta}_{\mathbf{n}}^{k}\Big)\\
        &= (\mathbf{I}_{mn} - \alpha \mathcal{H}^{0}) \hat{\Delta}^{k} - \alpha ((\mathcal{I}^{k}- \mathcal{H}^{0}) \cdot \hat{\Delta}^{k} + \hat{\delta}_{\mathbf{n}}^{k}).\\
    \end{align*}
    Then, it follows that for any $0 \le \tau \le k-1$, multiplying $(\mathbf{I}_{mn} - \alpha\mathcal{H}^{0})^{k-\tau-1}$ yields
    \begin{multline}
    \label{eq: Delat t}
        (\mathbf{I}_{mn} - \alpha\mathcal{H}^{0})^{k-\tau-1}\hat{\Delta}^{\tau+1} = (\mathbf{I}_{mn} - \alpha\mathcal{H}^{0})^{k-\tau} \hat{\Delta}^{\tau}\\
        - \alpha (\mathbf{I}_{mn} - \alpha\mathcal{H}^{0})^{k-\tau-1} ((\mathcal{I}^{\tau} - \mathcal{H}^{0}) \hat{\Delta}^{\tau} + \hat{\delta}_{\mathbf{n}}^{\tau}).
    \end{multline}
    Summing up \eqref{eq: Delat t} over $0 \le \tau \le k-1$ yields
    \begin{gather*}
        \hat{\Delta}^{k} = -\alpha \sum_{\tau=0}^{k-1} (\mathbf{I}_{mn} - \alpha \mathcal{H}^{0})^{k-\tau-1} \cdot ((\mathcal{I}^{\tau} - \mathcal{H}^{0}) \cdot \hat{\Delta}^{\tau} + \hat{\delta}_{\mathbf{n}}^{\tau})
    \end{gather*}
    as claimed.
\end{proof}

\begin{lemma}
\label{lem: Dynamic concentration 1}
    Let Assumption \ref{ass: Random perturbation} hold. There exists $\rho_{\min,1} \ge 1$ such that for any $\rho \ge \rho_{\min,1}$, with $\alpha$, $d$, $\sigma$, $\epsilon_{H}$, and $r$ dependent on $\rho$ as defined in \eqref{eq: parameters}, the following holds: If $\lambda_{\min}(\nabla^{2} Q_{\alpha}(\hat{\mathbf{x}}^{0})) < -\epsilon_{H}$ and $\{\hat{\mathbf{y}}^{k}\}$ and $\{\hat{\mathbf{z}}^{k}\}$ are coupling sequences, then for any $k \ge 0$,
    \begin{gather*}
        \mathbb{P} \left[\|\hat{\Delta}_{2}^{k}\| \le \alpha \hat{\varsigma}^{0}(k) \sqrt{2\rho} \right] \ge 1 - 2\mathrm{e}^{-\rho},\\
        \mathbb{P} \left[\|\hat{\Delta}_{2}^{r}\| \ge \frac{\alpha \hat{\varsigma}^{0}(r)}{5}\right] \ge \frac{2}{3},
    \end{gather*}
\end{lemma}

\begin{proof}
    Since $\mathbf{n}^{k}\sim\mathcal{N}(\bm{0},\sigma^{2}\mathbf{I})$ independently across $k$, for any $k\ge0$ the component along $\mathbf{e}_{\alpha}$ of 
    \begin{gather}
    \label{eq: Sum of dev}
        \sum_{\tau=0}^{k-1} (\mathbf{I}_{mn} - \alpha \nabla^{2} Q_{\alpha}(\hat{\mathbf{x}}^{0}))^{k-1-\tau} \hat{\delta}_{\mathbf{n}}^{\tau}
    \end{gather}
    is also a one-dimensional Gaussian random variable with zero mean and variance
    \begin{gather*}
        (\Sigma_{k})^{2} = 4\frac{(1+\alpha\gamma_{\alpha}^{0})^{2k} - 1}{2\alpha\gamma_{\alpha}^{0} + (\alpha\gamma_{\alpha}^{0})^{2}} \sigma^{2} \le (\hat{\varsigma}^{0}(k))^{2}.
    \end{gather*}
    By Definition \ref{def: Coupling sequence}, it holds that \eqref{eq: Sum of dev} is $0$ along all orthogonal directions of $\mathbf{e}_{\alpha}$. Thus, the first inequality can be concluded by the Gaussian concentration inequality. For the second one, due to the fact that $\alpha,\gamma_{\alpha}^{0} > 0$, it holds that there exists $\rho_{\min,1} \ge 1$ such that for any $\rho \ge \rho_{\min,1}$,
    \begin{gather*}
        (1 + \alpha\gamma_{\alpha}^{0})^{2r} - 1 \ge \frac{1}{\sqrt{3}} (1+\alpha\gamma_{\alpha}^{0})^{2r},
    \end{gather*}
    which means $\Sigma_{r} \ge \hat{\varsigma}^{0}(r)/\sqrt{3}$. Since $\hat{\Delta}_{2}^{r}$ is supported on $\operatorname{span}(\mathbf{e}_{\alpha})$, its signed scalar component along $\mathbf{e}_{\alpha}$ is Gaussian with variance $\alpha^{2}(\Sigma_{r})^{2}$. Therefore,
    \begin{gather*}
        \mathbb{P} \left[\|\hat{\Delta}_{2}^{r}\| \ge \frac{\alpha \hat{\varsigma}^{0}(r)}{5}\right] \ge \mathbb{P}(\|\hat{\Delta}_{2}^{r}\| \ge \alpha \Sigma_{r} \rho^{\prime}) \ge 1-\frac{2\rho^{\prime}}{\sqrt{2\pi}} \ge \frac{2}{3}
    \end{gather*}
    with $\rho^{\prime} = \sqrt{3}/5$ as claimed.
\end{proof}

\begin{lemma}
\label{lem: Dynamic concentration 2}
    Let Assumptions \ref{ass: Lipschitz} and \ref{ass: Random perturbation} hold. There exists $\rho_{\min,3} \ge 1$ such that for any $\rho \ge \rho_{\min,3}$, with $\alpha$, $d$, $\sigma$, $\epsilon_{H}$, and $r$ dependent on $\rho$ as defined in \eqref{eq: parameters}, the following holds: If $\lambda_{\min}(\nabla^{2} Q_{\alpha}(\hat{\mathbf{x}}^{0})) < -\epsilon_{H}$ and $\{\hat{\mathbf{y}}^{k}\}$ and $\{\hat{\mathbf{z}}^{k}\}$ are coupling sequences, then for any $t_{0} \ge 0$,
    \begin{gather*}
        \mathbb{P} \Big[\mathcal{E}_{A}^{t_{0},r} \cup \mathcal{E}_{C}^{t_{0},r}
        \Big] \ge 1 - (2r^{2} + 8r) \mathrm{e}^{-\rho},
    \end{gather*}
    where $\mathcal{E}_{A}^{t_{0},r}$ and $\mathcal{E}_{C}^{t_{0},r}$ are defined in \eqref{eq: Event A} and \eqref{eq: Event C}.
\end{lemma}

\begin{proof}
    Due to the time-invariant updates, it suffices to prove for the special case $t_{0} = 0$. By Lemma \ref{lem: Dynamic concentration 1}, there exists $\rho_{\min,1} \ge 1$ such that for any $\rho \ge \rho_{\min,1}$, the following holds: If $\lambda_{\min}(\nabla^{2} Q_{\alpha}(\hat{\mathbf{x}}^{0})) < -\epsilon_{H}$ and $\{\hat{\mathbf{y}}^{k}\}$ and $\{\hat{\mathbf{z}}^{k}\}$ are coupling sequences, then for any $t \ge 0$,
    \begin{gather}
    \label{eq: Delta 2 for t times}
        \mathbb{P} \left[\forall\tau\in(0,t],~\|\hat{\Delta}_{2}^{\tau}\| \le \alpha \hat{\varsigma}^{0}(\tau) \sqrt{2\rho} \right] \ge 1 - 2t\mathrm{e}^{-\rho}.
    \end{gather}
    First, it is needed to prove the following claim by induction for any $t \le r$ with $r$ defined in \eqref{eq: parameters}:
    \begin{gather}
    \label{eq: B to C}
        \mathbb{P} \Big[\mathcal{E}_{B}^{0,t} \Rightarrow \mathcal{E}_{C}^{0,t} \Big] \ge 1 - 2t^{2} \mathrm{e}^{-\rho}.
    \end{gather}
    By \eqref{eq: Material implication}, it holds that
    \begin{gather*}
        \mathcal{E}_{B}^{0,t} \Rightarrow \mathcal{E}_{C}^{0,t} = \bar{\mathcal{E}}_{B}^{0,t} \cup \mathcal{E}_{C}^{0,t}
    \end{gather*}
    with $\bar{\mathcal{E}}_{B}^{0,t}$ denoting the complementary event of $\mathcal{E}_{B}^{0,t}$ and $\mathcal{E}_{B}^{0,t}$ defined in \eqref{eq: Event B}. For the base case $t=0$, the claim holds since $\Delta^{0} = \hat{\Delta}_{1}^{0} = \hat{\Delta}_{2}^{0} = \bm{0}$ as per Definition~\ref{def: Coupling sequence}. For the induction step, suppose the claim \eqref{eq: B to C} holds for $t$, then
    \begin{gather}
    \label{eq: B to C and}
        (\mathcal{E}_{B}^{0,t} \Rightarrow \mathcal{E}_{C}^{0,t}) \cap \{\forall \tau \in (0,t],~\|\hat{\Delta}_{2}^{\tau}\| \le \alpha\hat{\varsigma}^{0}(\tau) \sqrt{2\rho}\}
    \end{gather}
    is equivalent to 
    \begin{gather*}
        (\bar{\mathcal{E}}_{B}^{0,t} \cup \mathcal{E}_{C}^{0,t}) \cap \{\forall \tau \in (0,t],~\|\hat{\Delta}_{2}^{\tau}\| \le \alpha\hat{\varsigma}^{0}(\tau) \sqrt{2\rho}\}
    \end{gather*}
    by \eqref{eq: Material implication}. Since i) $\bar{\mathcal{E}}_{B}^{0,t}$ is a superset of
        \begin{gather*}
            \bar{\mathcal{E}}_{B}^{0,t} \cap \{\forall \tau \in (0,t],~\|\hat{\Delta}_{2}^{\tau}\| \le \alpha\hat{\varsigma}^{0}(\tau) \sqrt{2\rho}\},
        \end{gather*}
    and ii) by \eqref{eq: Event C} and Lemma \ref{lem: Dynamic},
    \begin{gather*}
        \mathcal{E}_{C}^{0,t} \cap \{\forall \tau \in (0,t],~\|\hat{\Delta}_{2}^{\tau}\| \le \alpha\hat{\varsigma}^{0}(\tau) \sqrt{2\rho}\}
    \end{gather*}
    is a subset of
    \begin{gather*}
        \mathcal{E}_{C}^{0,t} \cap \{\forall \tau \in (0,t],~\|\Delta^{\tau}\| \le \alpha \hat{\varsigma}^{0}(\tau) (\sqrt{2\rho}+\frac{1}{10})\},
    \end{gather*}
    then \eqref{eq: B to C and} is a subset of
    \begin{gather*}
        \bar{\mathcal{E}}_{B}^{0,t} \cup (\mathcal{E}_{C}^{0,t} \cap \{\forall \tau \in (0,t],~\|\Delta^{\tau}\| \le \alpha \hat{\varsigma}^{0}(\tau) (\sqrt{2\rho}+\frac{1}{10})\}),
    \end{gather*}
    which is equivalent to
    \begin{gather*}
        \mathcal{E}_{B}^{0,t} \Rightarrow (\mathcal{E}_{C}^{0,t} \cap \{\forall \tau \in (0,t],~\|\Delta^{\tau}\| \le \alpha \hat{\varsigma}^{0}(\tau) (\sqrt{2\rho}+\frac{1}{10})\})
    \end{gather*}
    by \eqref{eq: Material implication}. Now, it is concluded that
    \begin{multline}
    \label{eq: Set B to C and Delta}
        \mathbb{P} \left[\mathcal{E}_{B}^{0,t} \Rightarrow (\mathcal{E}_{C}^{0,t} \cap \{\forall \tau \in (0,t],~
        \|\hat{\Delta}^{\tau}\| \le \alpha \hat{\varsigma}^{0}(\tau) (\sqrt{2\rho}+\frac{1}{10})\}) \right]\\
        \ge \mathbb{P} \left[(\mathcal{E}_{B}^{0,t} \Rightarrow \mathcal{E}_{C}^{0,t}) \cap \{\forall \tau \in (0,t],~\|\hat{\Delta}_{2}^{\tau}\| \le \alpha\hat{\varsigma}^{0}(\tau) \sqrt{2\rho}\} \right]\\
        \ge 1 - 2(t^{2}+t)\mathrm{e}^{-\rho},
    \end{multline}
    where the last inequality is by applying \eqref{eq: Boole's inequality} in Proposition \ref{pro: Boole-frechet Inequality} with \eqref{eq: Delta 2 for t times} and \eqref{eq: B to C}. By Assumption \ref{ass: Lipschitz}, $Q_{\alpha}$ has $L_{Q_{\alpha}}^{H}$-Lipschitz continuous Hessian. Then, for $t+1$, i) it follows that
    \begin{multline}
    \label{eq: Delta 1}
        \|\hat{\Delta}_{1}^{t+1}\| \le \alpha \sum_{\tau=0}^{t} (1+\alpha\gamma_{\alpha}^{0})^{t-\tau} L_{Q_{\alpha}}^{H}\\
        \max\{\|\hat{\mathbf{y}}^{\tau}-\hat{\mathbf{y}}^{0}\|,~\|\hat{\mathbf{z}}^{\tau}-\hat{\mathbf{z}}^{0}\|\}\cdot \|\hat{\Delta}^{\tau}\|
    \end{multline}
    by \eqref{eq: Dynamic expression}; ii) by \eqref{eq: Event B}, it follows that
    \begin{gather}
    \label{eq: B to B at t+1}
        \mathcal{E}_{B}^{0,t+1} \subset \mathcal{E}_{B}^{0,t}.
    \end{gather}
    Note that \eqref{eq: Delta 1} and \eqref{eq: B to B at t+1} hold for all $t \ge 0$. By \eqref{eq: varsigma}, it follows that
    \begin{gather}
    \label{eq: Geometric sequence}
        \sum_{\tau=1}^{t} (1+\alpha\gamma_{\alpha}^{0})^{t-\tau} \hat{\varsigma}^{0}(\tau) = t \hat{\varsigma}^{0}(t).
    \end{gather}
    Let 
    \begin{gather*}
        \mathcal{D}^{t} = \{\|\hat{\Delta}_{1}^{t}\| \le \alpha^{2}  L_{Q_{\alpha}}^{H}d (t-1) \hat{\varsigma}^{0}(t-1)(\sqrt{2\rho}+\frac{1}{10})\}.
    \end{gather*}
    Then, by \eqref{eq: Set B to C and Delta}, \eqref{eq: Delta 1} and \eqref{eq: B to B at t+1}, it follows that
    \begin{align*}
        &\quad\mathbb{P} \Big[\mathcal{E}_{B}^{0,t+1} \Rightarrow (\mathcal{E}_{C}^{0,t} \cap \mathcal{D}^{t+1}) \Big] \overset{\eqref{eq: B to B at t+1}}{\ge} \mathbb{P} \Big[\mathcal{E}_{B}^{0,t} \Rightarrow (\mathcal{E}_{C}^{0,t} \cap \mathcal{D}^{t+1}) \Big]\\
        &\overset{\eqref{eq: Event B},\eqref{eq: Geometric sequence}}{\ge} \mathbb{P} \Big[\mathcal{E}_{B}^{0,t} \Rightarrow (\mathcal{E}_{C}^{0,t} \cap \{\|\hat{\Delta}_{1}^{t+1}\| \le \alpha \sum_{\tau=1}^{t} (1+\alpha\gamma_{\alpha}^{0})^{t-\tau} L_{Q_{\alpha}}^{H}\cdot\\
        &\quad\quad\max\{\|\hat{\mathbf{y}}^{\tau}-\hat{\mathbf{y}}^{0}\|,~\|\hat{\mathbf{z}}^{\tau}-\hat{\mathbf{z}}^{0}\|\} \alpha \hat{\varsigma}^{0}(\tau)(\sqrt{2\rho}+\frac{1}{10})\})\Big]\\
        &\overset{\eqref{eq: Delta 1}}{\ge} \mathbb{P} \Big[\mathcal{E}_{B}^{0,t} \Rightarrow \mathcal{E}_{C}^{0,t} \cap \{\forall \tau \in (0,t],~\|\hat{\Delta}^{\tau}\| \le \alpha \hat{\varsigma}^{0}(\tau) (\sqrt{2\rho}+\frac{1}{10})\}\Big]\\
        &\overset{\eqref{eq: Set B to C and Delta}}{\ge} 1 - 2(t^{2}+t)\mathrm{e}^{-\rho}.
    \end{align*}
    Then, there always exists $\rho_{\min,2} \ge 1$ such that for any $\rho \ge \rho_{\min,2}$, by \eqref{eq: parameters}, it follows that for all $t \le r - 1$,
    \begin{gather*}
        d = \frac{\sqrt{\alpha}}{40L_{Q_{\alpha}}^{H}\rho^{3/2}}
        \le \frac{1}{L_{Q_{\alpha}}^{H}\alpha r(10\sqrt{2\rho}+1)},
    \end{gather*}
    which yields
    \begin{multline}
    \label{eq: B to C at t+1}
        \quad\mathbb{P} \Big[\mathcal{E}_{B}^{0,t+1} \Rightarrow (\mathcal{E}_{C}^{0,t} \cap \{\|\hat{\Delta}_{1}^{t+1}\| \le \frac{\alpha \hat{\varsigma}^{0}(t)}{10}\}) \Big]\\
        \ge 1 - 2(t^{2}+t)\mathrm{e}^{-\rho} \ge 1 - 2(t+1)^{2}\mathrm{e}^{-\rho}.
    \end{multline}
    By 
    \begin{gather*}
        (\mathcal{E}_{C}^{0,t} \cap \{\|\hat{\Delta}_{1}^{t+1}\| \le \frac{\alpha \hat{\varsigma}^{0}(t)}{10}\}) \subset \mathcal{E}_{C}^{0,t+1},
    \end{gather*}
    it follows that
    \begin{gather*}
        (\mathcal{E}_{B}^{0,t+1} \Rightarrow (\mathcal{E}_{C}^{0,t} \cap \{\|\hat{\Delta}_{1}^{t+1}\| \le \frac{\alpha \hat{\varsigma}^{0}(t)}{10}\})) \subset (\mathcal{E}_{B}^{0,t+1} \Rightarrow \mathcal{E}_{C}^{0,t+1}).
    \end{gather*}
    thus, the claim \eqref{eq: B to C} holds at $t+1$ for $t \le r - 1$ and the induction is concluded. Applying Lemma \ref{lem: Localization} yields there exists threshold $\rho_{\min,3} \ge \max\{\rho_{\min,1},\rho_{\min,2}\}$ such that for any $\rho \ge \rho_{\min,3}$, 
    \begin{gather*}
        \mathbb{P} [\bar{\mathcal{E}}_{A}^{0,r} \Rightarrow \mathcal{E}_{B}^{0,r}] = \mathbb{P} [\mathcal{E}_{A}^{0,r} \cup \mathcal{E}_{B}^{0,r}] \ge 1 - 8r\mathrm{e}^{-\rho}.
    \end{gather*}
    Finally, since
    \begin{gather*}
        (\bar{\mathcal{E}}_{A}^{0,r} \Rightarrow \mathcal{E}_{B}^{0,r}) \cap (\mathcal{E}_{B}^{0,r} \Rightarrow \mathcal{E}_{C}^{0,r})
    \end{gather*}
    is a subset of $\bar{\mathcal{E}}_{A}^{0,r} \Rightarrow \mathcal{E}_{C}^{0,r}$, then, by \eqref{eq: Material implication} and \eqref{eq: Boole's inequality} in Proposition \ref{pro: Boole-frechet Inequality}, with \eqref{eq: B to C}, it follows that
    \begin{gather*}
        \mathbb{P} \Big[\mathcal{E}_{A}^{0,r} \cup \mathcal{E}_{C}^{0,r} \Big] = \mathbb{P} \Big[\bar{\mathcal{E}}_{A}^{0,r} \Rightarrow \mathcal{E}_{C}^{0,r} \Big] \ge 1 - (2r^{2} + 8r)\mathrm{e}^{-\rho}
    \end{gather*}
    as claimed.
\end{proof}

\begin{lemma}
\label{lem: Decrease for saddle by coupling}
    Let Assumptions \ref{ass: Lipschitz} and \ref{ass: Random perturbation} hold. There exists $\rho_{\min,5} \ge 1$ such that for any $\rho \ge \rho_{\min,5}$, with $\sigma$ and $\epsilon_{H}$ dependent on $\rho$ as defined in \eqref{eq: parameters}, the following holds: for any $t_{0}\ge0$ satisfying $\lambda_{\min}(\nabla^{2} Q_{\alpha}(\hat{\mathbf{x}}^{t_{0}}))<-\epsilon_{H}$,
    \begin{multline*}
        \mathbb{P} \Big[\exists \tau \in (0,r],~Q_{\alpha}(\hat{\mathbf{x}}^{t_{0}+\tau}) - Q_{\alpha}(\hat{\mathbf{x}}^{t_{0}}) \le -\ell_{s} \Big]\\
        \ge \frac{1}{3} - (r^{2} + 8r) \mathrm{e}^{-\rho},
    \end{multline*}
    where $\ell_{s}$ is defined in \eqref{eq: ls}.
\end{lemma}

\begin{proof}
    Due to the time-invariant updates, it suffices to prove for the special case $t_{0} = 0$. By Lemma \ref{lem: Dynamic concentration 1} and \ref{lem: Dynamic concentration 2}, there exists $\rho_{\min,3} \ge 1$ such that for any $\rho \ge \rho_{\min,3}$, it holds that if $\lambda_{\min}(\nabla^{2} Q_{\alpha}(\hat{\mathbf{x}}^{0})) < -\epsilon_{H}$, and $\{\hat{\mathbf{y}}^{k}\}$ and $\{\hat{\mathbf{z}}^{k}\}$ are coupling sequences, then 
    \begin{gather}
    \label{eq: Delta two}
        \mathbb{P} \Big[\|\hat{\Delta}_{2}^{r}\| \ge \frac{\alpha\hat{\varsigma}^{0}(r)}{5} \Big] \ge \frac{2}{3},
    \end{gather}
    and with $t_{0} = 0$,
    \begin{multline}
    \label{eq: Delta one}
        \mathbb{P} \left[\bar{\mathcal{E}}_{A}^{0,r} \Rightarrow \left\{\forall \tau \in (0,r],~\|\hat{\Delta}_{1}^{\tau}\| \le \frac{\alpha \hat{\varsigma}^{0}(\tau)}{10} \right\} \right]\\
        \overset{\eqref{eq: Event C}}{=} \mathbb{P} \Big[\bar{\mathcal{E}}_{A}^{0,r} \Rightarrow \mathcal{E}_{C}^{0,r} \Big] = \mathbb{P} \Big[\mathcal{E}_{A}^{0,r} \cup \mathcal{E}_{C}^{0,r} \Big] \ge 1 - (2r^{2} + 8r) \mathrm{e}^{-\rho}
    \end{multline}
    by Proposition \ref{pro: Material implication}. By definition $r = \lceil\alpha^{-3/2}\rho\rceil$ and $\sigma^{2} = \alpha d^{2}\rho^{-3}/(mn)$, and the fact $\lim_{x \rightarrow \infty} (1+\frac{1}{x})^{x} = \mathrm{e}$, there exists $\rho_{\min,4} \ge 1$ such that for any $\rho \ge \rho_{\min,4}$,
    \begin{gather*}
        \frac{\alpha(1+\alpha )^{r-1}\sigma}{10} \ge \frac{\alpha\sigma}{10} \mathrm{e}^{\rho} \overset{\eqref{eq: parameters}}{\ge} d.
    \end{gather*}
    Thus, by Lemma \ref{lem: Dynamic} and $\hat{\mathbf{y}}^{0} = \hat{\mathbf{z}}^{0}$, 
    \begin{multline*}
        (\mathcal{E}_{C}^{0,r} \cap \{\|\hat{\Delta}_{2}^{r}\| \ge \frac{\alpha\hat{\varsigma}^{0}(r)}{5}\})\\
        \subset \{\max\{\|\hat{\mathbf{y}}^{r}-\hat{\mathbf{y}}^{0}\|,~\|\hat{\mathbf{z}}^{r}-\hat{\mathbf{z}}^{0}\|\} \ge d\}
    \end{multline*}
    since
    \begin{multline*}
        \max\{\|\hat{\mathbf{y}}^{r}-\hat{\mathbf{y}}^{0}\|,~\|\hat{\mathbf{z}}^{r}-\hat{\mathbf{z}}^{0}\|\}
        \ge \frac{1}{2} \|\Delta^{r}\|\\
        \ge \frac{1}{2} (\|\hat{\Delta}_{2}^{r}\| - \|\hat{\Delta}_{1}^{r}\|) \ge \frac{\alpha \hat{\varsigma}^{0}(r)}{20} \overset{\eqref{eq: varsigma}}{\ge} \frac{\alpha\sigma}{10} \mathrm{e}^{\rho} \overset{\eqref{eq: parameters}}{\ge} d
    \end{multline*}
    by Cauchy--Schwarz inequality and $\Delta^{k} = \hat{\mathbf{y}}^{k} - \hat{\mathbf{z}}^{k}$. Thus,
    \begin{gather}
    \label{eq: Delta 2 and C to B}
        (\{\|\hat{\Delta}_{2}^{r}\| \ge \frac{\alpha\hat{\varsigma}^{0}(r)}{5}\} \cap \mathcal{E}_{C}^{0,r} ) \Rightarrow \bar{\mathcal{E}}_{B}^{0,r},
    \end{gather}
    where $\mathcal{E}_{B}^{0,r}$ is defined in \eqref{eq: Event B}. Thus, by \eqref{eq: Boole's inequality} in Proposition \ref{pro: Boole-frechet Inequality},
    \begin{align*}
        \mathbb{P} \Big[\bar{\mathcal{E}}_{A}^{0,r} \Rightarrow \bar{\mathcal{E}}_{B}^{0,r} \Big] &\ge \mathbb{P} \Big[ \bar{\mathcal{E}}_{A}^{0,r} \Rightarrow (\{\|\hat{\Delta}_{2}^{r}\| \ge \frac{\alpha\hat{\varsigma}^{0}(r)}{5}\} \cap \mathcal{E}_{C}^{0,r})\\
        &\quad\cap (\{\|\hat{\Delta}_{2}^{r}\| \ge \frac{\alpha\hat{\varsigma}^{0}(r)}{5}\} \cap \mathcal{E}_{C}^{0,r}) \Rightarrow \bar{\mathcal{E}}_{B}^{0,r} \Big]\\
        &\overset{\eqref{eq: Delta 2 and C to B}}{=} \mathbb{P} \Big[ \bar{\mathcal{E}}_{A}^{0,r} \Rightarrow (\{\|\hat{\Delta}_{2}^{r}\| \ge \frac{\alpha\hat{\varsigma}^{0}(r)}{5}\} \cap \mathcal{E}_{C}^{0,r}) \Big]\\
        &\ge \mathbb{P} \Big[  \{\|\hat{\Delta}_{2}^{r}\| \ge \frac{\alpha\hat{\varsigma}^{0}(r)}{5}\} \cap (\bar{\mathcal{E}}_{A}^{0,r} \Rightarrow \mathcal{E}_{C}^{0,r}) \Big]\\
        &\ge 1 - (1 - \mathbb{P} \Big[  \{\|\hat{\Delta}_{2}^{r}\| \ge \frac{\alpha\hat{\varsigma}^{0}(r)}{5}\} \Big])\\
        &\quad- (1 - \mathbb{P} \Big[\bar{\mathcal{E}}_{A}^{0,r} \Rightarrow \mathcal{E}_{C}^{0,r} \Big]).
    \end{align*}
    Therefore, in view of \eqref{eq: Delta two} and \eqref{eq: Delta one},
    \begin{gather}
    \label{eq: Not A to not B}
        \mathbb{P} \Big[\bar{\mathcal{E}}_{A}^{0,r} \Rightarrow \bar{\mathcal{E}}_{B}^{0,r} \Big] \ge \frac{2}{3} - (2r^{2} + 8r) \mathrm{e}^{-\rho}.
    \end{gather}
    Additionally, applying Lemma \ref{lem: Localization} and Proposition \ref{pro: Material implication} yields for $\ell_{s}$ defined in \eqref{eq: ls} and any $\rho \ge 1$,
    \begin{gather}
    \label{eq: Not A to B}
        \mathbb{P} \Big[\bar{\mathcal{E}}_{A}^{0,r} \Rightarrow
        \mathcal{E}_{B}^{0,r} \Big] = \mathbb{P} \Big[\mathcal{E}_{A}^{0,r} \cup
        \mathcal{E}_{B}^{0,r} \Big] \ge  1 - 8r\mathrm{e}^{-\rho}.
    \end{gather}
    In view of \eqref{eq: Not A to not B} and \eqref{eq: Not A to B}, by \eqref{eq: Boole's inequality} in Proposition \ref{pro: Boole-frechet Inequality},
    \begin{multline*}
        \mathbb{P} \Big[\mathcal{E}_{A}^{0,r} \Big] = \mathbb{P}\Big[\bar{\mathcal{E}}_{A}^{0,r} \Rightarrow \emptyset \Big] \ge \mathbb{P} \Big[(\bar{\mathcal{E}}_{A}^{0,r} \Rightarrow \bar{\mathcal{E}}_{B}^{0,r}) \cap (\bar{\mathcal{E}}_{A}^{0,r} \Rightarrow
        \mathcal{E}_{B}^{0,r})\Big]\\
        \ge \frac{2}{3} - (2r^{2} + 16r) \mathrm{e}^{-\rho},
    \end{multline*}
    which further means
    \begin{multline*}
        \mathbb{P} \Big[\exists \tau \in (0,r],~Q_{\alpha}(\hat{\mathbf{y}}^{\tau}) - Q_{\alpha}(\hat{\mathbf{y}}^{0})\le -\ell_{s} \Big]\\
        + \mathbb{P} \Big[\exists \tau \in (0,r],~Q_{\alpha}(\hat{\mathbf{z}}^{\tau}) - Q_{\alpha}(\hat{\mathbf{z}}^{0})\le -\ell_{s} \Big]\\
        \overset{\eqref{eq: Event A}}{\ge} \mathbb{P} \Big[\mathcal{E}_{A}^{0,r}\Big] \ge \frac{2}{3} - (2r^{2} + 16r) \mathrm{e}^{-\rho}.
    \end{multline*}
    Since $\hat{\mathbf{y}}^{k}$ and $\hat{\mathbf{z}}^{k}$ share the same randomness by Definition \ref{def: Coupling sequence}, by choosing $\rho_{\min,5} = \max\{\rho_{\min,3},\rho_{\min,4}\}$, it follows that for any $\rho \ge \rho_{\min,5}$,
    \begin{multline*}
        \mathbb{P} \Big[\exists \tau \in (0,r],~Q_{\alpha}(\hat{\mathbf{y}}^{\tau}) - Q_{\alpha}(\hat{\mathbf{y}}^{0}) \le -\ell_{s} \Big]\\
        =\mathbb{P} \Big[\exists \tau \in (0,r],~Q_{\alpha}(\hat{\mathbf{z}}^{\tau}) - Q_{\alpha}(\hat{\mathbf{z}}^{0}) \le -\ell_{s} \Big]\\
        \ge \frac{1}{3} - (r^{2} + 8r) \mathrm{e}^{-\rho}.
    \end{multline*}
    Without loss of generality, let $\hat{\mathbf{x}}^{k} = \hat{\mathbf{y}}^{k}$,
    \begin{multline*}
        \mathbb{P} \Big[\exists \tau \in (0,r],~Q_{\alpha}(\hat{\mathbf{x}}^{\tau}) - Q_{\alpha}(\hat{\mathbf{x}}^{0}) \le -\ell_{s} \Big]\\
        \ge \frac{1}{3} - (r^{2} + 8r) \mathrm{e}^{-\rho}
    \end{multline*}
    as claimed.
\end{proof}

Finally, we are prepared to prove Lemma \ref{lem: Approximate second order}, which demonstrates that by selecting a sufficiently large confidence parameter $\rho > 0$, after a sufficient number of iterations, at least one of the updates is an approximate second-order stationary point with high probability.

\begin{proofs}
    \textbf{Proof of Lemma \ref{lem: Approximate second order}}: Given $\rho \ge 1$, with $\alpha$ and $\epsilon_{g}$ dependent on $\rho$ as defined in \eqref{eq: parameters}, let
    \begin{gather}
    \label{eq: K}
        K = \lceil (Q_{\alpha} (\hat{\mathbf{x}}^{0}) - \sum_{i=1}^{m} f_{i}^{\star}) \alpha^{-4} \rho^{5} \rceil
    \end{gather}
    and define events
    \begin{gather*}
        \mathcal{E}_{1}^{k} \coloneqq \{ \|\nabla Q_{\alpha}(\hat{\mathbf{x}}^{k})\| \ge \epsilon_{g} \},\\
        \mathcal{E}_{2}^{k} \coloneqq \{ \|\nabla Q_{\alpha}(\hat{\mathbf{x}}^{k})\| < \epsilon_{g}, \lambda_{\min}(\nabla^{2} Q_{\alpha}(\hat{\mathbf{x}}^{k})) \le -\epsilon_{H}\},\\
        \mathcal{E}_{3}^{k} \coloneqq \{ \|\nabla Q_{\alpha}(\hat{\mathbf{x}}^{k})\| < \epsilon_{g}, \lambda_{\min}(\nabla^{2} Q_{\alpha}(\hat{\mathbf{x}}^{k})) > -\epsilon_{H}\}.
    \end{gather*}
    Let $\mathcal{P}$ denote the event that $\mathcal{E}_{1}^{\tau}$ or $\mathcal{E}_{2}^{\tau}$ occur fewer than $K$ iterations as
    \begin{gather}
    \label{eq: P}
        \mathcal{P} \coloneqq \left\{\sum_{\tau=0}^{K}\mathbf{1}_{\mathcal{E}_{1}^{\tau}} + \mathbf{1}_{\mathcal{E}_{2}^{\tau}} < K \right\} = \left\{\sum_{\tau=0}^{K}\mathbf{1}_{\mathcal{E}_{3}^{\tau}}\ge 1 \right\}.
    \end{gather}
    Note that establishing \eqref{eq: Approximate second order} is equivalent to proving that the event $\mathcal{E}_{3}^{k}$ occurs at least once within $K$ iterations with probability at least $1 - p$, which in turn is equivalent to showing that the event $\mathcal{P}$ holds with probability at least $1 - p$. Let
    \begin{gather*}
        \mathcal{P}_{1} = \left\{\sum_{\tau=0}^{K} \mathbf{1}_{\mathcal{E}_{1}^{\tau}} + \mathbf{1}_{\mathcal{E}_{2}^{\tau}} < K \quad\lor~~ \sum_{\tau=0}^{K} \mathbf{1}_{\mathcal{E}_{1}^{\tau}} < \frac{K}{2}\right\}
    \end{gather*}
    and
    \begin{gather*}
        \mathcal{P}_{2} = \left\{\sum_{\tau=0}^{K} \mathbf{1}_{\mathcal{E}_{1}^{\tau}} + \mathbf{1}_{\mathcal{E}_{2}^{\tau}} < K \quad\lor~~ \sum_{\tau=0}^{K} \mathbf{1}_{\mathcal{E}_{2}^{\tau}} < \frac{K}{2}\right\}.
    \end{gather*}
    \begin{gather}
    \label{eq: P1 and P2 to P}
        (\mathcal{P}_{1} \cap \mathcal{P}_{2}) \subset \mathcal{P}
    \end{gather}
    since
    \begin{gather*}
        \left(\sum_{\tau=0}^{K} \mathbf{1}_{\mathcal{E}_{1}^{\tau}} < \frac{K}{2}~~\land~~\sum_{\tau=0}^{K} \mathbf{1}_{\mathcal{E}_{2}^{\tau}} < \frac{K}{2}\right) \Rightarrow \sum_{\tau=0}^{K} \mathbf{1}_{\mathcal{E}_{1}^{\tau}} + \mathbf{1}_{\mathcal{E}_{2}^{\tau}} < K.
    \end{gather*}
    Next, we separately show that the probabilities of $\mathcal{P}_{1}$ and $\mathcal{P}_{2}$ are lower bounded. First, we employ proof by contradiction to prove the probability of $\mathcal{P}_{1}$ has a lower bound. To establish the contradiction, suppose $\mathcal{P}_{1}$ does not happen, i.e.,
    \begin{gather*}
        \bar{\mathcal{P}_{1}} = \left\{\sum_{\tau=0}^{K} \mathbf{1}_{\mathcal{E}_{1}^{\tau}} + \mathbf{1}_{\mathcal{E}_{2}^{\tau}} = K ~~\land~~ \sum_{\tau=0}^{K} \mathbf{1}_{\mathcal{E}_{1}^{\tau}} \ge \frac{K}{2} \right\}
    \end{gather*}
    happens. Applying Lemma \ref{lem: Decrease} yields for any $\rho \ge 1$, the states $\hat{\mathbf{x}}^{0}$ and $\hat{\mathbf{x}}^{K}$ generated by \eqref{eq: NDGD} satisfy
    \begin{gather}
    \label{eq: Total possible decrease}
        \mathbb{P} \left[Q_{\alpha} (\hat{\mathbf{x}}^{K}) - Q_{\alpha} (\hat{\mathbf{x}}^{0}) \le - L_{1}\right]\ge 1 - 2\mathrm{e}^{-\rho},
    \end{gather}
    where
    \begin{gather*}
        L_{1} = \frac{1}{4} \cdot \alpha \epsilon_{g}^{2} K -   mn \alpha \sigma^{2} (K+ \sqrt{K\rho} + \rho).
    \end{gather*}
    Note that we have more than $K/2$ iterates for which gradient is large. By definition in \eqref{eq: parameters}, it follows that
    \begin{gather*}
        L_{1} \ge \alpha K(\frac{1}{4} \epsilon_{g}^{2} - 3 \alpha d^{2}\rho^{-3}) \ge \alpha^{2} K(\frac{1}{4} - \frac{3 }{200(L_{F}^{H})^{2}}\alpha \rho^{-6}).
    \end{gather*}
    It holds that there exists $\rho_{\min,6} \ge 1$ such that for any $\rho \ge \rho_{\min,6}$,
    \begin{gather*}
        K > 8 \frac{Q_{\alpha} (\hat{\mathbf{x}}^{0}) - Q_{\alpha}^{\star}}{\alpha^{2}},
    \end{gather*}
    as in \eqref{eq: K} yields
    \begin{gather*}
        L_{1} \ge \frac{1}{8} \alpha^{2} K > Q_{\alpha} (\hat{\mathbf{x}}^{0}) - Q_{\alpha}^{\star},
    \end{gather*}
    which implies that
    \begin{gather*}
        \mathbb{P} \Big[Q_{\alpha} (\hat{\mathbf{x}}^{K}) \le Q_{\alpha}^{\star}\Big] \ge 1 - 2\mathrm{e}^{-\rho}.
    \end{gather*}
    Hence, $\bar{\mathcal{P}_{1}}$ happening leads to a contradiction with certain probability since $Q_{\alpha} (\hat{\mathbf{x}}^{K}) \ge Q_{\alpha}^{\star}$ almost surely. We therefore conclude that $\mathcal{P}_{1}$ happens with probability at least $1 - 2\mathrm{e}^{-\rho}$, i.e.,
    \begin{gather}
    \label{eq: Contradiction 1}
        \mathbb{P} \left[\mathcal{P}_{1}\right] \ge 1 - 2\mathrm{e}^{-\rho}.
    \end{gather}
    Next, we also employ proof by contradiction to prove the probability of $\mathcal{P}_{2}$ has a lower bound. To establish the contradiction, suppose $\mathcal{P}_{2}$ does not happen, i.e.,
    \begin{gather*}
        \bar{\mathcal{P}_{2}} = \left\{\sum_{\tau=0}^{K} \mathbf{1}_{\mathcal{E}_{1}^{\tau}} + \mathbf{1}_{\mathcal{E}_{2}^{\tau}} = K ~~\land~~ \sum_{\tau=0}^{K} \mathbf{1}_{\mathcal{E}_{2}^{\tau}} \ge \frac{K}{2} \right\}.
    \end{gather*}
    Let
    \begin{gather}
    \label{eq: Lg}
        \begin{gathered}
            \ell_{g}(t) \coloneqq   mn \alpha \sigma^{2} (t+ \sqrt{t\rho} + \rho),
        \end{gathered}
    \end{gather}
    and recall $\ell_{s} = d^{2}/(4\alpha r) - 2  mn \alpha \sigma^{2}(r + \sqrt{r\rho} + \rho)$ in \eqref{eq: ls}. Applying Lemma \ref{lem: Decrease} yields for any $\tau,~t \ge 0$, and $\rho \ge 1$,
    \begin{gather*}
        \mathbb{P} \Big[Q_{\alpha} (\hat{\mathbf{x}}^{\tau+t}) - Q_{\alpha} (\hat{\mathbf{x}}^{\tau}) \le \ell_{g}(t)\Big] \ge 1 - 2\mathrm{e}^{-\rho}.
    \end{gather*}
    By definition $\sigma^{2} = \alpha d^{2}\rho^{-3}/(mn)$ as in \eqref{eq: parameters}, for any $t \ge \rho$,
    \begin{gather*}
        0 < \ell_{g}(t) =   mn \alpha \sigma^{2} (t+ \sqrt{t\rho} + \rho) \le 3 \alpha^{2}d^{2}\rho^{-3}t,
    \end{gather*}
    and
    \begin{gather*}
        \ell_{s} = \frac{d^{2}}{4 \alpha r} - 2  mn \alpha \sigma^{2}(r + \sqrt{r\rho} + \rho) \ge \frac{d^{2}}{4\alpha r} - \frac{6\alpha^{2} r d^{2}}{\rho^{3}}.
    \end{gather*}
    As such, by definition in \eqref{eq: parameters}, there exists $\rho_{\min,7} \ge 1$ such that for any $\rho \ge \rho_{\min,7}$, 
    \begin{gather*}
        \ell_{s} \ge \frac{d^{2}}{\alpha r}(\frac{1}{4} - \frac{9}{2}\alpha^{3}(\alpha^{-\tfrac{3}{2}}\rho+1)^{2}\rho^{-3}) \ge \frac{d^{2}}{8\alpha r} > 0.
    \end{gather*}
    Further, applying Lemma \ref{lem: Decrease for saddle by coupling} yields that there exists $\rho_{\min,8} \ge \max\{\rho_{\min,5},\rho_{\min,7}\}$ such that for any $\rho \ge \rho_{\min,8}$, the following two claims hold:\par
    \textit{Claim 1}: $\ell_{g}(t)$ defined in \eqref{eq: Lg} is upper bounded by
    \begin{gather*}
        0 < \ell_{g}(t) \le 3 \alpha^{2}d^{2}\rho^{-3}t,
    \end{gather*}
    and for any $t \ge \rho$,
    \begin{gather*}
        \mathbb{P} \left[Q_{\alpha} (\hat{\mathbf{x}}^{t}) - Q_{\alpha} (\hat{\mathbf{x}}^{0}) \le \ell_{g}(t) \right] \ge 1 - 2\mathrm{e}^{-\rho}.
    \end{gather*}\par
    \textit{Claim 2}: $\ell_{s}$ defined in \eqref{eq: ls} is lower bounded by
        \begin{gather*}
            \ell_{s} \ge \frac{1}{8}\alpha^{-1}r^{-1}d^{2} > 0,
        \end{gather*}
        and for any $t_{0} \ge 0$ with $r$ defined in \eqref{eq: parameters}, if $\lambda_{\min}(\nabla^{2} Q_{\alpha}(\hat{\mathbf{x}}^{t_{0}})) < -\epsilon_{H}$, then
        \begin{multline*}
            \mathbb{P} \left[\exists \tau \in (0,r],~Q_{\alpha}(\hat{\mathbf{x}}^{t_{0}+\tau}) - Q_{\alpha}(\hat{\mathbf{x}}^{t_{0}}) \le - \ell_{s} \right]\\
            \ge \frac{1}{3} - (r^{2} + 8r) \mathrm{e}^{-\rho}.
        \end{multline*}
    Define stochastic process $\{\eta_{i}\} \subset [0,K]$ as 
    \begin{align}
    \label{eq: stochastic process}
        \eta_{i} \coloneqq \begin{cases}0, & i=0\\ 
        \eta_{i-1} + 1, & \mathbf{1}_{\mathcal{E}_{2}^{\eta_{i-1}}} = 0\\
        \eta_{i-1} + \tau_{i-1}, & \mathbf{1}_{\mathcal{E}_{2}^{\eta_{i-1}}} = 1\end{cases},
    \end{align}
    where, whenever $\mathbf{1}_{\mathcal{E}_{2}^{\eta_{i-1}}}=1$, define
    \begin{multline*}
        \tau_{i-1} \coloneqq \min \Bigg\{ \tau \in (0,r]:~\mathbb{P} \Big[ Q_{\alpha}(\hat{\mathbf{x}}^{\eta_{i-1}+\tau}) - Q_{\alpha}(\hat{\mathbf{x}}^{\eta_{i-1}}) \le -\ell_{s} \Big]\\
        \ge \frac{1}{3} - (r^{2} + 8r) \mathrm{e}^{-\rho}\Bigg\}.
    \end{multline*}
    Let the last index before $K$ be
    \begin{gather}
    \label{eq: Kprime}
        K^{\prime} \coloneqq \max \left\{\eta_{i}:~ \eta_{i} \in [0,K] \right\},
    \end{gather}
    which satisfies that $K-r \le K^{\prime} \le K$. Let
    \begin{gather*}
        d^{\eta_{i}} = Q_{\alpha}(\hat{\mathbf{x}}^{\eta_{i+1}}) - Q_{\alpha}(\hat{\mathbf{x}}^{\eta_{i}}).
    \end{gather*}
    Defining
    \begin{multline}
    \label{eq: Event 4}
        \mathcal{E}_{4} \coloneqq \{Q_{\alpha} (\hat{\mathbf{x}}^{K^{\prime}}) - Q_{\alpha} (\hat{\mathbf{x}}^{0}) \le \ell_{g}(K^{\prime})~~\land\\
        \forall \eta_{i}\in \{\tau:~ \mathbf{1}_{\mathcal{E}_{2}^{\tau}} = 1\},~d^{\eta_{i}} \le \ell_{g}(r)\},
    \end{multline}
    by \eqref{eq: Boole's inequality} in Proposition \ref{pro: Boole-frechet Inequality}, in view of Claim 1, it holds that 
    \begin{gather}
    \label{eq: Probability of Event 4}
        \mathbb{P} [\mathcal{E}_{4}] \ge 1 - 2(K^{\prime}+1)e^{-\rho}.
    \end{gather}
    Next, we define
    \begin{gather}
    \label{eq: Event 5 6 7}
        \begin{gathered}
            \mathcal{E}_{5}^{\eta_{i}} \coloneqq \mathcal{E}_{4} \cap\{d^{\eta_{i}} \le - \ell_{s}\},\\
            \mathcal{E}_{6}^{\eta_{i}} \coloneqq \mathcal{E}_{4}\cap \{- \ell_{s} < d^{\eta_{i}} < 0\},\\
            \mathcal{E}_{7}^{\eta_{i}} \coloneqq \mathcal{E}_{4} \cap \{d^{\eta_{i}} \ge 0\}.            
        \end{gathered}
    \end{gather}
    Note that $\mathcal{E}_{5}^{\eta_{i}} \cup \mathcal{E}_{6}^{\eta_{i}} \cup \mathcal{E}_{7}^{\eta_{i}} = \mathcal{E}_{4}$ for any $\eta_{i}$ defined in \eqref{eq: stochastic process}. Then, for any $\eta_{i} $ defined in \eqref{eq: stochastic process} satisfying $\mathbf{1}_{\mathcal{E}_{2}^{\eta_{i}}} = 1$, by \eqref{eq: Boole's inequality} in Proposition \ref{pro: Boole-frechet Inequality}, \eqref{eq: Probability of Event 4} and \textit{Claim 2} yield 
    \begin{align*}
        \mathbb{P}[\mathcal{E}_{5}^{\eta_{i}}] &\ge 1 - (1- \mathbb{P} [\mathcal{E}_{4}]) - (1- \mathbb{P} [d^{\eta_{i}} \le - \ell_{s}])\\
        &\ge \frac{1}{3} - (r^{2} + 8r) \mathrm{e}^{-\rho} - 2(K^{\prime}+1)e^{-\rho},
    \end{align*}
    and thus, there exists $\rho_{\min,9} \ge 1$ such that for any $\rho \ge \rho_{\min,9}$,
    \begin{align}
    \label{eq: Conditional probability of Event 5}
        \begin{aligned}
            \mathbb{P}[\mathcal{E}_{5}^{\eta_{i}}\mid\mathcal{E}_{4}] &= \frac{\mathbb{P}[\mathcal{E}_{5}^{\eta_{i}} \cap \mathcal{E}_{4}]}{\mathbb{P}[\mathcal{E}_{4}]} = \frac{\mathbb{P}[\mathcal{E}_{5}^{\eta_{i}}]}{\mathbb{P}[\mathcal{E}_{4}]} \ge \mathbb{P}[\mathcal{E}_{5}^{\eta_{i}}]\\
            & \ge \frac{1}{3} - (r^{2} + 8r) \mathrm{e}^{-\rho} - 2(K^{\prime}+1)e^{-\rho}.
        \end{aligned}
    \end{align}
    Applying law of total expectation, by \textit{Claim 2}, yields
    \begin{align}
    \label{eq: Total expectation 1}
        \begin{aligned}
            &\mathbb{E}\Bigg [\sum_{\eta_{i} \in \{\eta_{i}\} \cap \{\tau:~\mathbf{1}_{\mathcal{E}_{2}^{\tau}} = 1\}} Q_{\alpha}(\hat{\mathbf{x}}^{\eta_{i+1}}) - Q_{\alpha}(\hat{\mathbf{x}}^{\eta_{i}})\mid\mathcal{E}_{4}\Bigg]\\
            &\quad= \sum_{\eta_{i} \in \{\eta_{i}\} \cap \{\tau:~\mathbf{1}_{\mathcal{E}_{2}^{\tau}} = 1\}} \Big(\mathbb{E} \left[d^{\eta_{i}}\mid\mathcal{E}_{5}^{\eta_{i}} \right] \cdot \mathbb{P}\left [\mathcal{E}_{5}^{\eta_{i}}\mid\mathcal{E}_{4}\right]\\
            &\quad+ \mathbb{E}\left[d^{\eta_{i}}\mid \mathcal{E}_{6}^{\eta_{i}}\right] \cdot \mathbb{P}\left[\mathcal{E}_{6}^{\eta_{i}}\mid\mathcal{E}_{4}\right] + \mathbb{E}\left[d^{\eta_{i}}\mid \mathcal{E}_{7}^{\eta_{i}}\right] \cdot \mathbb{P}\left[\mathcal{E}_{7}^{\eta_{i}}\mid\mathcal{E}_{4}\right] \Big )\\
            &\quad\overset{\eqref{eq: Conditional probability of Event 5}}{\le}
            \sum_{\eta_{i} \in \{\eta_{i}\} \cap \{\tau:~\mathbf{1}_{\mathcal{E}_{2}^{\tau}} = 1\}} \mathbb{E}\left[d^{\eta_{i}}\mid\mathcal{E}_{7}^{\eta_{i}}\right] - \left(\sum_{\eta_{i} \in \{\eta_{i}\}} \mathbf{1}_{\mathcal{E}_{2}^{\eta_{i}}}\right) \cdot\\
            &\quad\quad\ell_{s} \cdot \left(\frac{1}{3} - (r^{2} + 8r) \mathrm{e}^{-\rho} - 2(K^{\prime}+1)e^{-\rho}\right),
        \end{aligned}
    \end{align}
    and
    \begin{align}
    \label{eq: Total expectation 2}
        \begin{aligned}
            &\mathbb{E}\Bigg [\sum_{\eta_{i} \in \{\eta_{i}\} \cap \{\tau:~\mathbf{1}_{\mathcal{E}_{1}^{\tau}} = 1\}} Q_{\alpha}(\hat{\mathbf{x}}^{\eta_{i+1}}) - Q_{\alpha}(\hat{\mathbf{x}}^{\eta_{i}})\mid\mathcal{E}_{4}\Bigg]\\
            &\quad= \sum_{\eta_{i} \in \{\eta_{i}\} \cap \{\tau:~\mathbf{1}_{\mathcal{E}_{1}^{\tau}} = 1\}} \Big(\mathbb{E} \left[d^{\eta_{i}}\mid\mathcal{E}_{5}^{\eta_{i}} \right] \cdot \mathbb{P}\left [\mathcal{E}_{5}^{\eta_{i}}\mid\mathcal{E}_{4}\right]\\
            &\quad+ \mathbb{E}\left[d^{\eta_{i}}\mid \mathcal{E}_{6}^{\eta_{i}}\right] \cdot \mathbb{P}\left[\mathcal{E}_{6}^{\eta_{i}}\mid\mathcal{E}_{4}\right] + \mathbb{E}\left[d^{\eta_{i}}\mid \mathcal{E}_{7}^{\eta_{i}}\right] \cdot \mathbb{P}\left[\mathcal{E}_{7}^{\eta_{i}}\mid\mathcal{E}_{4}\right] \Big )\\
            &\le
            \sum_{\eta_{i} \in \{\tau:~\mathbf{1}_{\mathcal{E}_{1}^{\tau}} = 1\}} \mathbb{E}\left[d^{\eta_{i}}\mid\mathcal{E}_{7}^{\eta_{i}}\right],
        \end{aligned}
    \end{align}
    where the last two inequalities are implied by the fact that $\mathbb{E}\left[d^{\eta_{i}}\mid \mathcal{E}_{5}^{\eta_{i}}\right] \le -\ell_{s} < 0$, $\mathbb{E}\left[d^{\eta_{i}}\mid \mathcal{E}_{6}^{\eta_{i}}\right] < 0$ and $0 \le \mathbb{P}\left[\mathcal{E}_{7}^{\eta_{i}}\mid\mathcal{E}_{4}\right] \le 1$ by \eqref{eq: Event 5 6 7}. Then, summing \eqref{eq: Total expectation 1} and \eqref{eq: Total expectation 2} over $\eta_{i}$ yields that given $\bar{\mathcal{P}_{2}}$ happening,
    \begin{multline}
    \label{eq: Conditional decrease}
        \mathbb{E}\left[Q_{\alpha}(\hat{\mathbf{x}}^{K^{\prime}}) - Q_{\alpha}(\hat{\mathbf{x}}^{0})\mid\mathcal{E}_{4}\right]\\
        = \mathbb{E}\Bigg [\sum_{\eta_{i} \in \{\eta_{i}\} \cap \{\tau:~\mathbf{1}_{\mathcal{E}_{1}^{\tau}} = 1\}} Q_{\alpha}(\hat{\mathbf{x}}^{\eta_{i+1}}) - Q_{\alpha}(\hat{\mathbf{x}}^{\eta_{i}})\mid\mathcal{E}_{4}\Bigg]\\
        + \mathbb{E}\Bigg [\sum_{\eta_{i} \in \{\eta_{i}\} \cap \{\tau:~\mathbf{1}_{\mathcal{E}_{2}^{\tau}} = 1\}} Q_{\alpha}(\hat{\mathbf{x}}^{\eta_{i+1}}) - Q_{\alpha}(\hat{\mathbf{x}}^{\eta_{i}})\mid\mathcal{E}_{4}\Bigg]\\
        \le -L_{2},
    \end{multline}
    where
    \begin{align*}
        L_{2} &= \left(\sum_{\eta_{i} \in \{\eta_{i}\}} \mathbf{1}_{\mathcal{E}_{2}^{\eta_{i}}}\right) \cdot\ell_{s} \cdot \left(\frac{1}{3} - (r^{2} + 8r) \mathrm{e}^{-\rho} - 2(K^{\prime}+1)e^{-\rho}\right)\\
        &\quad- \mathbb{E}\left[Q_{\alpha} (\hat{\mathbf{x}}^{K^{\prime}}) - Q_{\alpha} (\hat{\mathbf{x}}^{0}) \mid(\bigcap_{\kappa \in \{\eta_{i}\}}\mathcal{E}_{7}^{\kappa})\right]\\
        &\overset{\eqref{eq: Event 4}}{\ge} \frac{K^{\prime}}{r} \cdot\ell_{s} \cdot \left(\frac{1}{3} - (r^{2} + 8r) \mathrm{e}^{-\rho} - 2(K^{\prime}+1)e^{-\rho}\right) -\ell_{g}(K^{\prime}).
    \end{align*}
    By \textit{Claim 2} and $r$ defined in \eqref{eq: parameters}, it follows that
    \begin{gather}
    \label{eq: r/ls}
        r/\ell_{s} \le 12800 (L_{F}^{H})^{2}\alpha^{-3}\rho^{5}.
    \end{gather}
    As $K$ and $K^{\prime}$ defined in \eqref{eq: K} and \eqref{eq: Kprime}, there exists $\rho_{\min,9} \ge 1$ such that for any $\rho \ge \rho_{\min,9}$,
    \begin{gather*}
         K^{\prime} \ge K-r > 16\frac{r}{\ell_{s}} \cdot (Q_{\alpha} (\hat{\mathbf{x}}^{0}) - Q_{\alpha}^{\star}),
    \end{gather*}
    and
    \begin{align*}
        L_{2} &\ge \frac{K^{\prime}}{r} \cdot \ell_{s} \cdot \left(\frac{1}{3} - (r^{2} + 8r) \mathrm{e}^{-\rho} - 2(K^{\prime}+1)e^{-\rho}\right) - \ell_{g}(K^{\prime})\\
        &\ge \frac{K^{\prime}}{4r}\ell_{s} - 3mn \alpha\sigma^{2}K^{\prime} \overset{\eqref{eq: r/ls}}{\ge} \frac{K^{\prime}}{8r}\ell_{s} > Q_{\alpha} (\hat{\mathbf{x}}^{0}) - Q_{\alpha}^{\star}
    \end{align*}
    with $\alpha$ as per \eqref{eq: parameters}. Thus, \eqref{eq: Conditional decrease} implies
    \begin{gather*}
        \mathbb{E}\left[Q_{\alpha} (\hat{\mathbf{x}}^{K^{\prime}})\mid\mathcal{E}_{4}\right]  < Q_{\alpha}^{\star}.
    \end{gather*}
    Since given $\mathcal{E}_{4}$, $Q_{\alpha} (\hat{\mathbf{x}}^{K}) \ge Q_{\alpha}^{\star}$ almost surely, then $\bar{\mathcal{P}_{2}}$ happening leads to a contradiction with certain probability. We therefore conclude that $\mathcal{P}_{2}$ happens with probability at least $1 - 2(K+1)e^{-\rho}$, i.e.,
    \begin{gather}
    \label{eq: Contradiction 2}
        \mathbb{P} \left[\mathcal{P}_{2}\right] \ge 1 - 2(K^{\prime}+2)e^{-\rho} \ge 1 - 2(K+1)e^{-\rho}.
    \end{gather}
    By \eqref{eq: Boole's inequality} in Proposition \ref{pro: Boole-frechet Inequality}, combining \eqref{eq: Contradiction 1} and \eqref{eq: Contradiction 2}, it follows that
    \begin{multline*}
         \mathbb{P} \left[\sum_{\tau=0}^{K} \mathbf{1}_{\mathcal{E}_{3}^{\tau}} \ge 1\right] \overset{\eqref{eq: P}}{=} \mathbb{P} \left[\mathcal{P}\right] \overset{\eqref{eq: P1 and P2 to P}}{\ge} \mathbb{P} \left[\mathcal{P}_{1} \cap \mathcal{P}_{2}\right]\\
         \ge 1 - (1-\mathbb{P} \left[\mathcal{P}_{1}\right]) - (1-\mathbb{P} \left[\mathcal{P}_{2}\right]).
    \end{multline*}
    In view of \eqref{eq: Contradiction 1} and \eqref{eq: Contradiction 2}, it holds that after $K$ iterations as defined in \eqref{eq: K},
    \begin{multline}
    \label{eq: second order tau}
        \mathbb{P} \Big[\exists \tau \in (0,K],~\|\nabla Q_{\alpha}(\hat{\mathbf{x}}^{\tau})\| < \epsilon_{g}\\
        \land~~\lambda_{\min}(\nabla^{2} Q_{\alpha}(\hat{\mathbf{x}}^{\tau})) > -\epsilon_{H} \Big]\\
        \ge 1 - 2(K+2)\mathrm{e}^{-\rho}.
    \end{multline}
    Let
    \begin{gather*}
        \mathcal{E}_{e,\tau} = \{\|\nabla Q_{\alpha}(\hat{\mathbf{x}}^{\tau})\| < \epsilon_{g}~~\land~~\lambda_{\min}(\nabla^{2} Q_{\alpha}(\hat{\mathbf{x}}^{\tau})) > -\epsilon_{H}\}
    \end{gather*}
    and
    \begin{gather*}
        \mathcal{E}_{f,\tau} = \{\| \hat{\mathbf{x}}^{\tau} - \bm{1}_{m} \otimes \operatorname{av}(\hat{\mathbf{x}}^{\tau})\| \le \zeta\},
    \end{gather*}
    where $\zeta$ is defined in \eqref{eq: Consensus error}. Again, in view of Lemma \ref{lem: Consensus error} and \eqref{eq: second order tau}, applying \eqref{eq: Boole's inequality} in Proposition \ref{pro: Boole-frechet Inequality}, it holds that there exists $\rho_{2} \ge \max\{\rho_{\min,6} ,\rho_{\min,8}, \rho_{\min,9},\rho_{1}\}$ such that for any $\rho \ge \rho_{2}$, after $K$ iterations, with the initial condition $\|\hat{\mathbf{x}}^{0} - \bm{1}_{m} \otimes \operatorname{av}(\hat{\mathbf{x}}^{0})\| = 0$,
    \begin{align}
    \label{eq: Approximately second order staionarity with consensus}
        \begin{aligned}
            &\quad\mathbb{P} \Big[\exists \tau \in (0,K],~\hat{\mathbf{x}}^{\tau} \in \mathcal{E}_{e,\tau} \cap\mathcal{E}_{f,\tau}\Big]\\
            &=\mathbb{P} \Big[\bigcup_{\tau=1}^{K}(\mathcal{E}_{e,\tau} \cap\mathcal{E}_{f,\tau})\Big]\\
            &\ge \mathbb{P} \Big[(\bigcup_{\tau=1}^{K}\mathcal{E}_{e,\tau}) \cap (\bigcap_{\tau=1}^{K}\mathcal{E}_{f,\tau})\Big]\\
            &\ge 1 - (1 - \mathbb{P} \Big[\bigcup_{\tau=1}^{K}\mathcal{E}_{e,\tau}\Big]) - (1 - \mathbb{P} \Big[\bigcap_{\tau=1}^{K}\mathcal{E}_{f,\tau}\Big])\\
            &\ge 1 - (1 - \mathbb{P} \Big[\exists \tau \in (0,K],~\hat{\mathbf{x}}^{\tau} \in \mathcal{E}_{e,\tau}\Big]) - \sum_{\tau=1}^{K} (1 - \mathbb{P} \Big[\mathcal{E}_{f,\tau}\Big])\\
            &\ge 1 - 2(2 + K + K^{2})\mathrm{e}^{-\rho} \ge 1 - 4K^{2}\mathrm{e}^{-\rho}
        \end{aligned}
    \end{align}
    as claimed.
\end{proofs}

\subsection{Proof of Lemma \ref{lem: Stay close to minimizers}}
\label{sec: Proof of Lemma Stay close to minimizers}
The key idea of the proof is based on a relaxed version of the supermartingale concentration inequality from \cite[Theorem 7.3]{chung2006concentration} (see Proposition \ref{pro: Supermartingale concentration}), which cannot be directly applied to analyze \textbf{NDGD} because the noise term is unbounded. The following result is a relaxed version of Proposition \ref{pro: Supermartingale concentration}, in which the conditional variance is taken to be bounded by a constant and the third bound is relaxed to handle unbounded noise.
\begin{proposition}[Relaxed Azuma's inequality]
\label{pro: Relaxed Azuma}
    Given a filtration $\mathcal{F}$
    \begin{gather*}
        \{\emptyset,~\Omega\} = \mathcal{F}^{0} \subset \mathcal{F}^{1} \subset \cdots \subset \mathcal{F}^{t} = \mathcal{F},
    \end{gather*}
    for $1 \le k \le t$, suppose a random variable $X^{k}$ is $\mathcal{F}^{k}$-measurable, satisfying the following conditions:
    \begin{enumerate}
        \item $ \mathbb{E}[X^{k} \mid \mathcal{F}^{k-1}] \leq X^{k-1}$ \text{almost surely};
        \item $\operatorname{Var}(X^{k} \mid \mathcal{F}^{k-1}) \leq \sigma_{k}^2$
        \item $\mathbb{P}\left(X^{k} - \mathbb{E}[X^{k} \mid \mathcal{F}^{k-1}] > a_{k} + M\right) \leq p_{k}.$
    \end{enumerate}
    Here, $\sigma_{k}$, $a_{k}$ and $M$ are given nonnegative constants, and $0 \le p_{k} \le 1$. Then, for any $\lambda > 0$,
    \[
    \mathbb{P}(X^{t} \geq X^{0} + \lambda) \leq \exp\left(-\frac{\lambda^2}{2\left(\sum_{k=1}^{t}(\sigma_{k}^2 + a_{k}^2) + \frac{M\lambda}{3}\right)}\right) + \sum_{k=1}^{t} p_{k}.
    \]
\end{proposition}

\begin{proof}
    Define 
    \begin{gather*}
        W^{k} \coloneqq \mathbb{E}[X^{k} \mid \mathcal{F}^{k-1}] - X^{k-1}, \quad 1 \leq k \leq t,\\
        Y^{k} \coloneqq \min\left\{X^{k} - \mathbb{E}[X^{k} \mid \mathcal{F}^{k-1}],~ a_{k} + M\right\}, \quad 1 \leq k \leq t.
    \end{gather*}
    Then, we have $W^{k} \le 0$, $Y^{k} \leq a_{k} + M$ and
    \begin{gather*}
        (Y^{k})^{2} \le (X^{k} - \mathbb{E}[X^{k} \mid \mathcal{F}^{k-1}])^{2}
    \end{gather*}
    almost surely. It follows that
    \begin{gather*}
        \mathbb{E}[W^{k} \mid \mathcal{F}^{k-1}] = W^{k} \le 0,\\
        \mathbb{E}[Y^{k} \mid \mathcal{F}^{k-1}] \le \mathbb{E}[X^{k} - \mathbb{E}[X^{k} \mid \mathcal{F}^{k-1}] \mid \mathcal{F}^{k-1}] = 0.
    \end{gather*} 
    Let $Z^{0} = X^{0}$ and $Z^{k} \coloneqq X^{0} + \sum_{\kappa=1}^{k} (W^{\kappa} + Y^{\kappa}) = Z^{k-1} + W^{k} + Y^{k}$, for $k \ge  1$, then,
    \begin{gather*}
        \mathbb{E}[Z^{k} \mid \mathcal{F}^{k-1}] = Z^{k-1} + \mathbb{E}[W^{k} \mid \mathcal{F}^{k-1}] + \mathbb{E}[Y^{k} \mid \mathcal{F}^{k-1}] \leq Z^{k-1}
    \end{gather*}
    with probability $1$. Further,
    \begin{multline*}
        Z^{k} - \mathbb{E}[Z^{k} \mid \mathcal{F}^{k-1}] = Z^{k-1} + W^{k} + Y^{k} - \mathbb{E}[Z^{k} \mid \mathcal{F}^{k-1}]\\
        = Y^{k} - \mathbb{E}[Y^{k} \mid \mathcal{F}^{k-1}] \le a_{k} + M
    \end{multline*}
    and 
    \begin{align*}
        \operatorname{Var}[Z^{k} \mid \mathcal{F}^{k-1}] &= \operatorname{Var}[Y^{k} \mid \mathcal{F}^{k-1}]\\
        &= \mathbb{E}[(Y^{k})^{2} \mid \mathcal{F}^{k-1}] - (\mathbb{E}[Y^{k} \mid \mathcal{F}^{k-1}])^{2}\\
        &\le \mathbb{E}[(Y^{k})^{2} \mid \mathcal{F}^{k-1}]\\
        &\le \mathbb{E}[(X^{k} - \mathbb{E}[X^{k} \mid \mathcal{F}^{k-1}])^{2} \mid \mathcal{F}^{k-1}]\\
        &=\operatorname{Var}[X^{k} \mid \mathcal{F}^{k-1}] \le \sigma_{k}^2.
    \end{align*}
    Applying Proposition \ref{pro: Supermartingale concentration} with $\phi_{k} = 0$ for all $1 \le k \le t$ yields
    \begin{gather*}
       \mathbb{P}(Z^{t} - X^{0} \geq \lambda) \leq \exp\left(-\frac{\lambda^2}{2\left(\sum_{k=1}^{t}(\sigma_{k}^2 + a_{k}^2) + \frac{M\lambda}{3}\right)}\right). 
    \end{gather*}
    Since
    \begin{align*}
        X^{t} - Z^{t} &= X^{0} + \sum_{k=1}^{t} (X^{k} - X^{k-1}) - (X^{0} + \sum_{k=1}^{t} (W^{k} + Y^{k}))\\
        &= \sum_{k=1}^{t} (X^{k} - \mathbb{E}[X^{k} \mid \mathcal{F}^{k-1}]) + \sum_{k=1}^{t}(\mathbb{E}[X^{k} \mid \mathcal{F}^{k-1}] -X^{k-1})\\
        &\quad- \sum_{k=1}^{t} (W^{k} + Y^{k})\\
        &= \sum_{k=1}^{t} (X^{k} - \mathbb{E}[X^{k} \mid \mathcal{F}^{k-1}]) - Y^{k},
    \end{align*}
    $X^{t} - Z^{t} \ne 0$ only if the saturation of $Y^{k}$ occurs at least once. Using Boole’s inequality yields
    \begin{gather*}
        \mathbb{P}(X^{t} \neq Z^{t}) \leq \sum_{k=1}^{t} \mathbb{P}(X^{k} - \mathbb{E}[X^{k} \mid \mathcal{F}^{k-1}] > a_{k} + M) \leq \sum_{k=1}^{t} p_{k}.
    \end{gather*}
    Thus, combining these two results via Boole’s inequality, since $X^{t} \ge Z^{t}$, it follows that
    \begin{align*}
        \mathbb{P}(X^{t} \geq X^{0} + \lambda) &\leq \mathbb{P}(\{Z^{t} \geq X^{0} + \lambda\} \cap \{X^{t} = Z^{t}\})\\
        &\quad+ \mathbb{P}(\{X^{t} \geq X^{0} + \lambda\} \cap \{X^{t} \neq Z^{t}\})\\
        &\le \mathbb{P}(Z^{t} \geq X^{0} + \lambda) + \mathbb{P}(X^{t} \neq Z^{t}),
    \end{align*}
    which gives
    \begin{gather*}
        \mathbb{P}(X^{t} \geq X^{0} + \lambda) \leq \exp\left(-\frac{\lambda^2}{2\left(\sum_{k=1}^{t}(\sigma_{k}^2 + a_{k}^2) + \frac{M\lambda}{3}\right)}\right) + \sum_{k=1}^{t} p_{k}.
    \end{gather*}
\end{proof}

\begin{proofs}
    \textbf{Proof of Lemma \ref{lem: Stay close to minimizers}}: Satisfying i) and ii) is a matter of having sufficiently large $\rho \ge 1$, the rest of the proof focuses on the establishing the same property for iii). Since $Q_{\alpha}$ has $L_{Q_{\alpha}}^{g}$-Lipschitz continuous gradient, by Assumption \ref{ass: Local convexity} and by \cite[Theorem 2.1.12]{nesterov2003introductory}, for any $\hat{\mathbf{x}}^{\star} \in \mathcal{X}_{Q_{\alpha}}^{\star}$ and $\hat{\mathbf{x}} \in (\mathbb{R}^{n})^{m}$ satisfying $\textup{dist} (\hat{\mathbf{x}}, \mathcal{X}_{Q_{\alpha}}^{\star}) \le \delta$,
    \begin{multline*}
        \nabla Q_{\alpha}(\hat{\mathbf{x}})^{\top} (\hat{\mathbf{x}} - \hat{\mathbf{x}}^{\star}) \ge \frac{L_{Q_{\alpha}}^{g} \mu}{L_{Q_{\alpha}}^{g} + \mu} \|\hat{\mathbf{x}} - \hat{\mathbf{x}}^{\star}\|^{2}\\
         + \frac{1}{L_{Q_{\alpha}}^{g} + \mu}\|\nabla Q_{\alpha}(\hat{\mathbf{x}}) - \nabla Q_{\alpha}(\hat{\mathbf{x}}^{\star})\|^{2}.
    \end{multline*}
    Let $\{\hat{\mathbf{x}}^{t}\}$ satisfy \eqref{eq: NDGD}. Let filtration $\mathcal{F}^{t} = \varsigma\{\mathbf{n}^{0},\dots,\mathbf{n}^{t-1}\}$, where $\varsigma\{\cdot\}$ denotes the sigma field. Given $\hat{\mathbf{x}}^{\star} \in \mathcal{X}_{Q_{\alpha}}^{\star}$, let event $\mathcal{E}^{t} \coloneqq \{\forall \tau \le t,~ \|\hat{\mathbf{x}}^{\tau} - \hat{\mathbf{x}}^{\star}\| \le 2\frac{\epsilon_{g}}{\mu}\}$ and random variable
    \begin{gather*}
        \mathbbm{1}_{\mathcal{E}^{t}} \coloneqq
        \begin{cases}
            1,&\forall\tau \le t,~ \|\hat{\mathbf{x}}^{\tau} - \hat{\mathbf{x}}^{\star}\| \le 2\frac{\epsilon_{g}}{\mu},\\
            0,&\exists \tau \le t,~ \|\hat{\mathbf{x}}^{\tau} - \hat{\mathbf{x}}^{\star}\| > 2\frac{\epsilon_{g}}{\mu},
        \end{cases}
    \end{gather*}
    Note $\mathcal{E}^{t+1} \subseteq \mathcal{E}^{t}$ and almost surely $\mathbbm{1}_{\mathcal{E}^{t+1}} \le \mathbbm{1}_{\mathcal{E}^{t}}$. It holds that
    \begin{align*}
        &\mathbb{E} \left[\|\hat{\mathbf{x}}^{t+1} - \hat{\mathbf{x}}^{\star}\|^{2} \cdot \mathbbm{1}_{\mathcal{E}^{t}} \mid \mathcal{F}^{t} \right]\\
        &\quad=\mathbb{E} \left[\|\hat{\mathbf{x}}^{t} - \alpha(\nabla Q_{\alpha}(\hat{\mathbf{x}}^{t}) + \mathbf{n}^{t}) - \hat{\mathbf{x}}^{\star}\|^{2} \mid \mathcal{F}^{t} \right] \cdot \mathbbm{1}_{\mathcal{E}^{t}}\\
        &\quad\le(\|(\hat{\mathbf{x}}^{t} - \hat{\mathbf{x}}^{\star})\|^{2} + \alpha^{2} \|\nabla Q_{\alpha}(\hat{\mathbf{x}}^{t}) - \nabla Q_{\alpha}(\hat{\mathbf{x}}^{\star})\|^{2}\\
        &\quad\quad- 2\alpha\nabla Q_{\alpha}(\hat{\mathbf{x}}^{t})^{\top} (\hat{\mathbf{x}}^{t} - \hat{\mathbf{x}}^{\star}) + mn \alpha^{2} \sigma^{2}) \cdot \mathbbm{1}_{\mathcal{E}^{t}}\\
        &\quad\le (\|(\hat{\mathbf{x}}^{t} - \hat{\mathbf{x}}^{\star})\|^{2} + \alpha^{2} \|\nabla Q_{\alpha}(\hat{\mathbf{x}}^{t}) - \nabla Q_{\alpha}(\hat{\mathbf{x}}^{\star})\|^{2}\\
        &\quad\quad- 2\alpha(\frac{1}{L_{Q_{\alpha}}^{g} + \mu}\|\nabla Q_{\alpha}(\hat{\mathbf{x}}^{t}) - \nabla Q_{\alpha}(\hat{\mathbf{x}}^{\star})\|^{2}\\ 
        &\quad\quad+ \frac{L_{Q_{\alpha}}^{g} \mu}{L_{Q_{\alpha}}^{g} + \mu}\|\hat{\mathbf{x}}^{t} - \hat{\mathbf{x}}^{\star}\|^{2}) + mn \alpha^{2} \sigma^{2}) \cdot \mathbbm{1}_{\mathcal{E}^{t}}\\
        &\quad= ((1 - 2\alpha\frac{L_{Q_{\alpha}}^{g} \mu}{L_{Q_{\alpha}}^{g} + \mu}) \|\hat{\mathbf{x}}^{t} - \hat{\mathbf{x}}^{\star}\|^{2} + (\alpha^{2} - 2\alpha\frac{1}{L_{Q_{\alpha}}^{g} + \mu})\\
        &\quad\quad\|\nabla Q_{\alpha}(\hat{\mathbf{x}}^{t}) - \nabla Q_{\alpha}(\hat{\mathbf{x}}^{\star})\|^{2} + mn \alpha^{2} \sigma^{2}) \cdot \mathbbm{1}_{\mathcal{E}^{t}}.
    \end{align*}
    Since $0 < \alpha \le 1 / L_{Q_{\alpha}}^{g}$, note that $\alpha - 2 / (L_{Q_{\alpha}}^{g} + \mu) < 0$ by $\mu \le L_{Q_{\alpha}}^{g}$. Then,
    \begin{align*}
        &\mathbb{E} \left[\|\hat{\mathbf{x}}^{t} - \hat{\mathbf{x}}^{\star}\|^{2} \cdot \mathbbm{1}_{\mathcal{E}^{t-1}} \mid\mathcal{F}^{t-1} \right]\\
        &\quad\le ((1-\frac{2\alpha \cdot L_{Q_{\alpha}}^{g} \mu}{L_{Q_{\alpha}}^{g} + \mu} + \alpha^{2}\mu^{2} - \frac{2\alpha \cdot \mu^{2}}{L_{Q_{\alpha}}^{g} + \mu}) \|\hat{\mathbf{x}}^{t-1} - \hat{\mathbf{x}}^{\star}\|^{2}\\
        &\quad\quad+ mn\alpha^{2} \sigma^{2}) \cdot \mathbbm{1}_{\mathcal{E}^{t-1}}\\
        &\quad\le((1-\alpha\mu)^{2} \|\hat{\mathbf{x}}^{t-1} - \hat{\mathbf{x}}^{\star}\|^{2} + mn \alpha^{2} \sigma^{2}) \cdot \mathbbm{1}_{\mathcal{E}^{t-1}}\\
        &\quad\le((1-\alpha\mu)\|\hat{\mathbf{x}}^{t-1} - \hat{\mathbf{x}}^{\star}\|^{2} + mn\alpha^{2} \sigma^{2}) \cdot \mathbbm{1}_{\mathcal{E}^{t-1}}.
    \end{align*}
    Thus, choosing $0 < \alpha \le 1 / L_{Q_{\alpha}}^{g} \le 1 / L_{F}^{g} \le \mu^{-1}$ yields
    \begin{multline*}
        \mathbb{E} \left[\|\hat{\mathbf{x}}^{t} - \hat{\mathbf{x}}^{\star}\|^{2} \cdot \mathbbm{1}_{\mathcal{E}^{t-1}} \mid\mathcal{F}^{t-1} \right] - \frac{mn\alpha \sigma^{2}}{\mu} \cdot \mathbbm{1}_{\mathcal{E}^{t-1}} \\
        \le (1-\alpha \mu) \left(\|\hat{\mathbf{x}}^{t-1} - \hat{\mathbf{x}}^{\star}\|^{2} - \frac{mn \alpha \sigma^{2}}{\mu} \right) \cdot \mathbbm{1}_{\mathcal{E}^{t-1}}.
    \end{multline*}
    Define random variable $\theta^{t} \coloneqq \max\{(1-\alpha \mu)^{-t} (\|\hat{\mathbf{x}}^{t} - \hat{\mathbf{x}}^{\star}\|^{2} - \frac{mn \alpha \sigma^{2}}{\mu}),~0\}$, and
    \begin{gather*}
        \theta^{t} \mathbbm{1}_{\mathcal{E}^{t-1}} =
        \begin{cases}
            \theta^{t},&\forall\tau \le t-1,~ \|\hat{\mathbf{x}}^{\tau} - \hat{\mathbf{x}}^{\star}\| \le 2\frac{\epsilon_{g}}{\mu},\\
            0,&\exists \tau \le t-1,~ \|\hat{\mathbf{x}}^{\tau} - \hat{\mathbf{x}}^{\star}\| > 2\frac{\epsilon_{g}}{\mu},
        \end{cases}
    \end{gather*}
    and it follows that
    \begin{gather}
    \label{eq: Admissible condition 1}
        \mathbb{E} \left[\theta^{t} \mathbbm{1}_{\mathcal{E}^{t-1}}\mid\mathcal{F}^{t-1} \right] \le \theta^{t-1} \mathbbm{1}_{\mathcal{E}^{t-1}} \le\theta^{t-1} \mathbbm{1}_{\mathcal{E}^{t-2}},
    \end{gather}
    which means $\{\theta^{0},~\theta^{1} \mathbbm{1}_{\mathcal{E}^{0}},~\dots,~\theta^{t} \mathbbm{1}_{\mathcal{E}^{t-1}}\}$ is a supermartingale. It also holds that
    \begin{align*}
        &\operatorname{Var} \left[\theta^{t} \mathbbm{1}_{\mathcal{E}^{t-1}} \mid\mathcal{F}^{t-1} \right]\\
        &\quad= (1-\alpha \mu)^{-2t}  \mathbbm{1}_{\mathcal{E}^{t-1}} \operatorname{Var} \left[\|\hat{\mathbf{x}}^{t} - \hat{\mathbf{x}}^{\star}\|^{2}\mid \mathcal{F}^{t-1} \right]\\
        &\quad\le (1-\alpha \mu)^{-2t}  \mathbbm{1}_{\mathcal{E}^{t-1}} \operatorname{Var} \Big[\|\hat{\mathbf{x}}^{t-1} - \alpha(\nabla Q_{\alpha} (\hat{\mathbf{x}}^{t-1})+\mathbf{n}^{t-1})\\
        &\quad\quad- \hat{\mathbf{x}}^{\star}\|^{2}\mid\mathcal{F}^{t-1} \Big]\\
        &\quad\le (1-\alpha \mu)^{-2t} \mathbbm{1}_{\mathcal{E}^{t-1}}\Big(4\alpha^{2}\sigma^{2}\|\hat{\mathbf{x}}^{t-1} - \alpha(\nabla Q_{\alpha} (\hat{\mathbf{x}}^{t-1}))\\
        &\quad\quad- \hat{\mathbf{x}}^{\star}\|^{2} + \alpha^{4}\operatorname{Var} \left[\|\mathbf{n}^{t-1}\|^{2}\mid\mathcal{F}^{t-1} \right]\Big).
    \end{align*}
    This leads to
    \begin{multline*}
        \operatorname{Var} \left[\theta^{t} \mathbbm{1}_{\mathcal{E}^{t-1}} \mid\mathcal{F}^{t-1} \right] \le (1-\alpha \mu)^{-2t}\\
        \left(4\alpha^{2}\sigma^{2}(1+\alpha L_{Q_{\alpha}}^{g})^{2} (2\frac{\epsilon_{g}}{\mu})^{2} + 2\alpha^{4}mn\sigma^{4}\right),
    \end{multline*}
    which means for $d$, $\sigma$, and $\epsilon_{g}$ defined in \eqref{eq: parameters},
    \begin{multline}
    \label{eq: Admissible condition 2}
         \operatorname{Var} \Big[\theta^{t} \mathbbm{1}_{\mathcal{E}^{t-1}} \mid\mathcal{F}^{t-1} \Big]\\
         \le 6400 \cdot \frac{(2-\lambda_{\min}(\mathbf{W}))^{2}}{mn(L_{F}^{H})^{2}\mu^{2}} \cdot (1-\alpha \mu)^{-2t} \cdot \alpha^{4} \rho^{-6}.
    \end{multline}
    It holds almost surely that
    \begin{align*}
        &\theta^{t} \mathbbm{1}_{\mathcal{E}^{t-1}} - \mathbb{E}[\theta^{t} \mathbbm{1}_{\mathcal{E}^{t-1}}\mid\mathcal{F}^{t-1}]\\
        &\quad = (1-\alpha\mu)^{-t} (\|\hat{\mathbf{x}}^{t} - \hat{\mathbf{x}}^{\star}\|^{2} - \mathbb{E}[\|\hat{\mathbf{x}}^{t} - \hat{\mathbf{x}}^{\star}\|^{2}\mid\mathcal{F}^{t-1}]) \mathbbm{1}_{\mathcal{E}^{t-1}}\\
        &\quad = (1-\alpha\mu)^{-t} (\|\hat{\mathbf{x}}^{t-1} - \alpha(\nabla Q_{\alpha} (\hat{\mathbf{x}}^{t-1})+\mathbf{n}^{t-1}) - \hat{\mathbf{x}}^{\star}\|^{2}\\
        &\quad\quad- \mathbb{E} \Big[\|\hat{\mathbf{x}}^{t-1} - \alpha(\nabla Q_{\alpha} (\hat{\mathbf{x}}^{t-1})+\mathbf{n}^{t-1})\\
        &\quad\quad- \hat{\mathbf{x}}^{\star}\|^{2}\mid\mathcal{F}^{t-1} \Big]) \mathbbm{1}_{\mathcal{E}^{t-1}}\\
        &\quad\le (1-\alpha\mu)^{-t} (2\|\hat{\mathbf{x}}^{t-1} - \alpha(\nabla Q_{\alpha} (\hat{\mathbf{x}}^{t-1})) - \hat{\mathbf{x}}^{\star}\| \cdot \alpha \|\mathbf{n}^{t-1}\|\\
        &\quad\quad+ \alpha^{2} \|\mathbf{n}^{t-1}\|^{2} - mn \alpha^{2} \sigma^{2}) \mathbbm{1}_{\mathcal{E}^{t-1}}\\
        &\quad\le (1-\alpha\mu)^{-t} (2\alpha(1+\alpha L_{Q_{\alpha}}^{g}) 2\frac{\epsilon_{g}}{\mu} \|\mathbf{n}^{t-1}\| + \alpha^{2} \|\mathbf{n}^{t-1}\|^{2}\\
        &\quad\quad- mn \alpha^{2} \sigma^{2}).
    \end{align*}
    By \eqref{eq: descent third term}, it follows that
    \begin{gather*}
         \mathbb{P} \left[\|\mathbf{n}^{t-1}\|^{2} \le mn\sigma^{2} + 16\sqrt{mn}\sigma^{2}\sqrt{2\rho} \right] \ge 1 - \mathrm{e}^{-\rho}.
    \end{gather*}
    Thus, 
    \begin{multline}
    \label{eq: Admissible condition 3}
        \mathbb{P} \Big[\theta^{t} \mathbbm{1}_{\mathcal{E}^{t-1}} - \mathbb{E}\Big[\theta^{t} \mathbbm{1}_{\mathcal{E}^{t-1}}\mid\mathcal{F}^{t-1}\Big] \le 640 \cdot \frac{2-\lambda_{\min}(\mathbf{W})}{(mn)^{\tfrac{1}{4}} L_{F}^{H} \mu}\\
        (1-\alpha\mu)^{-t} \cdot \alpha^{2}\rho^{-\tfrac{11}{4}}\Big]\ge 1 - \mathrm{e}^{-\rho}.
    \end{multline}
    Applying Proposition \ref{pro: Relaxed Azuma} with $X^{t} = \theta^{t} \mathbbm{1}_{\mathcal{E}^{t-1}}$, which satisfies conditions 1), 2) and 3) in view of \eqref{eq: Admissible condition 1}, \eqref{eq: Admissible condition 2}, and \eqref{eq: Admissible condition 3}, respectively, there exists constant $c > 0$ such that with $\lambda = c \cdot d_{t} \cdot \alpha^{2}\rho^{-\tfrac{9}{4}}$,
    \begin{align}
    \label{eq: Relaxed Azuma}
        \mathbb{P} \left[\theta^{t} \mathbbm{1}_{\mathcal{E}^{t-1}} - \theta^{0} \ge c \cdot d_{t} \cdot \alpha^{2}\rho^{-\tfrac{9}{4}}\right] \le (t+1)\mathrm{e}^{-\rho},
    \end{align}
    where $d_{t} = \sqrt{\sum_{\kappa=0}^{t}(1-\alpha\mu)^{-2\kappa}}$. By the definition of $\theta^{t}$, \eqref{eq: Relaxed Azuma} implies there exists constant $c > 0$ such that
    \begin{multline*}
        \mathbb{P} \Big[\bar{\mathcal{E}}^{t-1} \cup \{\|\hat{\mathbf{x}}^{t} - \hat{\mathbf{x}}^{\star}\|^{2} \le (1-\alpha\mu)^{t}(\|\hat{\mathbf{x}}^{0} - \hat{\mathbf{x}}^{\star}\|^{2} - mn\frac{\alpha\sigma^{2}}{\mu})\\
        + c \cdot (1-\alpha\mu)^{t} d_{t} \cdot \alpha^{2}\rho^{-\tfrac{9}{4}} + mn\frac{\alpha\sigma^{2}}{\mu}\} \Big] \ge 1 - (t+1)\mathrm{e}^{-\rho},
    \end{multline*}
    By $(1-\alpha\mu)^{t} d_{t} \le \sqrt{\frac{1}{2\alpha\mu - \alpha^{2}\mu^{2}}}$, there exist constant $c > 0$ and threshold $\rho_{4} \ge 1$ such that for any $\rho \ge \rho_{4}$ and $\epsilon_{g}$ per \eqref{eq: parameters}, 
    \begin{multline*}
        \mathbb{P} \Big[\bar{\mathcal{E}}^{t-1} \cup \{\|\hat{\mathbf{x}}^{t} - \hat{\mathbf{x}}^{\star}\|^{2} \le \|\hat{\mathbf{x}}^{0} - \hat{\mathbf{x}}^{\star}\|^{2} + c \cdot \mu^{-\tfrac{1}{2}} \alpha^{\tfrac{3}{2}}\rho^{-\tfrac{9}{4}}\\
        + mn\frac{\alpha\sigma^{2}}{\mu} \le \frac{4}{{\mu}^{2}}\alpha\} \Big] \ge 1 - (t+1)\mathrm{e}^{-\rho}.
    \end{multline*}
    By Assumption \ref{ass: Local convexity}, since $\textup{dist} (\hat{\mathbf{x}}^{0}, \mathcal{X}_{Q_{\alpha}}^{\star}) \le \sqrt{\alpha}/\mu$, it follows that
    \begin{multline*}
        \mathbb{P} [\bar{\mathcal{E}}^{t}] = \mathbb{P} \Big[\mathcal{E}^{t-1} \cap \{\|\hat{\mathbf{x}}^{t} - \hat{\mathbf{x}}^{\star}\| > \frac{2}{\mu}\sqrt{\alpha}\}\Big] + \mathbb{P} [\bar{\mathcal{E}}^{t-1}]\\
        \le (t+1)\mathrm{e}^{-\rho} + \mathbb{P} [\bar{\mathcal{E}}^{t-1}] \le t(t+1)\mathrm{e}^{-\rho} + \mathbb{P} [\bar{\mathcal{E}}^{0}]\\
        = t(t+1)\mathrm{e}^{-\rho},
    \end{multline*}
    which yields that for any $\rho \ge \rho_{4}$,
    \begin{multline}
    \label{eq: Stay close}
        \mathbb{P} \left[\forall \tau \in (0,K],~\|\hat{\mathbf{x}}^{\tau} - \hat{\mathbf{x}}^{\star}\| \le \frac{2}{\mu}\sqrt{\alpha} \right]\\
        \ge 1 - K(K+1)\mathrm{e}^{-\rho}.
    \end{multline}
\end{proofs}

\section{Numerical Examples}
\label{sec: Numerical Examples}
In this section, through two illustrative numerical examples, we demonstrate how random perturbations can help distributed gradient descent escape saddle points and converge to the vicinity of local minimizers.

\subsection{Non-convex Quartic Optimization Example}

We first consider a non-convex quartic optimization problem of a form that arises in applications including neural networks, optimal control, and signal processing. As an instance of~\eqref{eq: Unconstrained optimization problem}, consider
\begin{gather*}
    \min_{\bm{\theta}=[\theta_{1},\theta_{2}]^{\top}\in\mathbb{R}^{2}}
    J(\bm{\theta})
    =
    \sum_{i=1}^{m}J_i(\bm{\theta}),
\end{gather*}
where
\begin{gather*}
    J_i(\bm{\theta})
    =
    \frac{1}{m}
    \left(
    a_i\theta_{1}^{4}
    +b_i\theta_{1}^{2}
    +c_i\theta_{2}^{4}
    +d_i\theta_{2}^{2}
    \right),
\end{gather*}
and
\begin{gather*}
    \sum_{i=1}^{m}a_i>0,\qquad
    \sum_{i=1}^{m}b_i>0,\qquad
    \sum_{i=1}^{m}c_i>0,\qquad
    \sum_{i=1}^{m}d_i<0,
\end{gather*}
which ensures the existence of a saddle point and two local minimizers. We use $m=20$ agents connected by a degree-$4$ regular graph. The gradient and Hessian of each local objective are
\begin{gather*}
    \nabla J_i(\bm{\theta})
    =
    \frac{1}{m}
    \begin{bmatrix}
        4a_i\theta_{1}^{3}+2b_i\theta_{1}\\
        4c_i\theta_{2}^{3}+2d_i\theta_{2}
    \end{bmatrix},
\end{gather*}
and
\begin{gather*}
    \nabla^{2}J_i(\bm{\theta})
    =
    \frac{1}{m}
    \begin{bmatrix}
        12a_i\theta_{1}^{2}+2b_i & 0\\
        0 & 12c_i\theta_{2}^{2}+2d_i
    \end{bmatrix}.
\end{gather*}

The contour in Fig.~\ref{fig: Contour plot 2} shows the saddle point, marked by $\times$, at $(0,0)$ and the two local minimizers, marked by $\star$, symmetrically located on the positive and negative $\theta_{2}$-axis. The parameters $a_i,b_i,c_i,d_i$ are randomly selected subject to the conditions above. In Fig.~\ref{fig: Contour plot 2}, all agents are initialized at $(1,0)$ on the stable manifold of the saddle point. While \textbf{DGD} approaches the saddle, \textbf{NDGD} escapes from it and moves toward a local minimizer. Fig.~\ref{fig: Iteration plot 2} considers an initialization close to the stable manifold at $(1-10^{-5},10^{-5})$. In this case, \textbf{DGD} requires approximately $500$ iterations to depart from the saddle-point vicinity and approach a local minimizer, whereas \textbf{NDGD} escapes in about $100$ iterations and also approaches a local minimizer. These observations illustrate the effectiveness of random perturbations in accelerating saddle-point escape relative to standard \textbf{DGD}.

\begin{figure*}
    \centering
    \subfloat[Trajectory of $(\theta_{1,1}^{k},\theta_{2,1}^{k})$ at agent $1$.]{
        \includegraphics[width=0.45\textwidth]{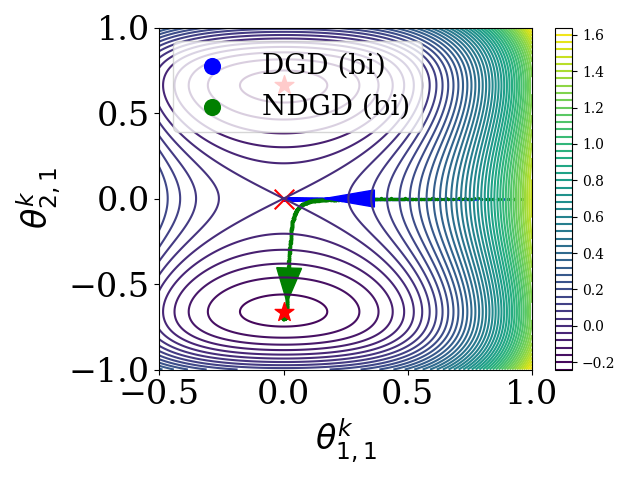}
        \label{fig: Contour plot 2}
    }
    \hfill
    \subfloat[Distance from the agent iterates to $\Theta^{\star}$.]{
        \includegraphics[width=0.45\textwidth]{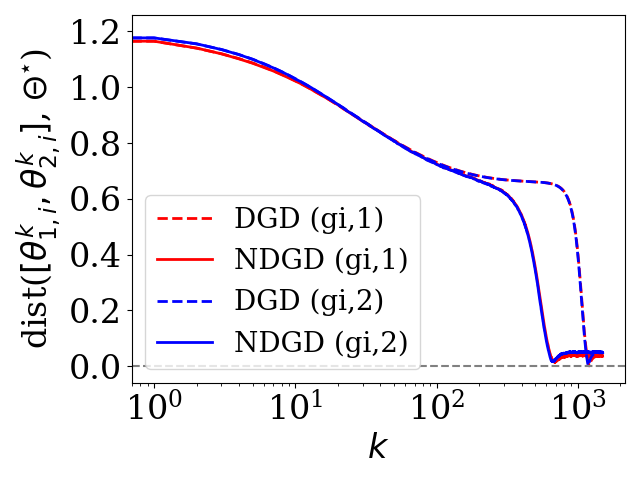}
        \label{fig: Iteration plot 2}
    }
    \caption{Second-order properties of \textbf{NDGD} and \textbf{DGD} for the non-convex quartic optimization example, where $\Theta^{\star}$ denotes the set of local minimizers of $J$, marked by red $\star$. ``bi'' and ``gi'' refer to bad and good initializations, respectively; ``1'' and ``2'' refer to agents $1$ and $2$, respectively.}
    \label{fig: Non-convex quartic optimization}
\end{figure*}

\subsection{Binary Classification}

We next consider training a basic neural network for binary classification as a non-convex distributed optimization example~\cite{daneshmand2020second,vlaski2021distributed2}. Given a dataset $(X,Y)$, where
\begin{gather*}
    X=\{\mathbf{x}^{(1)},\mathbf{x}^{(2)},\ldots,\mathbf{x}^{(m)}\}
\end{gather*}
contains the input features and
\begin{gather*}
    Y=\{y^{(1)},y^{(2)},\ldots,y^{(m)}\}
\end{gather*}
contains the corresponding binary labels, with $\mathbf{x}^{(i)}\in\mathbb{R}^{n}$ and $y^{(i)}\in\{0,1\}$, the neural network learns weights $\mathbf{v}_{1}\in\mathbb{R}^{d}$ and $\mathbf{V}_{2}\in\mathbb{R}^{d\times n}$. The predicted probability for the $i$-th sample is
\begin{gather*}
    \hat{y}^{(i)}
    =
    \frac{1}
    {1+\exp(-\mathbf{v}_{1}^{\top}\mathbf{V}_{2}\mathbf{x}^{(i)})}.
\end{gather*}

Let
\begin{gather*}
    \tilde{y}^{(i)}
    =
    \begin{cases}
        -1, & y^{(i)}=0,\\
        1, & y^{(i)}=1.
    \end{cases}
\end{gather*}
Using the corresponding logistic representation with quadratic regularization, the distributed learning problem can be written as
\begin{gather*}
    \min_{\mathbf{v}_{1},\mathbf{V}_{2}}
    \tilde{L}(\mathbf{v}_{1},\mathbf{V}_{2})
    =
    \sum_{i=1}^{m}
    \tilde{L}_{i}(\mathbf{v}_{1},\mathbf{V}_{2}),
\end{gather*}
where
\begin{multline*}
    \tilde{L}_{i}(\mathbf{v}_{1},\mathbf{V}_{2})
    =
    \frac{1}{m}
    \ln\left(
    1+
    \exp\left(
    -\tilde{y}^{(i)}
    \mathbf{v}_{1}^{\top}
    \mathbf{V}_{2}
    \mathbf{x}^{(i)}
    \right)
    \right)\\
    +
    \frac{\beta}{2m}
    \left(
    \|\mathbf{v}_{1}\|^{2}
    +
    \|\mathbf{V}_{2}\|_{F}^{2}
    \right).
\end{multline*}

For direct visualization of the saddle geometry and the algorithm trajectories, we set $n=d=1$ and use $m=5$ agents connected by a degree-$2$ regular graph. Each agent $i$ has access only to its local objective $\tilde{L}_{i}$.

For $n=d=1$, the gradient of the local objective is
\begin{gather*}
    \nabla\tilde{L}_{i}(v_{1},V_{2})
    =
    \frac{1}{m}
    \begin{bmatrix}
        \displaystyle
        \beta v_{1}
        -
        \frac{\tilde{y}^{(i)}V_{2}x^{(i)}}
        {1+\exp(\tilde{y}^{(i)}v_{1}V_{2}x^{(i)})}
        \\[3mm]
        \displaystyle
        \beta V_{2}
        -
        \frac{\tilde{y}^{(i)}v_{1}x^{(i)}}
        {1+\exp(\tilde{y}^{(i)}v_{1}V_{2}x^{(i)})}
    \end{bmatrix}.
\end{gather*}
At $(v_{1},V_{2})=(0,0)$, the Hessian is
\begin{gather*}
    \nabla^{2}\tilde{L}_{i}(0,0)
    =
    \frac{1}{m}
    \begin{bmatrix}
        \beta & -\dfrac{1}{2}\tilde{y}^{(i)}x^{(i)}\\[2mm]
        -\dfrac{1}{2}\tilde{y}^{(i)}x^{(i)} & \beta
    \end{bmatrix}.
\end{gather*}

For the generated instance, $\tilde{L}(v_{1},V_{2})$ has three critical points: two local minimizers and one saddle point. As illustrated in Fig.~\ref{fig: Contour plot 1}, the saddle point, marked by $\times$, is located at $(0,0)$, while the two local minimizers, marked by $\star$, lie symmetrically in the positive and negative quadrants. We set $\beta=0.1$, draw $\tilde{y}^{(i)}$ uniformly from $\{-1,1\}$, and generate $x^{(i)}\sim\mathcal{N}(\tilde{y}^{(i)},1)$.

In Fig.~\ref{fig: Contour plot 1}, all agents are initialized at $(-1,1)$ on the stable manifold of the saddle point. For agent $1$, \textbf{DGD} approaches the saddle point, whereas \textbf{NDGD} successfully escapes and moves toward a local minimizer. Fig.~\ref{fig: Iteration plot 1} considers all agents initialized at $(1-10^{-5},-1-10^{-5})$, close to the stable manifold of the saddle point. In this case, \textbf{DGD} takes about $600$ iterations to move away from the saddle-point vicinity and approach a local minimizer, whereas \textbf{NDGD} escapes in about $100$ iterations and also converges to the vicinity of a local minimizer. This example further illustrates the effectiveness of \textbf{NDGD} in escaping saddle points compared with standard \textbf{DGD}.

\begin{figure*}
    \centering
    \subfloat[Trajectory of $(v_{1,1}^{k},V_{2,1}^{k})$ at agent $1$.]{
        \includegraphics[width=0.45\textwidth]{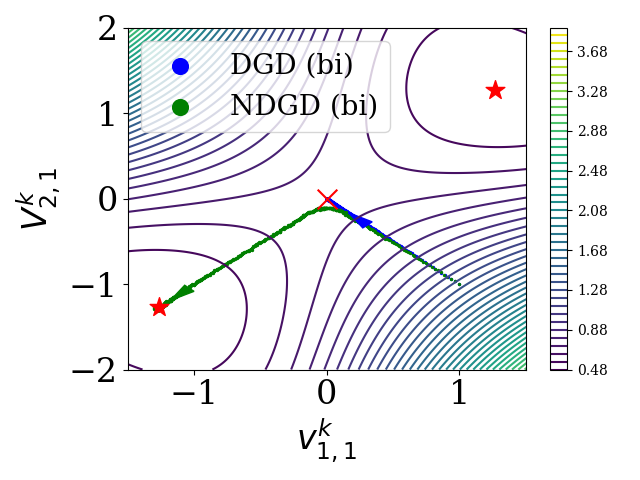}
        \label{fig: Contour plot 1}
    }
    \hfill
    \subfloat[Distance from the agent iterates to $\mathcal{V}^{\star}$.]{
        \includegraphics[width=0.45\textwidth]{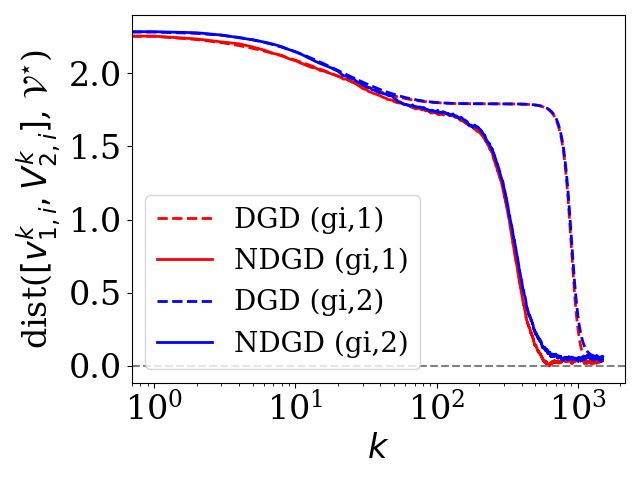}
        \label{fig: Iteration plot 1}
    }
    \caption{Second-order properties of \textbf{NDGD} and \textbf{DGD} for the binary classification example, where $\mathcal{V}^{\star}$ denotes the set of local minimizers of $\tilde{L}$, marked by red $\star$. ``bi'' and ``gi'' refer to bad and good initializations, respectively; ``1'' and ``2'' refer to agents $1$ and $2$, respectively.}
    \label{fig: Binary classification example}
\end{figure*}

\section{Conclusions and Discussion}
\label{sec: Conclusions and Discussion}
This paper established finite-time, high-probability second-order guarantees for \textbf{NDGD} in distributed non-convex optimization. Under additional local regularity, all agents approach the local minimizer set of the original objective with explicit optimality and consensus bounds. The analysis characterizes the effects of step-size, noise, and network structure, including the trade-off among confidence, optimization accuracy, consensus accuracy, and iteration complexity. Numerical results demonstrate improved saddle-point escape over standard \textbf{DGD}. Future work will consider weaker geometric assumptions and more general distributed settings.

\section*{References}

\bibliographystyle{unsrt}  
\bibliography{Reference} 

\end{document}